\documentclass[12pt, reqno,oneside]{amsart}
\usepackage{amsmath,amssymb,amsfonts,amscd,latexsym,amsthm,mathrsfs,verbatim,tikz-cd, hyperref}
\pdfoutput=1
\usepackage{graphicx}
\usepackage[backend=biber, style=trad-alpha,sorting=nyt]{biblatex}
\usepackage{fullpage}
\renewcommand{\P}{\mathbb{P}}
\newcommand{\Z}{\mathbb{Z}}
\renewcommand{\div}{\operatorname{div}}

\DeclareMathOperator{\Hom}{Hom}
\DeclareMathOperator{\im}{im}
\DeclareMathOperator{\spec}{Spec}

\DeclareMathOperator{\gal}{Gal}

\DeclareMathOperator{\Bl}{Bl}

\DeclareMathOperator{\mult}{mult}
\DeclareMathOperator{\id}{id}

\newcommand{\CHAR}{\mathrm{char}}
\renewcommand\d{\,\mathrm d}

\DeclareMathOperator{\Pic}{Pic}
\DeclareMathOperator{\Br}{Br}

\newtheorem{theorem}{Theorem}
\newtheorem{proposition}[theorem]{Proposition}
\newtheorem{lemma}[theorem]{Lemma}
\newtheorem{corollary}[theorem]{Corollary}

\theoremstyle{definition}
\newtheorem{definition}[theorem]{Definition}
\newtheorem{example}[theorem]{Example}
\newtheorem{remark}[theorem]{Remark}
\newtheorem{question}[theorem]{Question}
\newtheorem{conjecture}[theorem]{Conjecture}

\numberwithin{theorem}{section}
\numberwithin{equation}{section}

\newcommand{\lineN}{\overline{\mathbb{N}}}
\newcommand{\cA}{\mathcal{A}}
\newcommand{\cB}{\mathcal{B}}
\newcommand{\cC}{\mathcal{C}}
\newcommand{\cI}{\mathcal{I}}
\newcommand{\cO}{\mathcal{O}}
\newcommand{\cD}{\mathcal{D}}
\newcommand{\cP}{\mathcal{P}}

\newcommand{\cU}{\mathcal{U}}
\newcommand{\cV}{\mathcal{V}}

\newcommand{\cM}{\mathcal{M}}
\newcommand{\fM}{\mathfrak{M}}

\newcommand{\cX}{\mathcal{X}}
\newcommand{\cY}{\mathcal{Y}}
\newcommand{\cZ}{\mathcal{Z}}
\newcommand{\cG}{\mathcal{G}}
\newcommand{\cHom}{\mathcal{H}\! \mathit{om}}
\newcommand{\fp}{\mathfrak{p}}
\newcommand{\fq}{\mathfrak{q}}

\newcommand{\fin}{\textnormal{fin}}
\newcommand{\red}{\textnormal{red}}
\DeclareMathOperator{\Cl}{Cl}

\newcommand{\lParen}{(\!(}
\newcommand{\rParen}{)\!)}
\newcommand{\lBracket}{[\![}
\newcommand{\rBracket}{]\!]}

\begin{document}
	\title{Generalized Campana points and adelic approximation on toric
varieties}
	\author{Boaz Moerman}
	

        \begin{abstract}
            We introduce a general framework for studying special subsets of rational points on an algebraic variety, termed $\cM$-points. The notion of $\cM$-points generalizes the concepts of integral points, Campana points and Darmon points. We introduce and study $M$-approximation over number fields and function fields, which is a notion that generalizes weak and strong approximation. We show that this property implies that the set of $\cM$-points is not thin. We then give a simple characterisation of when a split toric variety satisfies $M$-approximation, generalizing work of Nakahara and Streeter. Further, we determine when the set of $\cM$-points on a split toric variety is thin.
        \end{abstract}
	\maketitle
        \setcounter{tocdepth}{1}
        \tableofcontents
        
\section{Introduction}
In recent years, the theory of Campana points has attracted much attention. 
Campana points are rational points which are integral with respect to weighted divisors, and in particular include the integral points which do not meet the divisors. Campana points interpolate between rational points and integral points.
Arithmetically, on a projective variety over $\mathbb{Q}$ Campana points give a way to study the rational points whose coordinates are $m$-full integers. Recall that an integer $n$ is $m$-full if for every prime number $p$ dividing $n$, $p^m$ also divides $n$. Geometrically, Campana points correspond to curve embeddings sufficiently tangent to given divisors.

Many results and conjectures concerning rational points extend to Campana points. One such example is the Mordell Conjecture \cite[Part E]{HiSi00}, whose analogue for Campana points has been introduced by Campana \cite{Cam05} 
and proven over function fields of characteristic $0$ by him \cite{Cam05} and recently in \cite{KPS22} in arbitrary characteristic. Over number fields the conjecture follows from the $abc$ Conjecture, see \cite[Appendix]{Sme17} for example. Recently in \cite{BaJa23} it was shown that the Kobayashi--Ochiai Theorem generalizes to Campana pairs, giving a higher dimensional analog of the Mordell Conjecture for the theory of Campana points over function fields.

Similarly, recently a generalization of Manin's conjecture on rational points of bounded height to the setting of Campana points has been introduced in \cite{PSTV20}. This conjecture is known for several varieties with appropriately chosen divisors, such as diagonal hypersurfaces \cite{Val12, BrYa21,BBKOPW23,Shu21,Shu22}, vector group compactifications \cite{PSTV20}, norm forms \cite{Str21}, complete split toric varieties \cite{PiSc23} and certain complete intersections therein \cite{PiSc24}, biequivariant compactifications of the Heisenberg group \cite{Xia21} and wonderful compactifications \cite{CLTT24}. In \cite{Fai23}, a motivic analogue of Manin's conjecture for Campana points is proven for vector group compactifications.

A Darmon point is a special kind of Campana point. On a projective variety  over $\mathbb{Q}$ such points correspond to rational points whose coordinates are, up to units, $m$-th powers of integers. These were dubbed Darmon points in \cite{MiNaSt22} in honour of Darmon's study of \emph{M-curves} \cite{Dar97}. In that paper, Darmon unconditionally proves the analog of the Mordell Conjecture for Darmon points over number fields and uses it to prove that generalized Fermat equations have finitely many solutions. Darmon points also appear in Campana's work as \emph{morphismes orbifoldes divisibles} \cite[D\'efinition 2.4]{Cam10}. Darmon points are intrinsically connected to obifolds through the root stack construction given in \cite{Cad07}, as we will explain in Section \ref{section: root stacks}. 

In several papers \cite{AbAl18,Str21} another variant of Campana points called weak Campana points is used in the study of analogues of the conjectures of Vojta \cite{Voj87} and Manin \cite{Pey95} on rational points.

\subsection{\texorpdfstring{$\cM$}{M}-points and \texorpdfstring{$M$}{M}-approximation}
This paper introduces a vast generalisation of the notions of integral points, Darmon points, Campana points and weak Campana points, which we call \textit{$\cM$-points}. To define all these points, one needs to fix an integral model of the variety. Here the parameter set $\cM$ encodes the boundary components on the integral model and the admissible intersection multiplicities, and the letter $\cM$ was chosen in reference to the latter. The precise definition is given in Section \ref{subsection: pairs}. For example, in $\mathbb{A}^n_{\mathbb{Z}}$, for suitable choices of the parameter set $\cM$, $\cM$-points can describe points with integer coordinates that are squarefree, all cubes, or all coprime. On the geometric side, $\cM$-points can describe tuples of polynomials that are squarefree, have no simple zeroes, or are coprime.
Many more examples are given in Section \ref{subsection: M-points examples}.

Like integrality, the $\cM$-condition is a local condition, so local-global principles, such as strong approximation as defined in \cite[\S 2.6.4.5]{Poo17}, have natural analogues for $\cM$-points.
In Section \ref{section: M-approximation}, we introduce $M$-approximation and integral $\cM$-approximation, which generalize and interpolate between weak approximation and (integral) strong approximation. Here $M$ encodes the boundary components on the variety and the same intersection multiplicities as $\cM$. The precise definition is given in Section \ref{subsection: pairs}.
A variety $X$ satisfies \textit{$M$-approximation} if, roughly speaking, for every finite set of places $S\subset \Omega_K$ and points $P_v\in X(K_v)$ for $v\in S$, there exists $P\in X(K)$ approximating all $P_v$, such that for every place $v\in \Omega_K\setminus S$, $P$ satisfies the $\cM$-condition at the place $v$. This is formalized in terms of density of rational points in an appropriate adelic space, see Definition \ref{def: adelic M-points}. Similarly to strong approximation, this property is independent of the choice of the integral model.

\textit{Integral $\cM$-approximation} is a variant on this notion obtained by letting $P_v\in X(K_v)$ be a local $\cM$-point and requiring $P\in X(K)$ to be an $\cM$-point. Under very mild conditions, $M$-approximation implies integral $\cM$-approximation, see Section \ref{section: M-approximation} and Proposition \ref{prop: integral M-approximation on toric}.

Integral $\cM$-approximation extends the notion of weak Campana approximation which was introduced by Nakahara and Streeter in \cite{NaSt20} and further studied in \cite{NaSt20, CLT24}. However, $M$-approximation behaves very differently from the notion of Campana strong approximation as given in their follow up paper with Mitankin \cite{MiNaSt22}. Their notion interpolates between strong approximation and integral strong approximation, rather than weak approximation.

\vspace{20 px}
The fields considered in this paper are PF fields $(K,C)$. Here $K$ is either a number field or the function field of a regular projective curve $C$ over some field $k$. If $K$ is a number field then $C=\spec \cO_K$ where $\cO_K$ is the ring of integers. Such fields have a good notion of places, as explained in Section \ref{section: PF fields}, and therefore allow for the study of local-global principles, which has been done previously in \cite{Yam96,Yam02}. The terminology `PF fields' is explained by the fact that such fields satisfy a product formula, see Remark \ref{remark: product formula}.

\subsection{$M$-approximation and the \texorpdfstring{$\cM$}{M}-Hilbert property}
If $X$ is an integral variety over a field $K$, then $A\subset X(K)$ is thin if it is contained in a finite union of proper closed subvarieties and images of the set of rational points under generically finite maps $Y\rightarrow X$ of integral varieties of degree greater than $1$. In positive characteristic this is less strict than the usual definition as given in \cite{BFP13, Lug22}, see Remark \ref{remark: different definition thin}.

Let $X$ be a variety over a PF field $(K,C)$. Let $B\subset C$ be an open subscheme. As is the case for integral points, we need to consider integral models over $B$ in order to define $\cM$-points over $B$ in $X(K)$. If $M$ is a parameter set defined by boundary components $D_\alpha$ and a set of admissible multiplicities $\fM$ as in Definition \ref{def: pair}, then we can consider an integral model $\cX$ with parameter set $\cM$ obtained by taking integral models $\cD_\alpha\subset \cX$ of the boundary components. We then call $(X,M)$ a pair and $(\cX,\cM)$ an integral model of $(X,M)$ over $B$ and write the set of $\cM$-points as $(\cX,\cM)(B)$. The boundary components $D_{\alpha}$ are allowed to be arbitrary closed subschemes, and not just divisors.

The first theorem shows that integral $\cM$-approximation implies that the set of $\cM$-points is Zariski dense, unless this set is empty. If $K$ is a global field it also shows that the set of $\cM$-points is not thin.
\begin{theorem} \label{theorem: M-approx implies not thin}
Let $(X,M)$ be a pair over a PF field $(K,C)$ with integral model $(\cX,\cM)$ over an open subscheme $B\subset C$. Assume that $X$ is a geometrically reduced variety and that $D_{\alpha}$ does not contain an irreducible component of $X$ for any $\alpha\in \cA$.

If $(\cX,\cM)$ satisfies integral $\cM$-approximation off a finite set of places $T\subset \Omega_K$ and $(\cX,\cM)(B)\neq \emptyset$, then $X$ is geometrically integral and $(\cX,\cM)(B)$ is Zariski dense. If furthermore $K$ is a global field then $(\cX,\cM)$ satisfies the $\cM$-Hilbert property over $B$, i.e., $(\cX,\cM)(B)$ is not thin in $X(K)$.
\end{theorem}
This theorem is a generalisation to $\cM$-points of a result of Nakahara and Streeter \cite[Theorem 1.1]{NaSt20} which shows that weak weak Campana approximation implies the Campana Hilbert property. Their result is in turn a generalisation of a theorem of Colliot-Th\'el\`ene and Ekedahl \cite[Theorem 3.5.7]{Ser16}, which states that weak weak approximation implies the Hilbert property. Theorem \ref{theorem: M-approx implies not thin} also generalizes all of these results from number fields to global fields, and removes the assumption that the variety is normal or even integral.

If the field $K$ is a number field and $X$ is geometrically irreducible, then the proof of Theorem \ref{theorem: M-approx implies not thin} closely follows the proof of Nakahara and Streeter. The main idea is to use the Lang--Weil bounds \cite{LaWe54} to show that, for a generically finite morphism $Y\rightarrow X$ of degree greater than $1$, the image of $Y(K_v)$ in $X(K_v)$ is too small to contain $(\cX,\cM)(\cO_v)$.
For global function fields, the proof is similar but more complicated, due to the existence of inseparable morphisms $Y\rightarrow X$. For these morphisms, we cannot apply the Lang--Weil bounds, and instead we prove and use Lemma \ref{lemma: image inseparable nowhere dense}, which implies that the image of $Y(K_v)\rightarrow X(K_v)$ is nowhere dense if $X$ is smooth.

For function fields of curves over infinite fields, the main difficulty in proving Zariski density of $(\cX,\cM)(B)$ is that we do not know whether $X(K_v)$ contains a smooth point for some place $v$. We show this by invoking recent results by Moret-Bailly \cite{MoBa20} and combining this with Hensel's Lemma.

\begin{remark}
The assumption that $X$ is geometrically reduced is necessary for the conclusion of Theorem \ref{theorem: M-approx implies not thin} to hold. Example \ref{example: M-approx not Zariski dense} gives an example of an integral curve such that $X(K_v)=X(K)$ consists of a single point for every place $v$, so that $X$ satisfies weak approximation but $X(K)$ is not Zariski dense in $X$.   
\end{remark}

For a function field $K$ of a curve over an algebraically closed field, Corollary \ref{corollary: M-approx does not imply M-Hilbert} shows integral $\cM$-approximation need not imply the $\cM$-Hilbert property. In particular, this also gives examples of varieties which satisfy strong approximation, but for which the integral points are both Zariski dense and thin.

\subsection{Results for split toric varieties}
Most of the remainder of the paper focuses on split toric varieties, and considers pairs $(X,M)$ where the boundary components used in defining $M$ are the torus-invariant prime divisors $D_1,\dots, D_n$. We will call such a pair $(X,M)$ a \emph{toric pair}. In this setting we give a necessary and sufficient criterion for $M$-approximation to hold off a given set of places $T$.
To state the theorem, we will need the set
$$\rho(K,C)=\left\{n\in \mathbb{N}^*\mid \cO_v^\times\xrightarrow{(\cdot)^n} \cO_v^\times \text{ is surjective for all } v\in \Omega_K\right\},$$
which has not been studied before, to the author's knowledge. This set is explicitly computed in Lemma \ref{lemma: computation rho}. In particular $\rho(K,C)=\{1\}$ if $K$ is a global field and $\rho(K,C)=\mathbb{N}\setminus \CHAR(K)\mathbb{N}$ if $K$ is the function field of a curve over a separably closed field.
\begin{theorem} \label{theorem: M-approx}
Let $(K,C)$ be a PF field and let $(X,M)$ be a toric pair where $X$ is a normal complete split toric variety over $K$ with co-character lattice $N$. Let $T\subset \Omega_K$ be a nonempty finite set of places and let $N_M, N_M^+\subset N$ be as in Definition \ref{def: invariants} and Definition \ref{def: N_M singular}.
Then
\begin{enumerate}
\item $(X,M)$ satisfies $M$-approximation off $T$ if $|N:N_M|\in \rho(K,C)$. If $\Pic(C)$ is finitely generated, then the converse also holds.
\item  Furthermore, $(X,M)$ satisfies $M$-approximation if and only if $N=N_M^+$.
\end{enumerate}
\end{theorem}
This considerably generalizes Nakahara's and Streeter's result \cite[Theorem 1.2(i)]{NaSt20} from projective space to general split toric varieties, from Campana points to $\cM$-points and from number fields to PF fields.

The proof of Nakahara and Streeter does not extend to the setting of $\cM$-points, as their proof essentially uses the fact that for any two $S$-integers $a,b$ such that $b$ divides $a$, the integer $a^m b$ is $m$-full for any positive integer $m$. Since we consider sets of points $(\cX,\cM)(B)$ which can greatly differ from Campana points and have much less structure, our proof of Theorem \ref{theorem: M-approx} takes a different approach. The proof is subdivided in two steps. First in Section \ref{section: squarefree strong approximation} we prove results on the density of squarefree elements in rings of integers and on affine curves, which can be thought of as ``squarefree strong approximation'' on the affine line. Then we use Cox coordinates as introduced in \cite{Cox95} to extend the results from the affine line to toric varieties.

\begin{remark}
The condition that $\Pic(C)$ is finitely generated is satisfied in many cases, such as when $C$ is rational or when $K$ is finitely generated over its prime field, by Néron's generalisation of the Mordell-Weil theorem \cite[Corollary 7.2]{Con06}.
\end{remark}
\begin{remark}
Note that $M$-approximation off a nonempty set of places only depends on $N_M$ rather than on $(X,M)$. This can be viewed as an analogue of purity of strong approximation as in \cite{CLX19, CaZh20, Wei21, Che24}. In fact, this theorem shows that purity holds for strong approximation with respect to toric subvarieties.
\end{remark}

By a classical theorem of Minchev \cite[Theorem 1]{Min89}, of which we give a new proof in Corollary \ref{cor: Minchev}, a variety over a number field can only satisfy strong approximation off a finite set of places $T$ if it is algebraically simply connected. The following consequence of Theorem \ref{theorem: M-approx} implies that, for split toric varieties, these two properties are actually equivalent.
\begin{corollary} \label{corollary: strong approximation 1}
	Let $(K,C)$ be a PF field of characteristic $0$, let $\overline{K}$ be an algebraic closure of $K$, let $X$ be a complete normal split toric variety over $K$ and let $V\subseteq X$ be an open toric subvariety. Then:
    
	\begin{enumerate}
		\item For any nonempty finite set of places $T$, $V$ satisfies strong approximation off $T$ if $\pi_1(V_{\overline{K}})$ is finite and $|\pi_1(V_{\overline{K}})|\in\rho(K,C)$. If $\Pic(C)$ is finitely generated, then the converse also holds.
		\item The variety $V$ satisfies strong approximation if and only if $V_{\overline{K}}$ is simply connected and $\cO(V_{\overline{K}})=\overline{K}$.
	\end{enumerate}
\end{corollary}
\begin{remark}
    If $\rho(K,C)=1$, then the first part of Corollary \ref{corollary: strong approximation 1} states that for any nonempty set of places $T$, $V$ satisfies strong approximation off $T$ if and only if $V$ is simply connected. On the other hand, if $\rho(K,C)=\mathbb{N}^*$ then $V$ satisfies strong approximation off $T$ if and only if its fundamental group is finite, or equivalently if and only if $V$ does not have torus factors by \cite[Exercise 12.1.6.]{CLS11}.
\end{remark}
In Section \ref{section: consequences}, we give two more characterisations of strong approximation on split toric varieties. Corollary \ref{corollary: strong approx and Pic} characterizes strong approximation in terms of the Picard group and is valid over any PF field.

Corollary \ref{corollary: strong approximation and Brauer Manin} implies that, over a number field, a smooth split toric variety satisfies strong approximation off a finite nonempty set of places if and only if the Brauer group modulo its constants vanishes. This strengthens the results of Cao and Xu in \cite{CaXu13} when the toric variety is split toric by also considering infinite places and generalizes them by also considering function fields. A similar result has recently also been shown over number fields by Santens in \cite[Theorem 1.3]{San23-2}, which implies that the algebraic Brauer-Manin obstruction is the only obstruction to strong approximation if the global sections of $V_{\overline{K}}$ are constant.

By applying Theorem \ref{theorem: M-approx} to Campana points as defined in Definition \ref{def: Campana points and Darmon points}, we obtain the following generalization of \cite[Theorem 1.2(i)]{NaSt20}:
\begin{corollary} \label{corollary: M-approx for Campana points}
	Let $(K,C)$ be a PF field, let $X$ be a complete normal split toric variety and let $T\subset \Omega_K$ be a finite set of places. Let $(X,M)$ be the toric pair corresponding to the Campana points on $(X,D_{\mathbf{m}})$ as defined in Definition \ref{def: Campana points and Darmon points}, where $m_1,\dots,m_n\in \mathbb{N}^*\cup\{\infty\}$ and $$D_{\mathbf{m}}=\sum_{i=1}^n \left(1-\frac{1}{m_i}\right)D_i.$$ Then $(X,M)$ satisfies $M$-approximation off $T$ if $X\setminus \lfloor D_{\mathbf{m}} \rfloor$ satisfies the conditions for strong approximation given in Corollary \ref{corollary: strong approximation 1} or Corollary \ref{corollary: strong approx and Pic}. If furthermore $\Pic(C)$ is finitely generated or $T=\emptyset$, then the converse also holds. In particular, $X$ satisfies $M$-approximation if $m_i<\infty$ for all $i=1,\dots, n$.
\end{corollary}
Corollary \ref{corollary: M-approx for Campana points} in the case of function fields of curves over algebraically closed fields of characteristic $0$ has been proven independently in a recent work by Chen, Lehmann and Tanimoto \cite{CLT24}, under the additional assumption that $m_i<\infty$ for all $i=1,\dots,n$.
They also obtain analogues of this result for other Campana pairs $(X,D_{\mathbf{m}})$ which are ``Campana rationally connected''. The method they use to prove the result relies on logarithmic geometry, and differs from the approach taken in this article.

We also study failures of the $\cM$-Hilbert property on split toric varieties. The main result in this direction is Theorem \ref{theorem: thinness}, which gives general sufficient conditions for the $\cM$-Hilbert property to fail, and gives a measure of how badly it fails. It also gives a precise characterisation for Zariski density of the set of $\cM$-points. As a consequence of the theorem it follows that over global fields, $M$-approximation is equivalent to the $\cM$-Hilbert property.

\begin{corollary} \label{corollary: M-approx equiv to M-Hilbert}
    Let $(K,C)$ be a global field, let $B\subset C$ be a nonempty open set and set $T=\Omega_K\setminus B$. Let $(X,M)$ be a toric pair where $X$ is a normal complete split toric variety over $K$ with toric integral model $(\cX,\cM)$ over $B$. Then $(X,M)$ satisfies $M$-approximation off $T$ if and only if the $\cM$-Hilbert property over $B$ is satisfied, meaning that $(\cX,\cM)(B)$ is not thin. 
 
    If $T\neq \emptyset$, then the same holds for \textit{any} integral model $(\cX,\cM)$ over $B$ such that $(\cX,\cM)(B)\neq \emptyset$.
\end{corollary}

\subsection{Darmon points and root stacks}
In the final section, we elucidate the relation between Darmon points and root stacks, and we relate the conditions in Theorem \ref{theorem: M-approx} to the fundamental group of the associated root stack. Proposition \ref{prop: relation Darmon points root stack} shows that outside of the boundary, integral points on the root stack $(\cX,\sqrt[\mathbf{m}]{\cD})$ are the same as Darmon points on $\cX$. We also prove in Proposition \ref{prop: M-approximation implies strong approximation stack} that the pair $(X,M)$ corresponding to the Darmon points satisfies $M$-approximation if and only if the root stack $(X,\sqrt[\mathbf{m}]{D})$ satisfies strong approximation, as studied in \cite{Chr20, San23-1}.

In Lemma \ref{lemma: fundamental group toric stack}, we compute the étale fundamental group of a toric root stack, and show that it coincides with the profinite completion of the group $N/N_M$ considered in Theorem \ref{theorem: M-approx}. By combining the lemma with Theorem \ref{theorem: M-approx}, we obtain a characterisation of strong approximation for split toric root stacks, generalizing Corollary \ref{corollary: strong approximation 1}.
\begin{corollary} \label{corollary: M-approximation for Darmon points}
	Let $X$ be a smooth split toric variety over a PF field $(K,C)$ of characteristic $0$, let $D_1,\dots,D_n$ be the the torus invariant prime divisors on $X$, let $m_1,\dots, m_n\in \mathbb{N}^*\cup\{\infty\}$ and let $D_{\mathbf{m}}$ be the corresponding Campana divisor as in Definition \ref{def: Campana points and Darmon points}. Let $T\subset\Omega_K$ be a finite nonempty set of places and let $\overline{K}$ be an algebraic closure of $K$. Then
	\begin{enumerate}
	\item $(X,\sqrt[\mathbf{m}]{D})$ satisfies strong approximation off $T$ if $\pi_1(X_{\overline{K}},\sqrt[\mathbf{m}]{D_{\overline{K}}})$ is finite and $|\pi_1(X_{\overline{K}},\sqrt[\mathbf{m}]{D_{\overline{K}}})|\in\rho(K,C)$. The converse also holds if $\Pic(C)$ is finitely generated.
	\item $(X,\sqrt[\mathbf{m}]{D})$ satisfies strong approximation if and only if $(X_{\overline{K}},\sqrt[\mathbf{m}]{D_{\overline{K}}})$ is simply connected and $\cO(X_{\overline{K}},\sqrt[\mathbf{m}]{D_{\overline{K}}})=\overline{K}$.
	\end{enumerate}
\end{corollary}

\subsection{Acknowledgments}
I would like to thank my supervisor Marta Pieropan for her support and suggestions for this paper, in particular by giving the idea of studying generalizations of Campana points. I am also grateful to Sam Streeter and Tim Santens for the fruitful discussions we had in Marseille. I am thankful to Sam for his detailed comments, and to Cedric Luger for pointing out some typos. Finally, I would also like to thank Arno Fehm and naf for their answers to my questions on MathOverflow, which helped me in proving Lemma \ref{lemma: computation rho} and Lemma \ref{lemma: Picard group divisible}, respectively.
\section{Notation and preliminaries}
\subsection{Natural numbers}
We use the convention that the set of natural numbers $\mathbb{N}$ contains $0$ and we write $\mathbb{N}^*$ for the set of nonzero natural numbers. We also define the set of extended natural numbers $\lineN:=\mathbb{N}\sqcup \{\infty\}$ as the one point compactification of the discrete space $\mathbb{N}$. The topology on $\lineN$ is the topology such that the map $\lineN\rightarrow \mathbb{R}$ given by $n\mapsto \frac{1}{n+1}$ is a homeomorphism onto its image.
We extend the greatest common divisor function to allow its arguments to lie in $\lineN$ by setting $\gcd(\infty, a_1,\dots, a_n)=\gcd(a_1,\dots, a_n)$ and $\gcd(\infty)=0$.

\subsection{Geometry}
For a field $k$ we write $\overline{k}$ for a choice of an algebraic closure. All schemes are taken to be separated.
For a scheme $X$ over a base scheme $S$ and a morphism $S'\rightarrow S$ of schemes, we denote the base change by $S'$ as $X_{S'}:=X\times_S S'$. If $S'=\spec R$, we also write $X_{R}$ in place of $X_{S'}$.

Given a $\mathbb{Q}$-Weil divisor $D=\sum_i a_iD_i$ on $X$ we define its floor as $\lfloor D\rfloor=\sum_i \lfloor a_i\rfloor D_i$.
If $D$ is an effective Cartier divisor on $X$, then we will routinely identify it with the closed subscheme of $X$ defined by the sheaf of ideals $\cO_X(-D)\subset \cO_X$. If $D_1$ and $D_2$ are Cartier divisors, then the closed subscheme $D_1+D_2$ is defined by the ideal sheaf $\cI=\cO_X(-(D_1+D_2))\subset \cO_X$.

We define a variety over a field $k$ to be a separated scheme of finite type over $k$, and a curve to be an integral variety of dimension $1$. Curves are defined to be integral and separated of finite type, but not necessarily geometrically integral.

If $k$ is a topological field and $X$ is a variety over $k$, then $X(k)$ is equipped with the induced topology. This is the topology such that for any affine open subvariety $U\subseteq X$ with a closed embedding $U\rightarrow \mathbb{A}_k^n$ the map $U(k)\rightarrow k^n$ is a homeomorphism onto its image.

All fundamental groups considered in this article are étale fundamental groups.
\subsection{PF fields} \label{section: PF fields}
We now introduce PF fields, based on a course taught by Artin at Princeton in 1950/51, of which the lecture notes are found in \cite[Chapter 12]{Art67}. See Remark \ref{remark: product formula} for the etymology of the term.
\begin{definition}
A \textit{PF field} is defined to be a pair $(K,C)$ where either
\begin{itemize}
\item $K$ is a number field, the function field of $C=\spec(\cO_K)$, where $\cO_K$ is the ring of integers of $K$, or
\item $K$ is the function field of a regular projective curve $C$ over a field $k$.
\end{itemize}
We call $K$ a \textit{global field} if $K$ is a number field or $k$ is finite.
\end{definition}
\begin{remark}
    The scheme $C$ is specified in order to give a good notion of a place of $K$ if $K$ is a function field. For example if $k=l(C')$ is a function field of some curve $C'$ over a field $l$ and $K=k(C)$, then the curve $C$ cannot be recovered from the field $K$ alone. This issue does not arise if $k$ is finite or when the embedding $k\rightarrow K$ is specified.
\end{remark}
Note that every finitely generated field extension $K/k$ of transcendence degree $1$ over a field $k$ is naturally a PF field. Each such field is the function field of an affine curve over $k$, which we can compactify and normalize to obtain a regular projective curve $C$. This curve is the unique regular projective curve $C$ with $K=k(C)$, since any birational map $C\rightarrow C'$ to a regular projective curve is a morphism and therefore an isomorphism.
\begin{remark} \label{remark: take geometrically connected}
Note that while the curve $C$ is regular, it need not be geometrically connected nor geometrically reduced over $k$. For example, we can consider $C=\P^1_k\times_k \spec l$ for a finite separable extension $l/k$ or a finite inseparable extension $l/k$, respectively.
The former subtlety disappears if we replace the base field $k$ with its algebraic closure $k'$ in $K$, since $C$ is a geometrically connected curve over $k'$. In particular, if $k$ is perfect, then $C$ will be a geometrically integral curve over $k'$.
However, $C$ need not be geometrically reduced over $k'$ without this assumption, as shown by the curve
$$C=\{sx^p + ty^{p} + z^p=0\}\subset \P_k^2$$ over the field $k=\mathbb{F}_p(s,t),$  where $p$ is a prime number. Furthermore, even if $C$ is geometrically integral over $k$, it need not be smooth as shown by the curve
$$C=\{t x^p + z^{p-1} y + y^p=0\}\subset \P_k^2$$ over the field $k=\mathbb{F}_p(t)$, where $p>2$ is a prime number. 
\end{remark}
Note that if $B\subset C$ is an open subscheme, then $B$ is affine, unless $K$ is a function field and $B=C$.

We use the convention that a discrete valuation on $K$ contains $1$ in its image. 
If $K$ is a number field, then a \textit{finite place} of $K$ is a discrete valuation on $K$. If $K$ is a function field, then a \textit{finite place} of $K$ is a discrete valuation on $K$ which is trivial on $k$. We denote the set of finite places of $K$ by $\Omega_K^{<\infty}$. There is a natural bijection between the closed points on $C$ and $\Omega_K^{<\infty}$, and we will thus routinely identify finite places and closed points. Given a finite place $v$ of $K$, its degree is the degree of the associated closed point, and define the absolute value on $K$ induced by $v$ as $$|a|_v=p^{-\deg (v)v(a)}.$$ Here, $p$ is the characteristic of the residue field $k_v$ if $\CHAR(k_v)>0$, and $p=2$ if $\CHAR(k_v)=0$.

For a number field, an \textit{infinite place} is an embedding $v\colon K\rightarrow \mathbb{C}$, where conjugate embeddings into $\mathbb{C}$ are identified. We denote the set of infinite places of $K$ by $\Omega_K^\infty$. An infinite place $v$ is \textit{real} if it factors through an embedding $K\rightarrow \mathbb{R}$ and \textit{complex} otherwise. An infinite place $v$ induces an absolute value on $K$ by $$|a|_v=|v(a)|^e,$$
where $|\cdot|$ is the standard absolute value on $\mathbb{C}$ and $e=1$ if $v$ is real and $e=2$ if $v$ is complex.
If $K$ is a function field, we set $\Omega_K^\infty=\emptyset$.
For any PF field $K$ we define the set of \textit{places} on $K$ to be $$\Omega_K=\Omega_K^{<\infty}\sqcup \Omega_K^{\infty}.$$

\begin{remark} \label{remark: product formula}
The term PF field stands for Product Formula field, named after the formula
$$\prod_{v\in \Omega_K} |x|_v=1$$
for all $x\in K^\times$.
\end{remark}

For a place $v\in \Omega_K$, we denote by $K_v$ the completion of $K$ with respect to the absolute value $|\cdot|_v$. This field is locally compact if and only if $K$ is a global field. If $v$ is a finite place, we set $$\cO_v=\{x\in K_v\mid \,  |x|_v\leq 1\}.$$
For an infinite place $v$, we simply set $\cO_v=K_v$.
If $\cX$ is a scheme over an open subscheme $B\subset C$, then for a place $v\in B$, we write $\cX_v=\cX\times_{B} \spec \cO_v$.

\subsection{Restricted products and adeles} \label{subsection: restricted products and adeles}
\begin{definition} \label{def: restricted product}
Let $I$ be an index set, and for each $i\in I$, let $X_i$ be a topological space with a subspace $U_i\subseteq X_i$, which is not necessarily open. Then the underlying set of the \textit{restricted product} of these spaces is
$$\prod_{i\in I}(X_i, U_i):=\left\{(x_i)_{i\in I}\in \prod_{i\in I}X_i \Biggm| x_i\in U_i\text{ for all but finitely many } i\in I\right\}.$$
This set is given the finest topology such that for all finite sets $J\subset I$, the inclusion map
$$\pi_J\colon\prod_{i\in J} X_i \times \prod_{i\in I\setminus J}U_i\hookrightarrow \prod_{i\in I}(X_i, U_i)$$
is continuous.
\end{definition}

If $J'\subset J$ are finite subsets of $I$, then the map $\prod_{i\in J'} X_i \times \prod_{i\in I\setminus J'}U_i\rightarrow \prod_{i\in J} X_i \times \prod_{i\in I\setminus J}U_i$ is a continuous map. Thus for any finite subset $J\subset I$ there is a natural homeomorphism
$$\prod_{i\in I}(X_i, U_i)\cong \prod_{i\in J}X_i \times\prod_{i\in I\setminus J}(X_i, U_i).$$

We will now consider the topological properties of two types of inclusions between restricted products.
\begin{proposition} \label{prop: inclusions restricted product}
    For $i\in I$, let $X_i$ be a topological space with subspaces $Z_i\subset Y_i\subset X_i$. Then
    \begin{enumerate}
        \item if $\prod_{i\in I} (X_i,Z_i)\neq \emptyset$, the natural inclusion $$\prod_{i\in I} (X_i,Z_i)\hookrightarrow \prod_{i\in I} (X_i,Y_i)$$ is continuous and has dense image.
        \item The natural inclusion $\prod_{i\in I} (Y_i,Z_i)\hookrightarrow \prod_{i\in I} (X_i,Z_i)$ is a topological embedding and it is open if $Y_i\subset X_i$ is open for all $i\in I$.
    \end{enumerate}
\end{proposition}
\begin{proof}
    For a finite set $J\subseteq I$, we define $(X,Y)_J:= \prod_{i\in J} X_i\times \prod_{i\in I\setminus J} Y_i$ and we define $(X,Z)_J$ and $(Y,Z)_J$ similarly.
    We will first prove the first statement.
    By the definition of the restricted product topology, a subset $V\subseteq \prod_{i\in I}(X_i, Y_i)$ is open if and only if for every finite subset $J\subseteq I$, $V\cap (X,Y)_J$ is open in $(X,Y)_J$. Thus if $V$ is an open subset of $\prod_{i\in I} (X_i,Y_i)$, then $V\cap (X,Z)_J$ is open in $(X,Z)_J$ so the inclusion is continuous.
    
    To show that the map has dense image, we show that $V\cap \prod_{i\in I} (X_i,Z_i)\neq \emptyset$ for every nonempty subset $V\subset \prod_{i\in I} (X_i,Y_i)$. Let $J\subseteq I$ be such that $(X,Z)_J\neq \emptyset$. For any finite subset $J\subseteq I$ such that $V\cap (X,Y)_J\neq \emptyset$, the set $V\cap (X,Y)_J$ contains an nonempty open subset $\prod_{i\in I} V_i$, with $V_i$ open in $X_i$ if $i\in J$, $V_i$ open in $Y_i$ if $i\in I\setminus J$ and $V_i=Y_i$ for all but finitely many $i\in I$, by basic properties of the product topology. Let $J'\subseteq I$ be the finite subset of $i\in I$ for which $V_i\neq Y_i$ or $Y_i\neq \emptyset$, then $V\cap (X,Z)_{J'}\neq \emptyset$ and the map therefore has dense image.

    Now we will prove the second statement. The image $\im(\iota)$ of $\iota\colon \prod_{i\in I} (Y_i,Z_i)\hookrightarrow \prod_{i\in I} (X_i,Z_i)$ with the subspace topology has the finest topology such that every continuous map $A\rightarrow \prod_{i\in I} (X_i,Z_i)$ with set-theoretic image in $\im(\iota)$ factors continuously through $\im(\iota)$. Therefore if $\iota$ is continuous, then it is an embedding.
    If $V\subseteq \prod_{i\in I}(X_i, Z_i)$ is open, then $V\cap (X,Z)_J$ is open in $(X,Z)_J$ for every finite subset $J\subseteq I$, and since $(Y,Z)_J$ is a subspace of $(X,Z)_J$, $V\cap (Y,Z)_J$ is open in $(Y,Z)_J$. Thus the inclusion map is continuous. If furthermore $Y_i\subset X_i$ is open for all $i\in I$, then $V\cap \prod_{i\in I}(X_i, Z_i)$ is open, so the inclusion map is an open map.
\end{proof}

Using this construction, we define the ring of adeles.
\begin{definition} \label{def: ring of adeles}
    Let $(K,C)$ be a PF field and let $T$ be a finite set of places. The \textit{ring of adeles of $K$ prime to} $T$ is the topological $K$-algebra $$\mathbf{A}^T_K=\prod_{v\in \Omega_K\setminus T} (K_v,\cO_v).$$
    For a nonempty open subset $B\subset C$ the \textit{ring of $B$-integral adeles prime to} $T$ is the topological $K$-algebra $$\mathbf{A}^T_B=\prod_{v\in B} \cO_v\times \prod_{v\in \Omega_K\setminus B}K_v.$$
\end{definition}
The following proposition will be used in Proposition \ref{prop: relation Darmon points root stack} to relate adelic Darmon points to adelic points on the associated root stack.
\begin{proposition} \label{prop: Picard group Adeles}
    Let $(K,C)$ be a PF field, let $T$ be a finite set of places and let $B\subset C$ be an open subset. Then every finitely generated ideal in $\mathbf{A}^T_K$ or $\mathbf{A}^T_B$ is principal. In particular, $\Pic(\mathbf{A}^T_K)=\Pic(\mathbf{A}^T_B)=0$.
\end{proposition}
\begin{proof}
    We will prove the statement for $\mathbf{A}_K^T$ and note that the statement for $\mathbf{A}_B^T$ follows analogously. If we have an ideal $I=((a_v)_{v\in \Omega_K\setminus T},(b_v)_{v\in \Omega_K\setminus T})$, then we can for every finite place $v\in \Omega_K^{<\infty}\setminus T$ consider $t_v,s_v\in \cO_v^\times$ such that $v(c_v)=\min(v(a_v),v(b_v))$, where $c_v=t_v a_v+ s_v b_v$. Then $I=((c_v)_{v\in \Omega_K\setminus T})$ and therefore $I$ is principal. Thus induction on the number of generators shows that every finitely generated ideal is principal. By \cite[Tag 0B8N]{Stacks}, invertible ideals are finitely generated, hence $\Pic(\mathbf{A}_K^T)=0$.
\end{proof}

\section{\texorpdfstring{$\cM$}{M}-points}
In this section we generalize the notions of integral points and Campana points. We fix a PF field $(K,C)$ and an open subscheme $B\subset C$.
\subsection{Integral models of pairs} \label{subsection: pairs}
First we define pairs and their integral models.
\begin{definition} \label{def: pair}
Let $X$ be a scheme over a scheme $B$ and let $(D_{\alpha})_{\alpha\in \cA}$ be a finite tuple of
closed subschemes on $X$. Let $\fM\subseteq \lineN^\cA$ be a subset satisfying $(0,\dots, 0)\in \fM$ and such that for all $\mathbf{m}\in\fM$ the element $\mathbf{m}'$ defined by $$m_\alpha'=\begin{cases}
    0 & \text{if }m_{\alpha}\neq \infty, \\\infty & \text{if }m_\alpha= \infty,
\end{cases}$$ lies in $\fM$. \newline For such a set we let $$M:=((D_{\alpha})_{\alpha\in \cA},\fM)$$ and we call $(X,M)$ a \textit{pair} over $B$, $\fM$ the \textit{set of multiplicities}, and $M$ the \textit{parameter set}.
We will call $\bigcup_{\alpha\in \cA} D_{\alpha}$ the \textit{boundary} of $(X,M)$ and denote its complement in $X$ by $U=X\setminus \bigcup_{\alpha\in \cA} D_{\alpha}$. 
If $\cA=\emptyset$, then we write $M=0$ and we say that $M$ is trivial.
\end{definition}
The technical condition on $\fM$ is very mild and it will ensure that for any place $v\in \Omega_K$ the $\cM$-points over $\cO_v$ lie in the $\cM$-points over $K_v$, as we will define in Definition \ref{def: M-points}.  
This definition generalizes the notion of Campana pairs given in \cite[D\' efinition 2.1]{Cam10}, which we recover if $B=\spec \mathbb{C}$, $X$ is a normal variety, the $D_\alpha$ are Weil divisors and $\fM$ is chosen to encode the Campana condition found in Definition \ref{def: Campana points and Darmon points}.
\begin{remark}
    Note that the notion of a pair $(X,M)$ is more general than the notion of an $M$-curve as studied by Darmon in \cite{Dar97}, even when we restrict $X$ to be a curve. However, an $M$-curve can be naturally viewed as a pair, as we will see in Definition \ref{def: Campana points and Darmon points}.
\end{remark}

In the later sections on toric varieties, the subschemes $D_{\alpha}$ will be divisors, but there are advantages to allowing them to be arbitrary closed subschemes, as we will see later in this section.
As in the case of Campana orbifolds \cite{Cam05,Abr09,PSTV20}, the points on the pair are only defined after a choice of an integral model, which we define as follows.
\begin{definition}
Let $X$ be a proper variety over a PF field $(K,C)$. A scheme $\mathcal{X}$ over $B$ is a \textit{integral model of $X$ over $B$} if it is proper over $B$ and its generic fiber is isomorphic to $X$ over $K$.
\end{definition}
Note that we do not require the integral models to be flat over $B$.
\begin{definition}
Given a pair $(X,M)$ with $X$ a proper variety over a PF field $(K,C)$, an \textit{integral model of $(X,M)$ over $B$} is a pair $(\cX,\cM)$, where $\cX$ is an integral model of $X$ over $B$ and $\cM=((\cD_{\alpha})_{\alpha\in \cA}, \fM)$, where for all $\alpha\in\cA$, $\cD_{\alpha}\subset \cX$ is an integral model of $D_{\alpha}$ over $B$. We also say that $(\cX,\cM)$ is a pair over $B$.
\end{definition}
Note that we do not require the $\cD_{\alpha}$ to be flat over $B$.
Given an integral model over $B$, we can restrict it to open sets in $B$ as follows.
\begin{definition}
	Let $(X,M)$ be a pair with integral model $(\cX,\cM)$ over $B$. If $B'\subseteq B$ is a nonempty open subset, then we define the integral model $(\cX,\cM)_{B'}$ over $B'$ as
	$$(\cX,\cM)_{B'}=(\cX\times_B B', \cM_{B'}),$$
	where $\cM_{B'}=((\cD_{\alpha}\times_B B')_{\alpha\in \cA}, \fM)$.
\end{definition}

In the literature on Campana points, such as \cite{PSTV20,NaSt20,MiNaSt22}, there is a canonical choice of an integral model $\cD_{\alpha}$ of $D_{\alpha}$ by taking the Zariski closure and endowing it with the reduced scheme structure. We generalize this construction to allow $D$ to be nonreduced.
\begin{proposition}
Let $X$ be a proper variety over $K$ with integral model $\cX$. If $D\subseteq X$ is a closed subscheme over $K$, then there exists a unique closed subscheme $\cD^c\subseteq \cX$ with generic fiber $\cD^c_K=D$ such that the inclusion $\cD^c\subseteq \cX$ factors through every closed subscheme $\widetilde{\cD}\subset \cX$ with $\widetilde{\cD}_K=D$.
\end{proposition}
\begin{proof}
	Consider the sheaf of ideals $\cI$ on $\cX$ defined as the kernel of the composition $\cO_\cX\rightarrow \cO_X\rightarrow \cO_D$ of $\cO_\cX$-algebras. This sheaf defines a closed subscheme $\cD^c$ with $\cD^c_K=D$. By construction $D^c$ has the desired universal property.
\end{proof}
We will call $\cD^c$ the \textit{closure} of $D$ in $\cX$.
The next proposition shows that the closure interacts very well with the structure of Cartier divisors.
\begin{proposition}
	Let $X$ be a proper variety over a PF field $(K,C)$ with integral model $\cX$ over $B\subset C$.
	If $\cD$ is an irreducible effective Cartier divisor on $\cX$ with $\cD_K=D\neq \emptyset$, then $\cD=\cD^c$. In particular, if $D_1,D_2\subseteq X$ are subschemes such that their closures $\cD_1^c, \cD_2^c$ in $\cX$ are effective Cartier divisors, then $(\cD_1+\cD_2)^c$ is an effective Cartier divisor and $$\cD_1^c+\cD_2^c=(\cD_1+\cD_2)^c$$ as subschemes of $\cX$. 
\end{proposition}
\begin{proof}
By construction, $\cD^c$ is a closed subscheme of $\cD$. Therefore, by \cite[Tag 0AGB]{Stacks}, there exists a Cartier divisor $\cD'$ on $\cX$ such that $\cD'\subset \cD^c$ is an isomorphism outside codimension $2$. Now \cite[Tag 02ON]{Stacks} implies that there exists a Cartier divisor $\cD''$ such that $\cD=\cD'+\cD''$. As $\cD$ is irreducible, we find $\cD_{\red}=\cD''_{\red}$ or $\cD''=\emptyset$. If we write $D'=\cD'_K$ and $D''=\cD''_K$, then we see $D=D'+D''$. However, since $D$ and $D'$ are isomorphic outside a codimension $2$ subset, the codimension of $D''$ is at least $2$ so $D''=\emptyset$, and thus $\cD''=\emptyset$. This implies $\cD=\cD'$ and therefore $\cD=\cD^c$. The second part of the proposition now follows from the fact that $D_1^c+D_2^c$ is a Cartier divisor with generic fiber $D_1+D_2$, so $(\cD_1+\cD_2)^c=\cD_1^c+\cD_2^c$.
\end{proof}
Using the above construction, an integral model of the variety $X$ induces an integral model of the pair $(X,M)$:
\begin{definition}
Given a pair $(X,M)$ and an integral model $\cX$ of $X$ over $B$, the \textit{integral model of $(X,M)$ induced by $\cX$} is the pair $(\cX,\cM^c)$ over $B$, where $\cM^c=((\cD^c_{\alpha})_{\alpha\in \cA},\fM)$.
\end{definition}
Note that by spreading out \cite[\S 3.2]{Poo17} any proper variety $X$ over $K$ has an integral model over some nonempty open subscheme $B\subset C$. Hence, any pair $(X,M)$ over $K$ has an integral model over such an open subscheme $B\subset C$.

\subsection{Multiplicities and \texorpdfstring{$\cM$}{M}-points}
Now we will define intersection multiplicities and $\cM$-points. 
As before, we let $X$ be a proper variety over a PF field $(K,C)$ with integral model $\cX$ over an open subscheme $B\subset C$. Let $v\in B$ be a closed point and let $P\in X(K_v)$. By the valuative criterion of properness, $P$ lifts to an unique point $\mathcal{P}\in \mathcal{X}(\cO_v)$.
For a closed subscheme $\cD\subseteq \cX$ we consider the scheme theoretic intersection $\cP\cap \cD$
, which is defined as the fiber product of $\cP\colon \spec \cO_v\rightarrow \cX$ and the closed immersion $i_{\cD}\colon\cD\hookrightarrow \cX$:
$$\begin{tikzcd}
	\cP\cap \cD \arrow[d, hook] \arrow[r] & \cD \arrow[d, "i_{\cD}", hook] \\
	\spec \cO_v \arrow[r, "\cP"]  & \cX.  
\end{tikzcd}$$
As base change preserves closed immersions it follows that $\cP\cap \cD=\spec(\cO_v/I)$
for an ideal $I\subseteq \cO_v$. As $\cO_v$ is a discrete valuation ring, $I=(0)$ or $I=(\pi^n)$ for some $n\in\mathbb{N}$, where $\pi$ is a uniformizer of $\cO_v$. As in \cite[Definition 2.4]{MiNaSt22} we make the following definition.
\begin{definition}
The \textit{(local) intersection multiplicity} $n_v(\cD, \cP)$ is defined to be
$$n_v(\cD, \cP)=\begin{cases}
   n \quad &\text{if } I=(\pi^n), \\
   \infty \quad &\text{if } I=(0).
\end{cases}$$
Note in particular that $n_v(\cD, \cP)=\infty$ exactly if $\cP\subseteq \cD(\cO_v)$.
\end{definition}
This definition agrees with the classical notion of local intersection multiplicity: if $\cX$ is a smooth surface over an algebraically closed field, $\cO_v=\cO_{\cC,p}$ is the local ring of a point $p$ on a smooth curve $\cC\subset \cX$ and $\cD\subset \cX$ is a Cartier divisor, then $n_v(\cP, \cD)=(\cC\cap \cD)_p$ is the local intersection multiplicity of $\cC$ and $\cD$ in $p$ as defined in \cite[Chapter V]{Har77} unless $\cC\subset \cD$ in which case $n_v(\cD, \cP)=\infty$.

\begin{example}
    If $\cX=\P^n_{\cO_v}$ and $\cD_i$ is the $i$-th coordinate hyperplane, then given an integral point $\cP=(a_0:\dots : a_n)$, with $a_i\in \cO_v$ for all $i\in \{0,\dots, n\}$ and $v(a_i)=0$ for some $i\in \{0,\dots, n\}$, the intersection multiplicity is just the valuation $n_v(\cD_i,\cP)=v(a_i)$.
\end{example}

The next proposition shows that the intersection multiplicity respects addition of Cartier divisors.
\begin{proposition} \label{prop: additivity}
Let $\cD_1$ and $\cD_2$ be Cartier divisors on $\cX$ and let $\cP\in \cX(\cO_v)$. Then $$n_v(\cD_1+\cD_2,\cP)=n_v(\cD_1,\cP)+n_v(\cD_2,\cP).$$
\end{proposition}
\begin{proof}
This follows from the equality
 $$\cP^*\cO_X(-(\cD_1+\cD_2))\cO_Y=\cP^*(\cO_X(-\cD_1)\cO_X(-\cD_2))\cO_Y=\cP^* \cO(-\cD_1)\cO_Y \cdot \cP^*\cO(-\cD_2)\cO_Y,$$
        of ideal sheaves on $Y=\spec \cO_v$,
	which implies the identity. 
\end{proof}
Given a PF field $(K,C)$ and a pair $(\cX,\cM)$ over $B\subset C$ and a finite place $v\in B$, we define the map
\begin{align*}
\mult_{v}\colon \cX(\cO_v)\rightarrow \lineN^\cA, \quad
\cP \mapsto (n_v(\cD_\alpha, \cP))_{\alpha\in \cA}.
\end{align*}
We also define, for a field extension $L/K$, the map
\begin{align*}
\mult_L\colon X(L)\rightarrow \{0,\infty\}^\cA, \quad
\cP \mapsto (n_L(D_\alpha, P))_{\alpha\in \cA},
\end{align*}
where $$n_L(D_\alpha, P):=\begin{cases}
    0 \quad &\text{if } P\not \in D_\alpha(L), \\
    \infty \quad &\text{if } P\in D_\alpha(L)
\end{cases}$$
indicates whether the point $P$ lies in $D_{\alpha}$.
Using these notions we are finally ready to define $\cM$-points.
\begin{definition} \label{def: M-points}
Let $(K,C)$ be a PF field, and let $(X,M)$ be a pair over $K$ with integral model $(\cX,\cM)$ over an open subscheme $B\subset C$.
For a field extension $L/K$, we set
$$(X,M)(L)=(\cX,\cM)(L)=\{P\in X(L)\mid \mult_L(P)\in \fM\}.$$
For a finite place $v\in B$, the set of \textit{$v$-adic $\cM$-points on $(\cX,\cM)$} is defined as
\begin{equation} \label{equation: v-adic M-points}
    (\cX,\cM)(\cO_v)=\{\cP\in \cX(\cO_v)\mid \mult_v(\cP)\in \fM\}.
\end{equation}
If $v\in \Omega_K\setminus B$, we set
$$(\cX,\cM)(\cO_v)=(X,M)(K_v).$$

The set of \textit{$\cM$-points on $(\cX,\cM)$ over $B$} is defined as the subset of $\cX(B)$ satisfying Condition \eqref{equation: v-adic M-points} at every place $v\in B$:
$$(\cX,\cM)(B)=\{\cP\in \cX(B)\mid \mult_v(\cP_v)\in \fM\text{ for all } v\in B\}.$$
\end{definition}

Note that $\bigcup_{B\subset C}(\cX,\cM)(B)=(X,M)(K)$, where the union runs over all nonempty open subschemes $B$ of $C$.
\begin{definition}
	Let $(X,M)$ be a pair over $K$ with integral model $(\cX,\cM)$ be a pair over a scheme $B$. We define $\fM_{\fin}=\fM\cap \mathbb{N}^\cA$ and we define $M_{\fin}$ and $\cM_{\fin}$ by replacing $\fM$ by $\fM_{\fin}$.
\end{definition}
 The points on $(\cX,\cM_{\fin})$ are those in $(\cX,\cM)$ whose image does not lie in the boundary $\bigcup_{\alpha\in \cA} \cD_{\alpha}$.

\begin{remark}
If we assume that the subschemes $\cD_{\alpha}$ are all Cartier divisors on $\cX$, then we can give a different description of $(\cX,\cM)(B)$. Namely, if the image of a morphism $\cP\colon B\rightarrow \cX$ does not lie in the boundary $\bigcup_{\alpha\in \cA} \cD_{\alpha}$, then the intersection with $\cD_{\alpha}$ is simply the pullback $\cP\cap \cD_{\alpha}=\cP^*\cD_{\alpha}$. So it follows that
$(\cX,\cM_{\fin})(B)$ is the set of points $\cP\in \cX(B)$ not contained in the boundary, such that there exist $\mathbf{m}_1,\dots, \mathbf{m}_k\in \fM_{\fin}$ and distinct prime divisors $\widetilde{\cD}_1,\dots,\widetilde{\cD}_k$ in $B$ such that $$\cP^*\cD_{\alpha}= \sum_{i=1}^k m_{i,\alpha} \widetilde{\cD}_i$$ for all $\alpha\in \cA$.
\end{remark}

For many interesting choices of $\cM$, the set of multiplicities $\fM$ is an open subset of $\lineN^\cA$. For example, this is the case for integral points, Campana points and strict Darmon points. In this case, the next proposition shows that the property of being a $\cM$-point is open in $X(K_v)$.
\begin{proposition} Let $X$ be a proper variety over a PF field $(K,C)$ and let $(X,M)$ be a pair with integral model $(\cX,\cM)$ over $B\subset C$. 
Then the map $\mult_v\colon X(K_v)\rightarrow \lineN^\cA$ is continuous for every finite place $v\in B$.

Therefore, if $\fM\subseteq \lineN^\cA$ is an open (or closed) subset, then $(\cX,\cM)(\cO_v)$ is an open (or closed) subset of $X(K_v)$.
\end{proposition} \label{prop: continuity multiplicity}
\begin{proof}
It suffices to prove $n_v(\cD,-)\colon X(K_v)\rightarrow \lineN$ is continuous for a single subscheme $\cD$, as continuity is equivalent to continuity in all coordinates. Note that for $\cP\in \cX(\cO_v)=X(K_v)$, the multiplicity $n_v(\cD,\cP)$ is the largest integer $n_0$ such that there exists a factorisation
$$\begin{tikzcd}
	& \cD\times_{\cO_v} \spec \cO_v/(\pi_v^{n_0}) \arrow[d] \\
	\spec \cO_v/(\pi_v^{n_0}) \arrow[r] \arrow[ru, dotted] & \cX\times_{\cO_v} \spec \cO_v/(\pi_v^{n_0}),                                      
\end{tikzcd}
$$
where the horizontal and vertical homomorphisms are induced by $\cP$ and by the inclusion morphism $i_{\cD}\colon \cD\rightarrow \cX$, respectively. In particular, if two points $\cP, \cP'\in \cX(\cO_v)$ have the same reduction modulo $(\pi_v^n)$ for some integer $n$, either $$n_v(\cD,\cP)=n_v(\cD,\cP')<n$$ or $$\min\{n_v(\cD,\cP),n_v(\cD,\cP') \}\geq n.$$ 
Since the collection of open sets of the form $$U(\cP,n)=\{\cP'\in \cX(\cO_v)\mid \cP \bmod \pi_v^n=\cP' 
\bmod \pi_v^n\in \cX(\cO_v/\pi_v^n)\},$$ $\cP\in X(K_v)$ and $n\in \mathbb{N}$, forms a basis for the topology on $\cX(\cO_v)$, it follows that $n_v(\cD,-)$ is indeed continuous. Thus $\mult_v$ is a continuous map.
\end{proof}

\subsection{Examples of \texorpdfstring{$\cM$}{M}-points}\label{subsection: M-points examples}
Let $X$ be a proper variety over a PF field $(K,C)$ with a finite collection of closed subschemes $(D_{\alpha})_{\alpha\in \cA}$. Fix an integral model $\cX$ of $X$ over $B\subset C$ and set $\cD_{\alpha}=\cD^c_{\alpha}$. By choosing different subsets $\fM\subset \lineN^\cA$, we can construct many different pairs $(\cX,\cM)$. We consider some choices, and afterwards we describe $\cM$-points on projective space for these choices. We write $U=X\setminus \bigcup_{\alpha\in \cA} D_{\alpha}$ and $\cU=\cX\setminus \bigcup_{\alpha\in \cA} \cD_{\alpha}$ for the complement of the boundary.
\begin{enumerate}
\item If $\fM=\{(0,\dots, 0)\}$, then the $\cM$-points over $B$ are the integral points on $\cU$: $(\cX,\cM)(B)=\mathcal{U}(B)$ and $(X,M)(K)=U(K)$.

More generally, if $\cB\subset \cA$ and $\fM=\{\bf{m}\in \lineN^\cA\mid m_{\alpha}=0 \text{ if }\alpha\in \cB\}$, then the $\cM$-points over $B$ are the integral points on $\cX\setminus \bigcup_{\alpha\in\cB}\cD_{\alpha}$:
$(\cX,\cM)(B)=(\cX\setminus \bigcup_{\alpha\in\cB}\cD_{\alpha})(B)$ and $(X,M)(K)=(X\setminus \bigcup_{\alpha\in\cB}D_{\alpha})(K)$. \label{pair integral points}
\item If $\fM=\lineN^\cA$, then the set of $\cM$-points is the entire set of rational points: $(\cX,\cM)(B)=(X,M)(K)=X(K)$. If on the other hand $\fM=\mathbb{N}^\cA$, then the set consists of only the points not contained in the boundary: $(\cX,\cM)(B)=(X,M)(K)=U(K)$.
\item If $\fM=\{0,1\}^\cA$, then $(\cX,\cM)(B)$ is the set of points on $\cX$ over $B$ that intersect all $\cD_{\alpha}$ transversally. As we will see, we can think of these points as a sort of ``squarefree'' points. We again have $(X,M)(K)=U(K)$. 
\item If $$\fM=\bigcup_{\alpha\in\cA}\{\mathbf{m}\in \lineN^\cA\mid w_{\alpha'}=0\, \forall \alpha'\neq \alpha\},$$ then $(\cX,\cM)(B)$ is the set of points on $\cX$ over $B$ which do not meet any intersection $\cD_{\alpha}\cap \cD_{\alpha'}$ for $\alpha,\alpha'\in\cA$, $\alpha\neq \alpha'$, while $(X,M)(K)$ consists of the rational points not contained in any of of the intersections $D_{\alpha}\cap D_{\alpha'}$.  
\end{enumerate}
For the following examples we assume that the closed subschemes $D_{\alpha}$ and $\cD_{\alpha}$ are prime Weil divisors, and $D_{\alpha}\neq D_{\alpha'}$ if $\alpha\neq \alpha'$. Consider a vector of multiplicities $\textbf{m}=(m_\alpha)_{\alpha\in \cA}$, where $m_{\alpha}\in \mathbb{N}^*\cup \{\infty\}$ and define the $\mathbb{Q}$-Weil divisor $\cD_{\textbf{m}}=\sum_{\alpha\in \cA}\left(1-\frac{1}{m_\alpha}\right)\cD_{\alpha}$, where we set $\frac{1}{\infty}=0$.
\begin{definition} For $\cX$ and $\cD_{\mathbf{m}}$ as above, we define special points as follows.
\label{def: Campana points and Darmon points}
\begin{itemize}
\item \textit{Campana points on $(\cX,\cD_{\mathbf{m}})$ over $B$} are the $\cM$-points over $B$ for the pair $(\cX,\cM)$, where $\fM$ is the collection of $\bf{w}\in \lineN^\cA$ such that for all $\alpha\in \cA$ we have
\begin{enumerate}
    \item $w_\alpha=0$ if $m_{\alpha}=\infty$ and
    \item $w_\alpha=0$ or $w_\alpha\geq m_\alpha$ if $m_\alpha\neq \infty$.
\end{enumerate}
\item \textit{Weak Campana points on $(\cX,\cD_{\mathbf{m}})$ over $B$} are the $\cM$-points over $B$ for $(X, \cM)$, where $\fM$ is the collection of all $\bf{w}$ such that
\begin{enumerate}
    \item $w_\alpha=0$ if $m_{\alpha}=\infty$ and
    \item either $w_\alpha=0$ for all $\alpha\in \cA$ or $$\sum_{\substack{\alpha\in \cA \\ m_{\alpha}\neq 1}} \frac{w_\alpha}{m_\alpha}\geq 1.$$
\end{enumerate}
\item \textit{Strict Darmon  points on $(\cX,\cD_{\mathbf{m}})$ over $B$} are the $\cM$-points where $\fM$ is the collection of $\bf{w}\in \mathbb{N}^\cA$ for which $m_\alpha|w_\alpha$ for all $\alpha\in \cA$. Here we use the convention that the only integer divisible by $
\infty$ is $0$. If we take the closure of $\fM$ in $\lineN^\cA$ then we obtain the \textit{Darmon points on $(\cX,\cD_{\mathbf{m}})$ over $B$}.
\end{itemize}
\end{definition}
Note that if $m_{\alpha}=\infty$ for all $\alpha\in \cA$, then all of the sets of $\cM$-points in Definition \ref{def: Campana points and Darmon points} reduce to the set of integral points on $\cU$.

Note that all examples given in this section satisfy the property that $\fM$ is an open subset of $\lineN^\cA$, except for the set $\fM$ encoding the multiplicities for the Darmon points. This follows from the fact that $\mathbb{N}\subset \lineN$ is an open subspace with the discrete topology and if $U\subset\lineN$ contains $\infty$, then $U$ is open if and only if it contains all integers greater than a fixed integer $N_0$. The set of multiplicities is also closed for the other examples except for the strict Darmon points, the integral points on $\cU$ and the rational points on $U$.

If we additionally assume that $X$ is smooth, $\cX$ is flat over $B$ and the $D_{\alpha}$ are Cartier divisors, then (weak) Campana points and Darmon points agree with their definition as given in \cite{MiNaSt22}. If $\sum_{\alpha\in \cA}D_{\alpha}$ is furthermore a strict normal crossings divisor, $X$ is geometrically integral and $\cX$ is regular, then the (weak) Campana points agree with the definition given in \cite{PSTV20}.

\begin{remark}
In \cite{MiNaSt22} strong Campana points and strong Darmon points are defined, of which the former were called Campana points in \cite{NaSt20, Str21}. The set of strong Campana points and the set of strong Darmon points are generally not examples of sets of $\cM$-points if the divisors $\cD_{\alpha}$ are not geometrically integral, as those points are defined using the intersection multiplicities of the irreducible components of $\cD_{\alpha, \cO_v}$. However, if the $D_{\alpha}$ are geometrically integral then strong Campana points and strong Darmon points coincide with Campana points and Darmon points, respectively.
\end{remark}
\begin{remark}
Consider a positive integer $m$ and a prime Cartier divisor $D$ on a smooth proper variety $X$ extending to a prime Cartier divisor $\cD$ on an integral model $\cX$. Then the Campana points and weak Campana points on $(\cX,(1-\tfrac{1}{m})\cD)$ coincide and agree with the weak Campana points as defined in \cite{Str21}, even if $D$ is not geometrically irreducible.
\end{remark}

\subsubsection{\texorpdfstring{$\cM$}{M}-points on projective space}\label{subsection: M-points on projective space}
To further illustrate the examples given above, we fix $K=\mathbb{Q}$, $B=\spec \Z$ and $X=\mathbb{P}^n_\mathbb{Q}$ with integral model $\cX=\mathbb{P}^n_\Z$. Fix $\cA=\{0,\dots, n\}$. For $i\in \{1,\dots, n\}$, we define $D_i$ to be the coordinate hyperplane $\{x_i=0\}$ and let $\cD_i$ be the Zariski closure of $D_i$ in $\cX$. Then for a point $P=(a_0:\dots: a_n)$ with $a_i\in \Z$ for all $i\in\{1,\dots, n\}$ and $\gcd(a_0,\dots,a_n)=1$, the identity $n_p(D_i,P)=v_p(a_i)$ holds for every prime number $p$. In particular, given a set $\fM\subseteq \lineN^{\{0,\dots, n\}}=\lineN^{n+1}$, we see that
$$(\mathbb{P}^n,\cM)(\Z)=\{(a_0:\dots: a_n)\in \mathbb{P}^n(\Z)\mid (v_p(a_0),\dots, v_p(a_n))\in \fM\text{ for all prime numbers } p\}.$$
So in particular:
\begin{itemize}
	\item If $\fM=\{0,1\}^{n+1}$, then
	$$(\mathbb{P}^n,\cM)(\Z)=\{(a_0:\dots: a_n)\in \mathbb{P}^n(\Z)\mid a_0,\dots, a_n\text{ squarefree}\}.$$
	\item If $\fM=\bigcup_{i=0}^n\{\mathbf{w}\in \lineN^{n+1}\mid w_{j}=0\, \forall j\neq i\}$, then
	$$(\mathbb{P}^n,\cM)(\Z)=\{(a_0:\dots: a_n)\in \mathbb{P}^n(\Z)\mid \gcd(a_i,a_j)=1\,\forall i,j, \, i\neq j\}.$$
	\item The set of Campana points for the multiplicities $m_0,\dots,m_n$ is 
	$$(\mathbb{P}^n,\cM)(\Z)=\{(a_0:\dots: a_n)\in \mathbb{P}^n(\Z)\mid a_i\text{ is } m_i\text{-full}\}.$$
	Here we recall that for an integer $m\geq 1$, we say that an integer $n$ is $m$-full if $p|n$ implies $p^m|n$ for every prime number $p$. Furthermore, we define $-1,0$ and $1$ to be the only $\infty$-full numbers.
	\item The set of Darmon points for the multiplicities $m_0,\dots,m_n$ is 
	$$(\mathbb{P}^n,\cM)(\Z)=\{(a_0:\dots: a_n)\in \mathbb{P}^n(\Z)\mid \;|a_i| \text{ is an } m_i\text{-th power}\}.$$
        \item If $\fM=\{\mathbf{w}\in \lineN^{n+1} \mid w_i\leq w_j \text{ if } i\leq j\},$
        then $$(\mathbb{P}^n,\cM)(\Z)=\{(a_0:\dots: a_n)\in \mathbb{P}^n(\Z)\mid \; a_i \text{ divides } a_j \text{ if }i\leq j\}.$$ 
\end{itemize}
\begin{remark}
The above descriptions easily generalize to the case where $B=\spec R$ for a principal ideal domain $R$. If on the other hand $R$ is not a principal ideal domain, then more care is needed since then it is not possible to write every rational point $P$ as $P=(a_0:\dots: a_n)$ such that there is an equality $(a_0,\dots, a_n)=R$ as ideals. For example, if $R=\mathbb{Z}[\sqrt{-5}]$, then the rational point $(2:1+\sqrt{-5})\in \P^1(\mathbb{Q}(\sqrt{-5}))$ cannot be written in such a form.
\end{remark}
For smooth split toric varieties with $D_\alpha$ the torus invariant prime divisors, we will see that we have a very similar concrete description for $\cM$-points, see Remark \ref{remark: description M-points toric variety}.

\subsection{Equivalence of pairs}
From this section onward, $X$ is a variety over a PF field $(K,C)$, $(X,M)$ is a pair and $B\subset C$ is an open subscheme, unless specified otherwise.
It is clear from the definition of $\cM$-points that some pairs $(\cX,\cM)$, $(\cX,\cM')$ have the same set of points or one is contained in the other for geometric reasons. Therefore we introduce the following definition.
\begin{definition}
	Let $(X,M)$ and $(X,M')$ be pairs with integral models $(\cX,\cM)$ and $(\cX',\cM')$ over $B$. Then we say that $(\cX,\cM)$ is \textit{equivalent} to $(\cX',\cM')$ 
	if $$(\cX,\cM)(\cO_v)=(\cX',\cM')(\cO_v) \text{ as subsets of }X(K_v)$$ for all places $v\in B$. If $\cX=\cX'$ we simply say that $\cM$ is equivalent to $\cM'$. Similarly we write $(\cX,\cM)\subseteq (\cX',\cM')$ if $$(\cX,\cM)(\cO_v)\subseteq(\cX',\cM')(\cO_v) \text{ as subsets of }X(K_v)$$ for all places $v\in \Omega_K$. If $\cX=\cX'$ we simply write $\cM\subseteq \cM'$.

    If there exists a nonempty open subset $B\subseteq C$ and equivalent integral models $(\cX,\cM)$ and $(\cX',\cM')$ over $B$, then we say that $M$ is equivalent to $M'$. Similarly if there exists a nonempty open subset $B\subseteq C$ and integral models $(\cX,\cM)\subseteq(\cX',\cM')$, then we write $M\subseteq M'$.
\end{definition}

The next proposition implies that any two integral models $(\cX,\cM)$ and $(\cX,\cM')$ of $(X,M)$ over $B$ become equivalent over some nonempty open subset $B'\subseteq B$.
\begin{proposition} \label{prop: spread out models}
	Let $(X,M)$ be a pair with integral models $(\cX,\cM)$ and $(\cX', \cM')$ over $B$. Then there is a nonempty open subset $B'\subseteq B$ such that there is an isomorphism $f\colon \cX_{B'}\rightarrow \cX'_{B'}$ such that $\cD_{\alpha, B'}$ maps isomorphically to $\cD'_{\alpha, B'}$ for all $\alpha\in \cA$. In particular, $(\cX, \cM)_{B'}$ is equivalent to $(\cX', \cM')_{B'}$.
\end{proposition}
\begin{proof}
As we can restrict to an open subset of $B$, we can without loss of generality assume that $\cX$ and $\cX'$ are flat over $B$.
The proof works via spreading out, and is analogous to \cite[Theorem 3.2.1(iii)]{Poo17}. Here we use the fact that $\cX$ and $\cX'$ are finitely presented over $B$ and therefore $\cD_{\alpha}$ and $\cD_{\alpha}'$ are as well. 
\cite[Theorem 3.2.1(iii)]{Poo17} implies that the identity $\id_X$ lifts to a morphism $f\colon\cX_{B'}\rightarrow \cX'_{B'}$ and a morphism $g\colon \cX'_{B'}\rightarrow \cX_{B'}$. By the same reasoning, we can take $B'\subseteq B$ a small enough open such that, for all $\alpha\in \cA$, we have $f(\cD_{\alpha,B'})\subset \cD_{\alpha,B'}'$ and $g(\cD_{\alpha,B'}')\subset \cD_{\alpha,B'}$ as schemes. As $g\circ f\colon \cX_{B'}\rightarrow \cX_{B'}$ and $f\circ g\colon \cX_{B'}'\rightarrow \cX_{B'}'$ restrict to the identity on $X_K$, it follows from \cite[Th\'{e}or\`{e}me 8.10.5(i)]{EGA-IV3} that there exists a nonempty open subset $B''\subset B'$ such that $f_{B''}$ and $g_{B''}$ are isomorphisms and are inverses of each other and therefore also identify $\cD_{\alpha, B''}$ and $\cD_{\alpha, B''}'$ under the isomorphism.
\end{proof}

It can sometimes be convenient to remove the elements in $\fM$ which correspond to empty intersections of the boundary components $D_{\alpha}$ in $X$, which is why we make the following definition.
\begin{definition}
	Let $(X,M)$ be a pair. Consider the subset $\fM_{\red}\subseteq \fM$ defined by
	$$\fM_{\red}=\{\mathbf{m} \in \fM\mid \cap_{m_{\alpha}\neq 0} D_{\alpha}\neq \emptyset\}.$$
	We write $M_{\red}=((D_{\alpha})_{\alpha\in \cA}, \fM_{\red})$. If $\fM=\fM_{\red}$, then we call $\fM$ \textit{reduced}.
\end{definition}
\begin{proposition} \label{prop: M red equivalent}
	If $(X,M)$ is a pair, then $M$ and $M_{\red}$ are equivalent.
\end{proposition}
\begin{proof}
	Let $\cX$ be an integral model of $X$ over an open subset $B\subseteq C$. Then for all subsets $V\subseteq \cA$ with $\cap_{\alpha\in V} D_{\alpha}=\emptyset$, there exists a nonempty open subset $B_V\subseteq B$ such that $\cap_{\alpha\in V} \cD^c_{\alpha, B_V}=\emptyset$. The intersection $B'=\cap B_V$ over all such $V$ is a nonempty open subset of $B$, since $\cA$ is finite. In particular, we see that for any integral model $(\cX,\cM)$ of $(X,M)$ over $B$, $(\cX,\cM)_{B'}$ is equivalent to  $(\cX,\cM_{\red})_{B'}$.
\end{proof}

\begin{example}
To illustrate the difference between $M$ and $M_{\red}$, we take $X=\P^1_{\mathbb{Q}}$ with integral model $\cX=\P^1_{\Z}$ and consider the disjoint divisors $D_1=\{X_0=0\}$ and $D_2=\{X_0-2X_1=0\}$ on $\P^1_{\mathbb{Q}}$ defined by the homogeneous ideals $(X_0)$ and $(X_0-2X_1)$. Then $\cD_1^c$ and $\cD_2^c$ are the divisors defined by the same ideals and $\cD_1^c\cap \cD_2^c$ is the subscheme defined by the homogeneous ideal $(2X_1,X_0)$. Note that $(2X_1,X_0)$ and $(2,X_0)$ define the same subscheme of $\P^1_{\Z}$ since $X_0$ and $X_1$ cannot vanish simultaneously. We define $$M=((D_1,D_2), \lineN^2),$$ so $$M_{\red}=((D_1,D_2),\{0\}\times \lineN\cup \lineN\times\{0\}).$$ Then $(\cX,\cM^c)(\Z)=\P^1(\mathbb{Q})$, while $$(\cX,(\cM_{\red})^c)(\Z)=\{(x:y)\in \P^1(\mathbb{Q})\mid x,y\in \Z,\, \gcd(x,y)=1,\, x\text { odd}\}$$ since $\cD^c_2\times_{\Z} \spec \mathbb{F}_2=\cD^c_1\times_{\Z} \spec \mathbb{F}_2=\{X_0=0\}\subset \P^1_{\mathbb{F}_2}$, where $\mathbb{F}_2$ is the field with two elements. However, if we invert $2$, then the two pairs become equivalent:
$(\cX,\cM^c)(\Z[\tfrac{1}{2}])=(\cX,(\cM_{\red})^c)(\Z[\tfrac{1}{2}])=\P^1(\mathbb{Q})$.
\end{example}
\begin{example}
Let $X$ be a curve with disjoint divisors $D_{\alpha}$ and multiplicities $m_{\alpha}\in \mathbb{N}^*\cup\{\infty\}$ for $\alpha\in \cA$, and consider the pairs $(X,M)$ and $(X,M')$ associated to Campana points and weak Campana points on $(X,D_{\mathbf{m}})$ as in Definition \ref{def: Campana points and Darmon points}. Then by Proposition \ref{prop: M red equivalent} the pairs $(X,M)$ and $(X,M')$ are equivalent as $M'_{\red}=M_{\red}$. In particular, for every choice of integral models $(\cX,\cM)$ and $(\cX,\cM')$ there exists an open subset $B\subset C$ such that $(\cX,\cM)(B')=(\cX,\cM')(B')$ whenever $B'\subset B$ is an open subset.
\end{example}

\subsection{Inverse image of a pair}
If we have a pair $(X,M)$ and a morphism $f\colon Y\rightarrow X$, we often want to to pull back the structure on $X$ to $Y$ to get a pair $(Y, f^{-1}M)$ such that, given a lift $f\colon \cY\rightarrow \cX$ between integral models, there is an equality $$f^{-1}((\cX,\cM)(B))=(\cY,f^{-1}\cM)(B)$$ of subsets of $\cY(B)$. 
\begin{definition} \label{def: inverse image pair}
	Let $f\colon \cY\rightarrow\cX$ be a morphism of schemes over a scheme $B$ and let $(\cX,\cM)$ be a pair over $B$. We define the inverse image of $(\cX,\cM)$ under $f$ to be the pair $(\cY,f^{-1}\cM)$, where
	$$f^{-1}\cM=((f^{-1}\cD_{\alpha})_{\alpha\in \cA},\fM),$$
    and $f^{-1}\cD_{\alpha}:=\cD_{\alpha}\times_\cX \cY$.

\end{definition}
\begin{proposition} \label{prop: inverse pair}
    Let $(K,C)$ be a PF field, let $(\cX,\cM)$ be a pair over $B$ and let $f\colon \cY\rightarrow \cX$ be a morphism of schemes over an open subscheme $B\subset C$, where $\cY$ is an integral model over $B$ of a variety $Y$ over $K$.
 Then for all closed points $v\in B$,
	$$(\cY,f^{-1}\cM)(\cO_v)=\{P\in Y(K_v)\mid f(P)\in (\cX,\cM)(\cO_v)\},$$
	and therefore
	$$(\cY,f^{-1}\cM)(B)=\{P\in Y(K)\mid f(P)\in (\cX,\cM)(B)\}.$$
\end{proposition}
\begin{proof}
	Let $v\in \Omega_K$, let $\cP_v\in \cY(\cO_v)$ and let $\alpha\in \cA$.
	Since every square in the diagram
	$$\begin{tikzcd}
		\cP_v\cap f^{-1}\cD_{\alpha} \arrow[d] \arrow[r] & f^{-1}\cD_{\alpha} \arrow[d] \arrow[r] & \cD_{\alpha} \arrow[d] \\
		\spec \cO_v \arrow[r, "\cP_v"]  & \cY \arrow[r, "f"] & \cX  
	\end{tikzcd}$$
	is Cartesian, the fiber product of $f\circ \cP_v$ and $\cD_{\alpha}\rightarrow \cX$ is the the same as the fiber product of $\cP_v$ and $f^{-1}\cD_{\alpha}\rightarrow \cX$. Therefore, $n_v(\cD_{\alpha},f\circ \cP_v)=n_v(f^{-1}\cD_{\alpha},\cP_v)$, and thus the desired equalities hold.
\end{proof}

Note that taking the induced integral model does not need to commute with taking the inverse image since we can have
$$(\cY, (f^{-1}\cM)^c)(B)\subsetneq (\cY,f^{-1}(\cM^c))(B),$$
as the next example shows.
\begin{example}
	Let $K=k(t)$, for $k$ a field, and let $X=Y=\P^1_K$, $f=\id_X$ and $D=\{(0:1)\}\in X$. Choose the integral model $\cX=\P^1_k\times_k \mathbb{A}^1_k$ of $X$ over $\mathbb{A}_k^1$, let $P=((0:1),0)\in \cX$, and let $\cY=\Bl_P \cX$ be the blowup of $\cX$ in $P$. Then $\cY$ is an integral model of $Y$ and $f$ lifts to the blowup morphism $f\colon \cY\rightarrow \cX$. Furthermore $(f^{-1}\cD)^c$ is the strict transform of $\cD^c$, which is strictly contained in the inverse image:
	$f^{-1}(\cD^c)=(f^{-1}\cD)^c\cup f^{-1}(P)$. If we restrict the models to the open subset $B=\mathbb{G}_{m,k}\subset \mathbb{A}^1_k$, this discrepancy disappears and we find that $\cX_B\cong \cY_B$.
\end{example}

Note that if $f\colon Y\rightarrow X$ is a dominant map of integral varieties and the closed subschemes $D_\alpha$ are Cartier divisors, the schemes $f^{-1}D_{\alpha}$ are Cartier divisors \cite[Tag 02OO]{Stacks}, but usually not prime divisors. It can be convenient to have a pair equivalent to $(Y,f^{-1}M)$ where the chosen closed subschemes are prime divisors, so we introduce the following notion.
\begin{definition}
Let $f\colon \cY\rightarrow \cX$ be a morphism of schemes over a scheme $B$ and let $(\cX,\cM)$ be a pair over $B$. Assume furthermore that $f^{-1}\cD_{\alpha}$ is a sum of prime Cartier divisors for all $\alpha\in\cA$. Then we define the \textit{pullback of $(\cX,\cM)$ under $f$} to be the pair $(\cY,f^* \cM)$, where $f^* \cM=(\{\widetilde{D}_{\beta}\}_{\beta\in \cB}, f^* \fM)$. The $\widetilde{\cD}_{\beta}$ are the prime divisors on $\cY$ contained in $f^{-1}\cD_{\alpha}$ for some $\alpha\in \cA$, without repetitions. We define
$$f^*\fM=\left\{{w}'\in \lineN^\cB \Biggm| \left(\sum_{\beta\in \cB} c_{\alpha,\beta}w'_\beta\right)_{\alpha\in \cA}\in \fM\right\},$$
where the $c_{\alpha,\beta}$ are given by
$$f^{-1} \cD_{\alpha}=\sum_{\beta\in \cB} c_{\alpha,\beta} \widetilde{\cD}_{\beta}.$$
\end{definition}
Note that this definition is unique up to changing the indexing on the $\widetilde{\cD}_\beta$. As a consequence of Proposition \ref{prop: additivity}, $(\cY,f^*\cM)$ is equivalent to $(\cY,f^{-1}\cM)$. 




\section{Adelic \texorpdfstring{$M$}{M}-points and \texorpdfstring{$M$}{M}-approximation} \label{section: Adelic M-points and approx}
\subsection{Adelic \texorpdfstring{$M$}{M}-points}
In this section we introduce adelic $M$-points and integral adelic $\cM$-points.
Recall from Section \ref{section: PF fields} that the adele ring of a PF field $(K,C)$ prime to a finite set of places $T\subset \Omega_K$ is the restricted product
$$\mathbf{A}_K^T=\prod_{v\in \Omega_K\setminus T}(K_v, \cO_v),$$
which is given the structure of a topological $K$-algebra as in Definition \ref{def: ring of adeles}. Using the restricted product, we generalize the notion of adelic points on a variety $X$ to adelic $M$-points on a pair $(X,M)$.
\begin{definition} \label{def: adelic M-points}
    Let $(K,C)$ be a PF field, let $T\subset \Omega_K$ be a finite sets of places, and let $B\subset C$ be an open subscheme. Let $(X,M)$ be a pair over $(K,C)$ with integral model $(\cX,\cM)$ over $B$. We define the space of \textit{integral adelic $\cM$-points over $B$ prime to $T$} to be the product $$(\cX,\cM)(\mathbf{A}_B^T)=\prod_{v \in \Omega_K\setminus T} (\cX, \cM)(\cO_v)$$
    with the product topology.
    The space of \textit{adelic $M$-points over $B$ prime to $T$} is defined as the restricted product
    $$(X,M)(\mathbf{A}_K^T)=\prod_{v \in \Omega_K\setminus T} ((X,M)(K_v),(\cX, \cM)(\cO_v)).$$
\end{definition}
In Appendix \ref{appendix} we will generalize these notions further in order to compare them to the adelic points considered in \cite{MiNaSt22}.
While the space of integral adelic $\cM$-points depends on the choice of an integral model, even as a set, the space of adelic $M$-points does not depend on such a choice.
\begin{proposition} \label{prop: independence model}
	If $(X,M)$ is a pair over $(K,C)$ with integral models $(\cX,\cM)$ and $(\cX',\cM')$ over $B\subset C$, and $T\subset \Omega_K$ is a finite set of places, then there is a canonical homeomorphism
        $$\prod_{v \in \Omega_K\setminus T} ((X,M)(K_v),(\cX, \cM)(\cO_v))\rightarrow \prod_{v \in \Omega_K\setminus T} ((X,M)(K_v),(\cX', \cM')(\cO_v)).$$
\end{proposition}
\begin{proof}
	By Proposition \ref{prop: spread out models} there exists a nonempty open $B'\subseteq B$ over which $(\cX,\cM)$ and $(\cX',\cM')$ are equivalent. Denote by $S'$ the set of places in $B\setminus (B'\cup T)$. 
	By the properties of the restricted product as recalled in Section \ref{subsection: restricted products and adeles},
	\begin{align*}
	&\prod_{v \in \Omega_K\setminus T} ((X,M)(K_v),(\cX, \cM)(\cO_v)) \\ &\cong\prod_{v\in S'} (X,M)(K_v) \times \prod_{v\in \Omega_{K}\setminus (T\cup S')} ((X,M)(K_v), (\cX, \cM)(\cO_v))
	\\& \cong\prod_{v \in \Omega_K\setminus T} ((X,M)(K_v),(\cX', \cM')(\cO_v))
	\end{align*}
\end{proof}
In particular, $(X,M)(\mathbf{A}_K^T)$ is well-defined for any pair $(X,M)$, because every pair has an integral model.
\begin{example}
    If $\fM=\mathbb{N}^{\cA}$, then $$(\cX,\cM)(\mathbf{A}_B^T)=(X,M)(\mathbf{A}_K^T)=\prod_{v\in \Omega_K\setminus T} U(K_v),$$
    where $U=X\setminus \bigcup_{\alpha\in \cA} D_\alpha$.
\end{example}
\begin{example}
    If $\fM=\{(0,\dots,0)\}$, then the space of integral adelic points on $\cU$ (defined in \cite[page 2]{LoSa16} as $S$-adelic points) is the space of integral adelic $\cM$-points on $\cX$: 
    $$\cU(\mathbf{A}_B^T):=(\cX,\cM)(\mathbf{A}_B^T)=\prod_{v\in \Omega_K\setminus S\cup T}  \mathcal{U}(\mathcal{O}_v)\times \prod_{v\in S\setminus T} U(K_v),$$
    where we write $S=\Omega_K\setminus B$, $U=X\setminus \bigcup_{\alpha\in \cA} D_\alpha$ and $\cU=\cX\setminus \bigcup_{\alpha\in \cA} \cD_\alpha$.
    The space of adelic points on $U$ as in \cite[page 2]{LoSa16} is the space of adelic $M$-points:
    $$U(\mathbf{A}_K^T):=(X,M)(\mathbf{A}_K^T)=\prod_{v\in \Omega_K\setminus S\cup T} \left(U(K_v), \mathcal{U}(\mathcal{O}_v)\right)\times \prod_{v\in S\setminus T} U(K_v).$$
\end{example}

Integral adelic $\cM$-points also generalize the notion of adelic Campana points given in \cite{PSTV20,NaSt20}: if $\cM$ encodes the Campana condition for a divisor $\cD_{\mathbf{m}}$ as in Definition \ref{def: Campana points and Darmon points}, then we have an equality of topological spaces
$$(\cX,\cM)(\mathbf{A}_B^T)= (\cX,\cD_{\mathbf{m}})(\mathbf{A}_K^T),$$
where the right hand side is defined as in \cite[Section 3.2]{PSTV20} and \cite[Definition 2.4]{NaSt20}.

Given inclusions $(X,M)\subseteq (X,M')$, the natural injection $(X,M)(\mathbf{A}_K^T)\hookrightarrow (X,M')(\mathbf{A}_K^T)$ is continuous but it need not be a topological embedding. This is because generally the restricted product topology is strictly finer than the subspace topology on the product. This map does have dense image, as the next proposition shows.
\begin{proposition} \label{prop: inclusion M}
	Let $(X,M)\subseteq (X,M')$ be pairs over $(K,C)$ and let $(\cX, \cM)$ be an integral model of the former pair over an open subscheme $B\subset C$, and let $T\subset \Omega_K$ be a finite set of places. Then:
	\begin{enumerate}
		\item The natural inclusion $(X,M)(\mathbf{A}_K^T)\hookrightarrow (X,M')(\mathbf{A}_K^T)$ is continuous. Furthermore, it has dense image if for all $v\in \Omega_K\setminus T$, the subset $(X,M)(K_v)\subset (X,M')(K_v)$ is dense and $(X,M)(\mathbf{A}_K^T)\neq \emptyset$. The former assumption is automatic if $X$ is smooth and none of the $D_\alpha$ contain irreducible components of $X$.
		\item The natural inclusion $(\cX,\cM)(\mathbf{A}_B^T)\hookrightarrow (X,M)(\mathbf{A}_K^T)$ is a topological embedding. Furthermore, if $(\cX,\cM)(\cO_v)\subset (X,M)(K_v)$ is open for all $v\in B$, then the inclusion is an open embedding.
	\end{enumerate}
\end{proposition}
\begin{proof}
    These statements follow from general properties of restricted products and are a special case of Proposition \ref{prop: inclusions restricted product}. If $X$ is connected and none of the $D_\alpha$ contain irreducible components of $X$, then $(X,M_{\fin})(K_v)\subset X(K_v)$ is dense by Proposition \ref{prop: Open subset dense}, proved below.
\end{proof}

\begin{proposition} \label{prop: Open subset dense}
    Let $X$ be a connected smooth variety over a PF field $(K,C)$ and let $U\subset X$ be a nonempty open subset. Then $U(K_v)\subset X(K_v)$ is dense for all places $v\in \Omega_K$. Therefore any nonempty analytic open $V\subset X(K_v)$ is Zariski dense.
\end{proposition}
\begin{proof}
Consider the complement $Z=X\setminus U$. Since $X$ is smooth, \cite[Lemma 5.3]{Con12} allows us to reduce to when $X=\mathbb{A}^n_{K}$ and $Z$ is contained in the zero locus of a single nonzero polynomial $P$. If $X(K_v)=K_v^n$ contained a nonempty open $V$ such that $V\subset Z(K_v)$, then $P$ vanishes identically on $V$. For any point $u\in V$ and a line $L\subset K_v^n$ through $x$, $P$ vanishes identically on $L$ since $L\cap V$ is infinite and any nonzero univariate polynomial only has finitely many zeroes. Since though any two points in $K_v^n$ there exists a line through these points, this implies that $P$ vanishes on $K_v^n$ and therefore is the zero polynomial, which is a contradiction. Thus $Z(K_v)$ contains no nonempty open sets in $X(K_v)$, so $U(K_v)\subset X(K_v)$ is dense.

Suppose $X$ is connected and $V\subset X(K_v)$ is an analytic open. By the previous part, $V$ cannot lie in a proper closed subscheme of $X$, and therefore $V$ is Zariski dense in $X$.
\end{proof}
\begin{example}
If $M\subseteq M'$, then Proposition \ref{prop: inclusion M} implies that the natural inclusion $(X,M)(\mathbf{A}_K^T)\rightarrow (X,M')(\mathbf{A}_K^T)$ has dense image. The analogous statement for integral adelic $\cM$-points is not always true, however. Consider $\cX= \P^1_\mathbb{Z}$, $M=((1:0), \{0\})$ and $M'=((1:0), \{0,\infty\})$. Then the set of points in $(\cX,\cM^c)(\Z_p)=\{(a:1)\mid a\in \mathbb{Z}_p\}$ is a closed subset of $\P^1_{\Z}(\Z_p)$ which is strictly smaller than $(\cX,\cM'^c)(\Z_p)=\{(a:1)\mid a\in \mathbb{Z}_p\}\cup \{(1:0)\}$.
\end{example}

\subsection{\texorpdfstring{$M$}{M}-approximation}
\label{section: M-approximation}
We now generalize the notion of strong approximation to $\cM$-points.
\begin{definition} \label{Def: M-approx}
	Let $(K,C)$ be a PF-field, let $T\subset \Omega_K$ be a finite set of places, let $(X,M)$ be a pair over $(K,C)$ with integral model $(\cX,\cM)$ over $B\subset C$. Then we say that $(\cX,\cM)$ satisfies \textit{integral $\cM$-approximation off $T$} if
	$$(\cX,\cM)(B)\hookrightarrow (\cX,\cM)(\mathbf{A}_B^T)$$ has dense image, and we say that $(X,M)$ satisfies \textit{$M$-approximation off $T$} if
	$$(X,M)(K)\hookrightarrow (X,M)(\mathbf{A}_K^T)$$ has dense image. We say that $(X,M)$ satisfies $M$-approximation if it satisfies $M$-approximation off $T=\emptyset$.
\end{definition}
Note that these notions only depend on the equivalence classes of $(\cX,\cM)$ and $(X,M)$.

If we take $(X,M)$ to encode the integrality condition of an open subset $U\subset X$, then we recover strong approximation as in \cite[\S 2.6.4.5]{Poo17}.
\begin{definition}
The open subset $U\subset X$ satisfies \textit{strong approximation off $T$} if
$(X,M)$ satisfies $M$-approximation off $T$ with $\fM=\{(0,\dots, 0)\}$, where $U=X\setminus \bigcup_{\alpha\in\cA} D_{\alpha}$. Explicitly this says that
$$U(K)\hookrightarrow U(\mathbf{A}_K^T)$$
has dense image. If $\cM=0$, then we also say that $X=U$ satisfies \textit{weak approximation off $T$}.
\end{definition}

\begin{example}
If $\cM$ is the Campana condition for $\cD_{\mathbf{m}}$ as defined in Definition \ref{def: Campana points and Darmon points}, then integral $\cM$-approximation for $(\cX,\cM)$ coincides with weak Campana approximation for $(\cX,\cD_{\mathbf{m}})$ as studied in \cite{NaSt20}. 
\end{example}

Now we will relate the different notions of approximation to each other. The next proposition shows that integral $\cM$-approximation is equivalent to integral $\cM_\fin$-approximation, provided that the $\cM$-points in the boundary lie in the closure of the set of $\cM_\fin$-points. 
\begin{proposition} \label{prop: relation M and Mfin}
    Let $T\subset \Omega_K$ be a finite set of places, let $(X,M)$ be a pair over $(K,C)$ with integral model $(\cX,\cM)$ over $B\subset C$.
    \begin{enumerate}
        \item If $(X,M)$ satisfies $M$-approximation off $T$, then $(X,M_\fin)$ satisfies $M_\fin$-approximation off $T$. If for every place $v\in \Omega_K\setminus T$, the subset $(X,M_\fin)(K_v)\subset (X,M)(K_v)$ is dense, then the converse also holds.
        \item If $(\cX,\cM)$ satisfies integral $\cM$-approximation off $T$, then $(\cX,\cM_\fin)$ satisfies $\cM_\fin$-approximation off $T$. If for every place $v\in \Omega_K\setminus T$, the subset $(\cX,\cM_\fin)(\cO_v)\subset (\cX,\cM)(\cO_v)$ is dense, then the converse also holds.
    \end{enumerate}
\end{proposition}
\begin{proof}
    The first statement follows from the fact that for any place $v\in \Omega_K\setminus T$, $(X,M_\fin)(K_v)$ is open in $(X,M)(K_v)$. The converse statement follows from the fact that $(X,M_\fin)(\mathbf{A}_K^T)\rightarrow (X,M)(\mathbf{A}_K^T)$ has dense image by Proposition \ref{prop: inclusion M}.
    Now we prove the statements for integral $\cM$-approximation. Assume that $(\cX,\cM)$ satisfies integral $\cM$-approximation off $T$. Then in particular, for any place $v\in \Omega_K\setminus T$, $(\cX,\cM_{\fin})(B)$ is dense in $(\cX,\cM_\fin)(\cO_{v})\times (\cX,\cM)(\mathbf{A}_B^{T\sqcup \{v\}})$, since $(\cX,\cM_\fin)(\cO_{v})$ is open in $(\cX,\cM)(\cO_{v})$ by Proposition \ref{prop: continuity multiplicity}. Therefore every open $\prod_{v\in \Omega_K\setminus T} U_v\subset (\cX,\cM_{\fin})(\mathbf{A}_B^{T\sqcup \{v\}})$ contains an element in $(\cX,\cM_{\fin})(B)$, so $(\cX,\cM_{\fin})$ satisfies integral $\cM_{\fin}$-approximation off $T$.
    If $(\cX,\cM_\fin)(\cO_v)$ is a dense subset of $(\cX,\cM)(\cO_v)$ for all $v\in \Omega_K\setminus T$, then the inclusion map $(\cX,\cM_\fin)(\mathbf{A}_B^T)\rightarrow (\cX,\cM)(\mathbf{A}_B^T)$ has dense image. Therefore, if $(\cX,\cM_\fin)$ satisfies integral $\cM_\fin$-approximation off $T$, then $(\cX,\cM)$ satisfies integral $\cM$-approximation off $T$.
\end{proof}

Using the Proposition \ref{prop: relation M and Mfin}, we are able to relate integral $\cM$-approximation and $M$-approximation.
\begin{proposition} \label{prop: inclusion M approx}
Let $(X,M)$ be a pair over $(K,C)$ with an integral model $(\cX,\cM)$ over $B\subset C$. Let $T\subset\Omega_K\setminus B$ be a set of places disjoint from the points of $B$. If $(X,M)$ satisfies $M$-approximation off $T$, then $(\cX,\cM_{\fin})$ satisfies integral $\cM_{\fin}$-approximation off $T$.
\end{proposition}
\begin{proof}
By Proposition \ref{prop: relation M and Mfin}, $(X,M_\fin)$ satisfies $M_\fin$-approximation off $T$. By Proposition \ref{prop: inclusion M}, the inclusion  $(\cX,\cM_\fin)(\mathbf{A}_B^T)\hookrightarrow (X,M_\fin)(\mathbf{A}_K^T)$ is an open embedding, and $(X,M_\fin)(K)\cap (\cX,\cM_\fin)(\mathbf{A}_B^T)=(\cX,\cM_\fin)(B)$. Since $(X,M_\fin)(K)$ is dense in $(X,M_\fin)(\mathbf{A}_K^T)$, $(\cX,\cM_\fin)(B)$ is dense in $(\cX,\cM_\fin)(\mathbf{A}_B^T)$ as the intersection of a dense set with an open subset is dense in the open subset.
\end{proof}

The following proposition is an analog of \cite[Proposition 3.22]{MiNaSt22} (and a generalization of \cite[Lemma 2.8]{NaSt20}). It is a partial converse to Proposition \ref{prop: inclusion M approx} and that states if integral $\cM$-approximation on $(\cX,\cM)$ is preserved under restricting the base $B$, then $(X,M)$ satisfies $M$-approximation.
\begin{proposition}
    Let $(X,M)$ be a pair over $(K,C)$ with integral model $(\cX,\cM)$ over $B\subset C$ and let $T\subset \Omega_K$ be a finite set of places. If for every nonempty open subscheme $B'\subset B$, $(\cX,\cM)_{B'}$ satisfies integral $\cM_{B'}$-approximation off $T$, then $(X,M)$ satisfies $M$-approximation off $T$. If furthermore, $\fM\subset \lineN^\cA$ is open, then the converse also holds.
\end{proposition}
\begin{proof}
    We first assume that $(\cX,\cM)_{B'}$ satisfies integral $\cM_{B'}$-approximation off $T$ for all $B'\subset B$.
    Note that every $P\in (X,M)(\mathbf{A}_K^T)$ lies in the subspace $(\cX,\cM)_{B'}(\mathbf{A}_{B'}^T)$ for some $B'\subset B$. Thus since every open neighbourhood of $P$ in $(\cX,\cM)_{B'}(\mathbf{A}_{B'}^T)$ has nonempty intersection with $(\cX,\cM)_{B'}(B')$, the same holds for any open neighbourhood of $P$ in $(X,M)(\mathbf{A}_K^T)$. Hence $(X,M)$ satisfies $M$-approximation off $T$. 

    In the other direction, note that if $\fM\subset \lineN^\cA$ is open, then $(\cX,\cM)(\cO_v)\subset (X,M)(K_v)$ is open by Proposition \ref{prop: continuity multiplicity}. Therefore, by Proposition \ref{prop: inclusion M}, $(\cX,\cM)_{B'}(\mathbf{A}_{B'}^T)$ is open in $(X,M)(\mathbf{A}_K^T)$ so $(X,M)(K)\cap (\cX,\cM)_{B'}(\mathbf{A}_{B'}^T)=(\cX,\cM)_{B'}(B')$ is dense in $(\cX,\cM)_{B'}(\mathbf{A}_{B'}^T)$.
\end{proof}

Using the concept of an inverse image of a pair introduced in Definition \ref{def: inverse image pair}, we can use birational morphisms, such as a resolution of singularities, to understand $M$-approximation.

\begin{proposition} \label{prop: birational morphisms M-approximation}
Let $(K,C)$ be a PF field, let $T\subset \Omega_K$ be a finite set of places and let $f\colon Y\rightarrow X$ be a birational morphism of proper integral $K$-varieties, where $Y$ is smooth over $K$. 
Suppose that the induced map $Y(K_v)\rightarrow X(K_v)$ is surjective for all places $v\in \Omega_{K}\setminus T$.
Then any pair $(X,M)$ satisfies $M$-approximation off $T$ if and only if $(Y,f^{-1}M)$ satisfies $f^{-1}M$-approximation off $T$.

If $B\subset C$ is an open subscheme and $f\colon Y\rightarrow X$ spreads out to a morphism $f\colon \cY\rightarrow\cX$ of integral models, then any pair $(\cX,\cM)$ satisfies integral $\cM$-approximation off $T$ if and only if $(\cY,f^{-1}\cM)$ satisfies integral $f^{-1}\cM$-approximation off $T$.
\end{proposition}
\begin{proof}
By the assumption, it follows that the induced map $(Y,f^{-1}M)(\mathbf{A}_K^T)\rightarrow (X,M)(\mathbf{A}_K^T)$ is surjective by Proposition \ref{prop: inverse pair}. Therefore, if $(Y,f^{-1}M)$ satisfies $f^{-1}M$-approximation, then $(X,M)$ satisfies $M$-approximation. In the other direction, if $V\subseteq Y$ is the locus over which $f$ is an isomorphism, then $f^{-1}((X,M)(K))\cap V(K)$ is dense in $(Y,f^{-1}M)(\mathbf{A}_K^T)\cap \prod_{v\in \Omega_{K}\setminus T} V(K_v)$. Since $Y$ is smooth, $V(K_v)$ is dense in $Y(K_v)$ by Proposition \ref{prop: Open subset dense}, so $f^{-1}((X,M)(K))\cap V(K)$ is dense in $(Y,f^{-1}M)(\mathbf{A}_K^T)$ for every place $v\in \Omega_K\setminus T$.
\end{proof}
\begin{remark}
    The surjectivity condition in Proposition \ref{prop: birational morphisms M-approximation} is referred to as ``arithmetic surjectivity'' in the literature, see \cite{LSS20}.
\end{remark}

\subsection{Integral \texorpdfstring{$\cM$}{M}-approximation and the \texorpdfstring{$\cM$}{M}-Hilbert property}
Now we will explore the relationship between integral $\cM$-approximation, the $\cM$-Hilbert property and Zariski density. This is an extension of the classical theory of the Hilbert property and weak approximation as presented in \cite[Chapter 3]{Ser16}, and of the Campana version introduced in \cite{NaSt20}.
\begin{definition} \label{def: thin set}
	Let $X$ be an integral variety over $K$ and let $A\subseteq X(K)$. We say that $A$ is of \textit{type I} if $A\subseteq Z(K)$ for some proper closed subset $Z$ of $X$.
	
	We say that $A$ is of \textit{type II} if there is an integral variety $Y$ with $\dim Y=\dim X$ and a generically finite morphism $f\colon Y\rightarrow X$ of degree $\geq 2$ such that $A\subseteq f(Y(K))$.
	
	We say that $A$ is \textit{thin} if it is a finite union of sets of type I and II.
\end{definition}
\begin{remark} \label{remark: different definition thin}
    We do not assume that the morphisms in Definition \ref{def: thin set} are separable, unlike the definitions given in \cite{BFP13,Lug22}. In particular, thin sets as considered in those articles are thin sets as defined here, but not vice versa. For example, if $k$ is a perfect field of characteristic $p>0$, then $k(t^p)\subset \mathbb{A}^1(k(t))$ is a thin set in the terminology of this article, but not according to the notion in \cite{BFP13,Lug22}.
\end{remark}
\begin{definition}
	Let $(\cX,\cM)$ be an integral model over $B$ of a pair $(X,M)$. We say that $(\cX,\cM)$ satisfies the \textit{$\cM$-Hilbert property over $B$} if $(\cX,\cM)(B)$ is not thin as a subset of $X(K)$.
\end{definition}
Note that every PF field $K$ is Hilbertian (see for example \cite[Chapter 13]{FrJa05}) and imperfect if it is of positive characteristic, so by \cite[Proposition 12.4.3]{FrJa05}, $\mathbb{A}^1(K)$ is an example of a set which is not thin.

In order to relate the $\cM$-Hilbert property to integral $\cM$-approximation over global function fields as we do in Theorem \ref{theorem: M-approx implies not thin}, we need the following lemma.
\begin{lemma} \label{lemma: image inseparable nowhere dense}
    Let $k$ be a field of positive characteristic and let $f\colon Y\rightarrow X$ be an inseparable generically finite morphism of integral varieties over $k\lParen t\rParen$, where $X$ is smooth over $k\lParen t\rParen$. Then the image of the induced map $Y(k\lParen t\rParen)\rightarrow X(k\lParen t\rParen)$ is a nowhere dense subset of $X(k\lParen t\rParen)$, where the topologies are induced by the topology on $k\lParen t\rParen$.
\end{lemma}
\begin{proof}
    The proof is based on the observation that $k\lParen t \rParen\rightarrow k\lParen t \rParen$ given by $x\mapsto x^p$ has nowhere dense image in $k\lParen t \rParen$, where $p=\CHAR(k)$. This is because the image is contained in the closed set $\bigcap_{i\in \mathbb{Z}, p\nmid i} C_i$, where $C_i$ is the set of Laurent polynomials with vanishing $i$-th coefficient. Since every nonempty open set contains elements outside of this closed set, the image is nowhere dense in $k\lParen t \rParen$.

    First we show it suffices to prove the statement for an open subvariety $U\subset X$. Since $X$ is smooth, by \cite[\S 2.2, Proposition 11]{BLR90} it is covered by open subvarieties $U\subset X$ such that there exists étale morphism $U\rightarrow\mathbb{A}^d_{k\lParen t \rParen}$, where $d$ is the dimension of $X$. By \cite[Lemma 5.3]{Con12} the induced map $U(k\lParen t \rParen)\rightarrow\mathbb{A}^d_{k\lParen t \rParen}(k\lParen t \rParen)$ is a local homeomorphism. Since for a closed subvariety $V\subset \mathbb{A}^d_{k\lParen t \rParen}$, the subset $V(k\lParen t \rParen)\subset \mathbb{A}^d_{k\lParen t \rParen}(k\lParen t \rParen)$ is nowhere dense, it follows that for any proper closed subvariety $V\subset X$, the subset $V(k\lParen t \rParen)\subset X(k\lParen t \rParen)$ is nowhere dense. Therefore it suffices to prove the statement for some nonempty open subvariety $U\subset X$.

    For a generically finite dominant morphism $f\colon Y\rightarrow X$ of integral varieties, we can factor the extension $K(Y)/K(X)$ as a separable extension $L/K(X)$ followed by a totally inseparable extension $K(Y)/L$, see \cite[Tag 030K]{Stacks}. This corresponds to factoring the morphism into rational maps $Y\dashrightarrow Z\dashrightarrow X$, where $Z\dashrightarrow X$ is separable and $Y\dashrightarrow Z$ is totally inseparable.
    We can assume that $Z$ is not geometrically integral, as otherwise \cite[Tag 0CDW]{Stacks} implies that $Z(k\lParen t \rParen)$ is not Zariski dense in $Z$.
    Since we can restrict $X$ to a nonempty open $U$, without loss of generality, we can assume that the maps are morphisms and $Z\rightarrow X$ is étale. Therefore, by \cite[Lemma 5.3]{Con12} we can without loss of generality assume that $f\colon Y\rightarrow X$ is purely inseparable of degree equal to $p:=\CHAR(k)$.
    
    Choose a separable generically finite rational map $X\dashrightarrow \mathbb{A}^d_{k\lParen t \rParen}$. Then the field extension $k\lParen t \rParen(x_1,\dots,x_d)\subset k\lParen t \rParen(Y)$ induced by this rational map has inseparable degree $p$.
    There are only finitely many fields in $k\lParen t \rParen(Y)$ containing $k\lParen t \rParen(x_1,\dots,x_d)$. Indeed, by Galois theory there are only finitely many separable extensions $L$ of $k\lParen t \rParen(x_1,\dots,x_d)$ in $k\lParen t \rParen(Y)$. Furthermore, for any such a field extension $L$, there is at most a single nontrivial purely inseparable extension $L'/L$ in $k\lParen t \rParen(Y)$, since given another such extension $L''$ the compositum $L''L'/L$ is purely inseparable and therefore has degree $p$, by multiplicativity of inseparable degrees \cite[Tag 09HK]{Stacks}.
    The primitive element theorem \cite[Tag 030N]{Stacks} implies that $k\lParen t \rParen(x_1,\dots,x_d)\subset k\lParen t\rParen(Y)$ is a simple extension. Therefore for some nonzero polynomial $g\in k\lParen t\rParen[y,x_1,\dots, x_d]$ which is separable in $y$,
    $$k\lParen t\rParen(Y)=k\lParen t \rParen(x_1,\dots,x_d)[z]/(g(z^p,x_1,\dots,x_d)),$$
    where the $p$-th power is present since the extension is not separable.
    The separable closure of $k\lParen t \rParen$ in this field is
    $$k\lParen t\rParen(X)=k\lParen t \rParen(x_1,\dots,x_d)[y]/(g(y,x_1,\dots,x_d)).$$

    In order to reduce the argument to the power map $x\mapsto x^p$, we construct separable dominant maps $Y\dashrightarrow \mathbb{A}^d_{k\lParen t \rParen}$ and $X\dashrightarrow \mathbb{A}^d_{k\lParen t \rParen}$. If $\frac{\d g}{\d x_d}$ is not the zero polynomial, then the Jacobi criterion implies that the projections to affine space corresponding to the inclusions $$k\lParen t \rParen(y,x_1,\dots,x_{d-1})\subset k\lParen t \rParen(X)$$ and 
    $$k\lParen t \rParen(z,x_1,\dots,x_{d-1})\subset k\lParen t \rParen(Y)$$ are separable. If $\frac{\d g}{\d x_d}$ is the zero polynomial, we can reduce to the case that $\frac{\d g}{\d x_d}$ is not identically zero, by using the fact that $g$ is separable in $y$ and by using linear change of variables replacing $y$ with $y+x_d$.
    
    Therefore, we get a Cartesian diagram of dominant rational maps
    $$\begin{tikzcd}
    Y \arrow[r, dashed] \arrow[d, "f"] & \mathbb{A}^d_{k\lParen t \rParen} \arrow[d, "h"] \\
    X \arrow[r, dashed]                   & \mathbb{A}^d_{k\lParen t \rParen},        \end{tikzcd}$$
    where the horizontal maps are separable, and the map $h$ is the $p$-th power map in the first coordinate and the identity in the other coordinates. Since we can without loss of generality restrict $X$ to an open subset $U$, we can assume that the horizontal maps are étale morphisms.
    Since the image of $\mathbb{A}^d_{k\lParen t \rParen}(k\lParen t \rParen)\rightarrow \mathbb{A}^d_{k\lParen t \rParen}(k\lParen t \rParen)$ is a nowhere dense subset, applying \cite[Lemma 5.3]{Con12} once more to the horizontal maps shows that the image of $Y(k\lParen t \rParen)$ in $X(k\lParen t \rParen)$ is nowhere dense.
    \end{proof}

Now we will prove Theorem \ref{theorem: M-approx implies not thin}, which we will prove separately for global fields and for function fields over infinite fields.
\begin{proof}[Proof of Theorem \ref{theorem: M-approx implies not thin} for global fields]

First we prove the statements for global fields.
This part of the proof is based on \cite[Theorem 1.1, Remark 2.12]{NaSt20}, and generalizes it to global fields and varieties which are not normal or even integral.

We first assume that $X$ is geometrically irreducible. Note that the statement of the theorem is equivalent to saying that if the $\cM$-Hilbert property over $B$ is not satisfied and $(\cX,\cM)(B)\neq \emptyset$, then integral $\cM$-approximation does not hold for any finite set of places $T$. We proceed by proving the following stronger claim: if $A\subseteq (\cX,\cM)(B)$ is thin, then there is a finite set of places $T'\subset \Omega_K$ disjoint from $T$ such that $A$ is not dense in $\prod_{v\in T'} (\cX,\cM)(\cO_v)$. We recover the original statement by taking $A=(\cX,\cM)(B)$. By the argument in \cite[Proof of Theorem 3.5.3]{Ser16} if the claim holds for thin sets $A_1, A_2\in (\cX,\cM)(B)$, then it holds for $A_1\cup A_2$.

For a set $A$ which is not Zariski dense (thin of type I), the result follows from the Lang--Weil bound \cite{LaWe54} in the same manner as in the proof of \cite[Theorem 1.1]{NaSt20}. In particular, we conclude that $(\cX,\cM)(B)$ is Zariski dense in $X$.

Now let $A\subset f(Y(K))\cap (\cX,\cM)(B)$ for $Y$ an integral variety over $K$ and $f\colon Y\rightarrow X$ a dominant generically finite morphism of degree $\geq 2$. Without loss of generality, we can assume that $f$ is finite, since the image of a Zariski closed subset in $Y$ is a type I thin set. There are two cases to consider: either $f$ is separable or it is inseparable. If $f$ is separable, then the result follows from \cite[Theorem 3.6.2]{Ser16} (which extends to separable morphisms over global fields) and the proof of \cite[Theorem 1.1]{NaSt20}, see also \cite[Remark 2.12]{NaSt20}. This finishes the claim for separable morphisms. For global function fields we need to consider inseparable morphisms.

Thus we now assume $f \colon Y\rightarrow X$ is an inseparable morphism. For every place $v\in \Omega_{K}$, the Cohen structure theorem \cite[Tag 0C0S]{Stacks} implies $\cO_v\cong k_v\lBracket t\rBracket$. Therefore Lemma \ref{lemma: image inseparable nowhere dense} implies that for every place $v\in \Omega_{K}$ the subset $f(Y(K_v))\cap U(K_v)\subset U(K_v)$ is nowhere dense, where $U$ is the smooth locus of $X$. Since $(\cX,\cM_\fin)(\cO_v)$ is nonempty and open in $X(K_v)$ by Proposition \ref{prop: continuity multiplicity}, $A\cap U(K_v)\cap (\cX,\cM_\fin)(\cO_v)\subset (\cX,\cM_\fin)(\cO_v)$ is nowhere dense. Since $X(K_v)\setminus U(K_v)$ is a thin set of type I, its restriction to $(\cX,\cM_\fin)(\cO_v)$ is not dense in $(\cX,\cM_\fin)(\cO_v)$. Thus $A\cap (\cX,\cM_\fin)(\cO_v)\subset (\cX,\cM_\fin)(\cO_v)$ is not dense so $A\subset (\cX,\cM)(\cO_v)$ is not dense. This proves the claim.

Now assume that $X$ is irreducible, but not necessarily geometrically irreducible. We will prove that $(\cX,\cM)(\cO_v)$ contains a smooth point in $X(K_v)$ for some place $v\in \Omega_K$, and then conclude that therefore $X(K)$ contains a smooth point. As $X$ is irreducible, $U=X\setminus \bigcup_{\alpha\in\cA} D_\alpha$ is irreducible as well. Write $\cU=\cX\setminus \bigcup_{\alpha\in\cA} \cD_\alpha$.
If $X$ is not geometrically irreducible, choose a finite separable extension $K'$ of $K$ such that $X_{K'}$ splits into geometrically irreducible components $X_i$. Write $B'$ for the regular scheme associated to the integral closure of $\cO_B$ in $K'$. The splitting also induces a splitting of the integral model $\cX\times_{B} B'$ into components $\cX_i$ which are integral models of the $X_i$. Write $\cU_i=\cX_i\cap (\cU\times_{B} B')$. Write $T'$ for the places in $\Omega_{K'}$ above $T$. As in the previous part, the Lang--Weil bound \cite{LaWe54} implies that, for all but finitely many places $v'\in \Omega_{K'}$, $\cU_i(k_{v'})$ contains a smooth point. In particular, if we choose a finite place $v\in B\setminus T$ which splits completely in $K'$ and such that for a place $v'\in B'$ above $v$, $\cU_i(k_{v'})$ contains a smooth point, then $\cU(k_v)=\cU(k_{v'})$ contains a smooth point as it contains $\cU_i(k_{v'})$. Thus by Hensel's lemma \cite[Theorem 3.5.63]{Poo17} $\cU(\cO_v)$ contains a smooth point. Since $(\cX,\cM)(B)\neq \emptyset$ and $(\cX,\cM)$ satisfies integral $\cM$-approximation off $T$, we see that $X(K)$ contains a smooth point, and thus by \cite[Tag 0CDW]{Stacks}, we see that $X$ is geometrically integral.

If $X$ is not irreducible, then there are disjoint open sets $\cU_1,\cU_2\subset \cU$, where $\cU$ is as above. By the same reasoning as when $X$ was irreducible, the Lang--Weil bound implies that we can find distinct places $v_1,v_2\in \Omega_K\setminus T$ such that $\cU_1(\cO_{v_1})$ and $\cU_2(\cO_{v_2})$ are nonempty. But $P\in (\cX,\cM)(B)$ cannot simultaneously simultaneously lie in the open sets $\cU_1(\cO_{v_1})$ and $\cU_2(\cO_{v_2})$ as $\cU_1$ is disjoint from $\cU_2$. This is a contradiction, so we find that $X$ has to be irreducible.

\end{proof}
\begin{proof}[Proof of Theorem \ref{theorem: M-approx implies not thin} for function fields over infinite fields]
Now assume that $K=k(C)$ for an infinite field $k$ and a regular curve $C$ over $k$.
First we assume that $X$ is irreducible. Let $\cU=\cX\setminus \bigcup_{\alpha\in\cA} \cD_\alpha$ be the closure of $U=X\setminus \bigcup_{\alpha\in\cA} D_\alpha$ in $\cX$. Since $X$ is geometrically reduced, there exists a nonempty open subset $\cY\subseteq \cU$ such that the restriction of the structure morphism $\cX\rightarrow B$ to $\cY$, written $f\colon \cY\rightarrow B$, is smooth.
Note that $B$ is a Jacobson scheme \cite[Tag 01P2]{Stacks}, $f$ is dominant and of finite presentation, and by \cite[Théorème 4.2.3.(1)]{MoBa20} $B$ is a ``schéma de Poonen''. Therefore \cite[Théorème 3.2.(3)]{MoBa20} implies that there exists a point $y\in \cY$ such that $f(y)$ is a closed point and the residue field satisfies $k(y)=k(f(y))$. If we write $v=f(y)$, then this implies that $y\in\cY(k_v)\subset \cU(k_v)$ is a rational $k_v$-point, and thus by Hensel's Lemma \cite[Theorem 3.5.63]{Poo17}, $\cU(\cO_v)$ contains a smooth point $P$. In particular, \cite[Tag 0CDW]{Stacks} implies that $X$ is geometrically integral.

Since $P$ is smooth, there exists an open $V\subset X$ which is smooth over $K_v$. Thus Proposition \ref{prop: Open subset dense} implies $V\cap \cU(\cO_v)$ is Zariski dense as it contains $P$, so $\cU(\cO_v)$ is Zariski dense. Since $(\cX,\cM)(B)$ is dense in $(\cX,\cM)(\cO_v)$ in the analytic topology, it is Zariski dense in $X$.

Now if we assume that $X$ is not irreducible, we find a contradiction in the exact same manner as in the global function field case. So $X$ has to be geometrically integral.
\end{proof}

The next example shows that the set of rational points on a variety that satisfies weak approximation can fail to be Zariski dense if the variety is not geometrically reduced.
\begin{example} \label{example: M-approx not Zariski dense}
    Let $k=\mathbb{F}_p(a,b)$ and $K=k(t)$ for algebraically independent variables $a,b,t$.
    Define $X$ to be the closed subvariety of $\P^3_{k}(x,y,z,w)$ given by $$x^p-z^p a= y^p-z^p b=0.$$ For every place of $v$ of $K$, $K_v$ is a simple extension of $k\lParen t\rParen$. However, $k\lParen t\rParen(a^{1/p}, b^{1/p})$ is a degree $p^2$ extension of $k\lParen t\rParen$ which is not simple. Indeed if it were simple, then there would be a primitive element $\alpha$ such that $k\lParen t\rParen(a^{1/p}, b^{1/p})=k\lParen t\rParen(\alpha)$, but $\alpha^p\in k\lParen t\rParen$ would imply that the extension has degree $p$ rather than $p^2$. Thus we see that $K_v$ cannot contain $k\lParen t\rParen(a^{1/p}, b^{1/p})$. Consequently, for any $K_v$-point on $X$ we must have $z=0$, and therefore
	$$X(K_v)=X(K)=X(k)=\{(0:0:0:1)\}.$$
	This implies that $X$ satisfies weak approximation while $X(K)$ is not Zariski dense.
\end{example}

As a consequence of Theorem \ref{theorem: M-approx implies not thin}, we give a new proof of Minchev's theorem.
\begin{corollary}[{\cite[Theorem 1]{Min89}}] \label{cor: Minchev}
	Let $U$ be a normal variety over a number field $K$, which satisfies strong approximation off a finite set of places $T\subset \Omega_{K}$. Then $U_{\overline{K}}$ is algebraically simply connected, i.e. $\pi_1(U_{\overline{K}})=\{1\}$.
\end{corollary}
\begin{proof}
	Let $X$ be a normal compactification of $U$ and let $\cX$ be a normal integral model of $X$ over $\cO_S$, where $T\subset S$ is some finite set of places. Let $\cD$ be the Zariski closure of $X\setminus U$ in $\cX$ and let $\cU=\cX\setminus \cD$ be its complement. By Proposition \ref{prop: inclusion M approx}, $\cU$ satisfies integral strong approximation off $S$. Therefore, Theorem \ref{theorem: M-approx implies not thin} implies that $X$ is geometrically integral and $\cU(\cO_S)\subset X(K)$ is not thin. Assume that $U_{\overline{K}}$ is not simply connected. Then \cite[Lemma 3.5.57]{Poo17} implies that there exists a nontrivial finite étale morphism $f\colon Y\rightarrow U$ where $Y$ is a geometrically integral variety over $K$. By spreading out \cite[Theorem 3.2.1(ii)]{Poo17}, there exists a finite set of places $S'$ containing $S$ and a finite étale morphism $\cY\rightarrow \cU\times_{\cO_S} \spec \cO_{S'}$ extending $f$, where $\cY$ is a scheme with $\cY_K\cong Y$. This contradicts \cite[Theorem 1.8]{Lug22}, so $U_{\overline{K}}$ has to be simply connected.
\end{proof}

\section{Split toric varieties and \texorpdfstring{$M$}{M}-approximation}
\subsection{Cox coordinates on toric varieties} \label{section: Cox coordinates}
In this section we introduce toric varieties and their Cox coordinates, following and generalizing \cite[Chapter 5.1]{CLS11} and \cite[Chapter 8]{Sal98}. Given any fan $\Sigma$ as in \cite[Definition 8.1.1]{Sal98} (not necessarily complete or regular), we define the toric scheme associated to $\Sigma$ over $\Z$ to be $\cX_{\Sigma,\Z}$ as in \cite[Remark 8.6]{Sal98}. This is a normal, separated scheme over $\Z$ and it is proper (resp. smooth) over $\Z$ if and only if $\Sigma$ is complete (resp. regular). For any scheme $S$, we define the toric scheme associated to $\Sigma$ over $S$ to be $\cX_{\Sigma,S}:=\cX_{\Sigma}\times_{\Z} S$. For the remainder of the section, we let $K$ be a field and write $X_{\Sigma}:=\cX_{\Sigma,K}$ for the normal split toric variety associated to the fan $\Sigma$.

For a fan $\Sigma$, we write $\Sigma(1)$ for the collection of rays in $\Sigma$, and $\{D_1,\dots, D_n\}$ for the set of torus invariant prime divisors on $X:=X_{\Sigma}$. We denote the lattice of cocharacters of $\cX$ by $N$ and its dual by $N^\vee$. For a torus invariant divisor $D_i$, we write $\rho_i\subset N_{\mathbb{R}}$ for the associated ray and $n_{\rho_i}\in N$ for its ray generator.

\begin{definition}
	If $X=X_{\Sigma}$ is a normal split toric variety over a PF field $(K,C)$, then its \textit{toric integral model} over $B\subset C$ is $\cX=\cX_{\Sigma,B}$. If $(X,M)$ is a pair with $\cA=\{1,\dots,n\}$ and $D_1,\dots,D_n$ the torus invariant prime divisors on $X$, then we call $(X,M)$ a \textit{toric pair} and we say that its \textit{toric integral model} over $B$ is $(\cX_{\Sigma,B},\cM^c)$. We denote the open torus in $X$ by $U$.
\end{definition}

In the remainder of this section, we assume that the ray generators $n_{\rho_i}$ span $N_\mathbb{R}$. This is equivalent to the split toric variety $X=X_{\Sigma,K}$ not having torus factors, or equivalently $\cO_X(X_{\overline{K}})^\times=\overline{K}^\times$.

Now we introduce Cox coordinates on the integral points on the toric variety: we write the toric scheme $\cX=\cX_{\Sigma,\Z}$ as a quotient $\cX=\cY/\cG$ for some open subscheme $\cY\subset \mathbb{A}^n_\Z$ and a group scheme $\cG\subset \mathbb{G}_{m,\Z}^n$.
We have an exact sequence
\begin{equation} \label{equation: exact sequence class group}
0\rightarrow N^\vee\rightarrow \Z^{\Sigma(1)}\rightarrow \Cl(\Sigma)\rightarrow 0, 
\end{equation}
as in \cite[Theorem 4.1.3]{CLS11}, where $\Cl(\Sigma)$ is the class group of any toric variety over a field with the fan $\Sigma$.
We define the $\Z$-group scheme $$\cG=\cHom(\Cl(\Sigma),\mathbb{G}_{m,\Z}),$$
which by the above exact sequence is the kernel of the homomorphism
$$\mathbb{G}_{m,\Z}^n\cong\cHom(\Z^{\Sigma(1)},\mathbb{G}_{m,\Z})\rightarrow \mathbb{G}_{m,\Z}^d\cong\cHom(N^\vee,\mathbb{G}_{m,\Z})$$
induced by the inclusion
$N^\vee\hookrightarrow \Z^{\Sigma(1)}$. Here $d$ is the rank of the lattice $N$. In particular, for every ring $R$, we have the description
$$\cG(R)=\left\{\mathbf{t}\in (R^\times)^n \,\middle\vert\, \prod_{i=1}^n t_i^{\langle m, n_{\rho_i} \rangle}=1 \text{ for all } m\in N^\vee\right\},$$
where $\langle \cdot , \cdot \rangle\colon N^\vee\times N\rightarrow \Z$ is the natural pairing. Now for each cone $\sigma\in \Sigma$, let $$x^{\hat{\sigma}}=\prod_{\substack{i=1 \\ \rho_i\not\subseteq \sigma}}^nx_i,$$ and let $$\cZ=\{x^{\hat{\sigma}}=0\mid \text{for all } \sigma\in \Sigma\}\subset \mathbb{A}_\Z^n.$$
Then $\cY=\mathbb{A}_\Z^n\setminus \cZ$ carries the natural structure of a toric scheme, and the subscheme $\cG$ acts on it by coordinate-wise multiplication. Similarly to \cite[Theorem 5.1.9]{CLS11}, we find the \textit{Cox morphism} $\pi \colon\mathbb{A}_\Z^n\setminus \cZ\rightarrow \cX$ of toric schemes, which is constant on $\cG$-orbits and gives a bijection between the closed toric subschemes of $\cY$ and those of $\cX$.
\begin{lemma} \label{lemma: Cox quotient universal}
    Let $\cX$, $\cY$ and $\pi$ be as in the discussion above. The morphism $\pi \colon \cY\rightarrow \cX$ is an universal categorical quotient for the action of $\cG$ on $\mathbb{A}_\Z^n\setminus \cZ$. If furthermore $\Sigma$ is regular and $\Cl(\Sigma)$ is torsion-free, then $\pi$ is a $\cG$-torsor.
\end{lemma}
If $\Sigma$ is regular, the Cox morphism is referred to as the universal torsor of $\cX$, such as in \cite{Sal98,FrPi13}. To prove the lemma we first need the following generalization of \cite[Proposition 5.0.9]{CLS11} to general rings.
\begin{proposition} \label{prop: affine scheme categorical quotient}
    Let $\cG$ be a linearly reductive group scheme over a ring $R$ (as in \cite[Section 12]{Alp13}) acting on an affine $R$-scheme $\spec A$. Then the induced morphism $\spec A\rightarrow \spec A^\cG$ is a universal categorical quotient for this action, where $A^\cG$ is the subalgebra containing the elements in $A$ invariant under the action of $\cG$.
\end{proposition}
\begin{proof}
   Since $\cG$ is linearly reductive, by \cite[Remark 4.8]{Alp13} (see also \cite[Theorem 13.2]{Alp13}) the morphism of $R$-stacks $[\spec A/\cG]\rightarrow \spec A^\cG$ is a good moduli space \cite[Definition 4.1]{Alp13}. In particular, any morphism $\spec A\rightarrow S$ to some $R$-scheme $S$ which is constant on $\cG$-orbits factors uniquely as $\spec A\rightarrow [\spec A/\cG]\rightarrow \spec A^\cG\rightarrow S$, since good moduli spaces are universal for maps to schemes \cite[Theorem 4.16(vi)]{Alp13}. Thus we see that $\spec A\rightarrow \spec A^\cG$ is a categorical quotient. Since good moduli spaces are preserved under base change \cite[Proposition 4.7]{Alp13}, it is furthermore a universal categorical quotient.
\end{proof}
\begin{proof}[Proof of Lemma \ref{lemma: Cox quotient universal}]
    The proof of the first part is exactly as in \cite[Theorem 5.1.11]{CLS11}, where we replace the use of \cite[Proposition 5.0.9]{CLS11} by Proposition \ref{prop: affine scheme categorical quotient}, where we use that $\cG$ is linearly reductive by \cite[Example 12.4.(2)]{Alp13}.
    The second part follows from \cite[Theorem 3.3.19]{CLS11}, where regularity is used to ensure that the bijective morphisms of cones $\hat{\sigma}\rightarrow \sigma$ induced by $\pi$ restrict to bijective maps $\hat{\sigma}\cap \left(\Z^{\Sigma(1)}\right)^\vee\rightarrow \sigma\cap N$.
\end{proof}



Using the construction just given, we can define Cox coordinates on a split toric variety.
\begin{definition}
    Let $\cX$ and $\pi\colon \cY\rightarrow \cX$ be as above, and let $B$ be a scheme.
    For any $P\in \cY(B)$, we say that $P=(P_1, \dots, P_n)$ are \textit{Cox coordinates} for the point $\pi(P)\in\cX(B)$, where $P_i\in \cO(B)$. We will write $\pi(P)=(P_1: \dots: P_n)$, in analogy with homogeneous coordinates.
\end{definition}

When can all points in $\cX(B)$ be represented by Cox coordinates for a scheme $B$? If $\Sigma$ is regular and complete, then $\Cl(\Sigma)$ is torsion-free. Therefore Lemma \ref{lemma: Cox quotient universal} shows that the Cox morphism is a $\cG$-torsor. Therefore \cite[Proposition 2.1]{FrPi13} gives a decomposition
\begin{equation}
\cX(B)=\bigsqcup_{[W]\in H^1_{fppf}(B,\cG)} {}_W\pi( {}_W\cY(B)), 
\end{equation}
for every scheme $B$. Here $_W\pi\colon _W \cY\rightarrow \cX$ is the twist of $\pi$ by $-[W]\in H^1_{fppf}(B,\cG)$ as defined in \cite[p.22]{Sko01}.

Note that in this case $\cG\cong \mathbb{G}_{m,B}^{n-d}$. Thus $H_{fppf}^1(B,\cG)\cong \Pic(B)^{n-d}$ if $B$ is regular. If $B=\spec R$ for a unique factorisation domain then this implies that every $R$-point is represented by Cox coordinates.
\begin{proposition} \label{prop: Cox coordinates}
Let $\Sigma$ be a regular and complete fan, let $\cX=\cX_\Sigma$, and let $R$ be a unique factorisation domain. Then every $P\in \cX(R)$ is represented by Cox coordinates:
\begin{equation}
\cX(R)=\pi(\cY(R)),
\end{equation}
where $\pi$ is the Cox morphism.
\end{proposition}

If we instead consider singular toric varieties, then we can still show that rational points are represented by Cox coordinates, as long as $\Cl(\Sigma)$ is torsion-free.
\begin{proposition} \label{prop: Cox coordinates singular}
Let $\Sigma$ be a fan such that $\Cl(\Sigma)$ is torsion-free and let $k$ be a field. Then every point $P\in \cX(k)=\cX_\Sigma(k)$ is represented by Cox coordinates:
\begin{equation}
\cX(k)=\pi(\cY(k)),
\end{equation}
where $\pi$ is the Cox morphism.
\end{proposition}
\begin{proof}
By the Orbit-Cone Correspondence \cite[Theorem 3.2.6, Proposition 3.2.7]{CLS11}, the torus invariant orbits $V\cong \mathbb{G}^{s}_m(k)$ on a split toric scheme $\cX$ correspond to the closed toric subschemes $\cZ$ of $\cY$ of dimension $s$ by restricting to the dense torus in $\cZ$ and by taking $k$-points.
Since $\pi$ induces a bijection between the closed toric subschemes of $\cY$ and those of $\cX$, it also induces a bijection between the torus orbits in $\cY(k)$ and in $\cX(k)$. Since every torus orbit contains a $k$-point, every torus orbit $V\subset \cX(k)$ contains the image of a point in $\cY(k)$. Since $\Cl(\Sigma)$ is torsion-free, the short exact sequence $0\rightarrow N^\vee\rightarrow \Z^{\Sigma(1)}\rightarrow \Cl(\Sigma)\rightarrow 0$ splits, which gives a splitting of the associated exact sequence
$$0\rightarrow \cG\rightarrow \mathbb{G}_{m,\Z}^n\rightarrow \mathbb{G}_{m,\Z}^d\rightarrow 0,$$
and thus the map $\mathbb{G}_{m,\Z}^n(k)\rightarrow \mathbb{G}_{m,\Z}^d(k)$ is surjective. Let $V\subset \cX(k)$ be a torus orbit and consider the map $\phi\colon \mathbb{G}_{m,\Z}^d(k)\rightarrow V$ from the $k$-points of the dense torus in $\cX$ to $V$, induced by the map on tori. By combining \cite[Lemma 3.2.5]{CLS11} and the same splitting argument as above to the exact sequence (3.2.6) in \cite{CLS11}, we see that $\phi$ is surjective. Combining the surjectivity of these maps, the composite map $\mathbb{G}_{m,\Z}^n(k)\rightarrow V$ is surjective. By the commutative diagram
    
    $$\begin{tikzcd}
    \mathbb{G}_{m,\Z}^n(k) \arrow[d] \arrow[r] & V' \arrow[d] \\
    \mathbb{G}_{m,\Z}^d(k) \arrow[r, "\phi"]           & V,           
    \end{tikzcd}$$
where $V'$ is the torus orbit in $\cY(k)$ above $V$, the map $V'\rightarrow V$ is surjective. Since every rational point lies in a torus orbit, the map $\cY(k)\rightarrow \cX(k)$ is surjective.
\end{proof}

By applying Proposition \ref{prop: Cox coordinates singular}, we will show that any toric resolution of singularities of normal split toric varieties is surjective on rational points. For this we first need the following proposition, which characterizes when the class group of a split toric variety is torsion-free.
\begin{proposition} \label{prop: torsion-free Picard}
Let $X$ be a normal split toric variety without torus factors.
Then $\Cl(X)$ is torsion-free if and only if the ray generators $\{n_{\rho}\mid \rho\in \Sigma(1)\}$, span $N$ as a lattice. Similarly, for a prime number $p$, $\Cl(X)$ does not have $p$-torsion if and only if the ray generators $\{n_{\rho}\mid \rho\in \Sigma(1)\}$ span $N/pN$ as a lattice.
\end{proposition}
\begin{proof}
The class group being torsion-free is equivalent to $\Cl(X)$ being a projective $\Z$-module, which by the splitting lemma is equivalent to the existence of a retraction $\Z^{\Sigma(1)}\rightarrow N^\vee$ of the exact sequence \eqref{equation: exact sequence class group}. Taking $\Z$-duals, we see this is equivalent to the projection $\Z^{\Sigma(1)}\rightarrow N$ having a section, and therefore to it being surjective. The second statement follows from tensoring with $\Z/p\Z$ and using the same argument for $\Z/p\Z$ instead of $\Z$.
\end{proof}
\begin{corollary} \label{corollary: resolution of singularities surjective}
Let $k$ be a field and let $X$ be a normal split toric variety over $k$. Then any proper birational toric morphism $f\colon \widetilde{X}\rightarrow X$ induces a surjection $\widetilde{X}(k)\rightarrow X(k)$. In particular, this holds whenever $f$ is a resolution of singularities.
\end{corollary}
\begin{proof}
	Let $\Sigma$ and $\widetilde{\Sigma}$ be the fans of $X$ and $\widetilde{X}$ in a common co-character lattice $N$ and let $P\in X(k)$ be a point, not contained in the open torus. We first show that we can assume that $\Cl(X)$ is torsion free. We compactify $X$ to obtain a complete toric variety $X\subseteq X'$. Since the intersection of all toric affine opens $U_\sigma\subset X'$ corresponding to maximal cones $\sigma\in \Sigma$ is just the torus, there exists such an open $U_\sigma$ not containing $P$. We subdivide $\sigma$ into a collection of smooth maximal cones including $\sigma''$ and let $X''$ be the resulting complete toric variety. This yields a birational morphism $X''\rightarrow X'$ of complete toric varieties, such that $X''$ contains an affine open $U_{\sigma''}$ corresponding to a smooth maximal cone. Therefore Proposition \ref{prop: torsion-free Picard} implies that the class group of $X''$ is torsion-free. Since the statement to be proved is Zariski local, and there exists an open subset $P\in U\subset X$ such that the restriction $\tilde{X}\times_{X'} U\rightarrow U$ is an isomorphism, we can assume without loss of generality that $\Cl(X)$ is torsion free.
 
    The toric birational morphism $\tilde{X}\rightarrow X$ is induced by subdividing the cones in $\Sigma$ into smaller cones, which gives an inclusion $\Sigma(1)\subseteq \widetilde{\Sigma}(1)$ of sets of rays. Let $\widetilde{\Sigma}(1)\setminus\Sigma(1)=\{\rho_{n+1},\dots, \rho_{\tilde{n}}\}$ be the rays in $\widetilde{\Sigma}$ which do not lie in $\Sigma$. By Proposition \ref{prop: Cox coordinates singular} there exists Cox coordinates for $P$: $P=(a_1:\dots: a_n)\in X(K)$. Then consider the point $\tilde{P}=(a_1:\dots:a_n:1\dots:1)\in \widetilde{X}(k)$. Then it follows from the construction of Cox coordinates that $f(\widetilde{P})=P$.
\end{proof}

\subsection{\texorpdfstring{$M$}{M}-approximation for normal split toric varieties}
Now we restrict ourselves to the case where $X$ is a complete normal split toric variety of dimension $d$ over a PF field $K$. Consider a toric pair $(X,M)$, where $M=((D_i)_{i=1}^n, \fM)$, and let $(\cX,\cM)$ be its toric integral model. These assumptions are fixed for the rest of the paper, unless specified otherwise.

To understand whether a toric pair $(X,M)$ satisfies $M$-approximation, we can use Proposition \ref{prop: birational morphisms M-approximation} to reduce to the case of a smooth toric variety:
\begin{corollary} \label{corollary: M approx toric stable under birational transformations}
Let $(K,C)$ be a PF field, let $T\subset \Omega_K$ be a finite set of places, let $(X,M)$ be a toric pair and let $f\colon \widetilde{X}\rightarrow X$ be a birational toric morphism of complete normal split toric varieties over $K$. Then $(X,M)$ satisfies $M$-approximation off $T$ if and only if $(\widetilde{X}, f^{-1}M)$ satisfies $f^{-1}M$-approximation off $T$.
\end{corollary}
\begin{proof}
This follows directly from combining Proposition \ref{prop: birational morphisms M-approximation} with Corollary \ref{corollary: resolution of singularities surjective}.
\end{proof}

\subsection{Monoids in Theorem \ref{theorem: M-approx}}
In this section we introduce the monoids $N_M,N_M^+$ and $\rho(K,C)$, which are used in Theorem \ref{theorem: M-approx} and Theorem \ref{theorem: thinness}. These monoids indicate how $(\cX,\cM)(B)$ is distributed in $X(K)$.
\begin{definition} \label{def: invariants}
Let $(X,M)$ be a toric pair where $X$ is a complete smooth split toric variety. Define the homomorphism of monoids
\begin{align*} \label{equation: phi}
	\phi\colon \mathbb{N}^n &\rightarrow N\\
	(m_1,\dots, m_n) &\mapsto \sum_{i=1}^n m_i n_{\rho_i}.
\end{align*}
Define the sublattice
$$N_M=\left\langle\phi(\mathbf{m})\mid \mathbf{m}\in \fM_{\fin,\red}\right\rangle\subset N,$$
and the submonoid $N_M^+$ generated by nonnegative linear combinations of the same elements. 
\end{definition}
The restriction to the reduced part reflects the fact that for every finite place $v\in \Omega_K$ and a point $P\in X(K_v)$, $\mult_v (P)$ lies in $\fM_{\red}$.
\begin{remark}
    The monoid $N_M^+$ is equal to $N_M$ if and only if the cone $N_{M,\mathbb{R}}^+$ generated by $N_M^+$ is equal to $\mathbb{R}^d$. This will be used in the proof of Theorem \ref{theorem: M-approx}.
\end{remark}

For each finite place $v$, we write
\begin{equation} \label{eq: phi_v}
\phi_v:=\phi\circ \mult_v\colon U(K_v)\rightarrow N.  
\end{equation}
The map $\phi_v\colon U(K_v)\rightarrow N$ is a group homomorphism, as the next proposition shows.
\begin{proposition} \label{prop: Compatibility multiplicity with fan}
	Let $X$ be a smooth, complete split toric variety. For each finite place $v$, we write
\begin{equation}
\phi_v:=\phi\circ \mult_v\colon U(K_v)\rightarrow N.  
\end{equation}
	The map $\phi_v\colon U(K_v)\rightarrow N$ defined above is a surjective group homomorphism with kernel $\cU(\cO_v)$, where $\cU\cong \mathbb{G}^d_{m}$ is the open torus in $\cX$. The homomorphism can be given in Cox coordinates as
	$$\phi_v(u_1\pi^{w_1}: \dots: u_n\pi^{w_n})=\sum_{i=1}^n w_i n_{\rho_i},$$
	where $w_i\in \Z$, $u_i\in \cO_v^\times$ and $\pi\in \cO_v$ is a uniformizer.
	This homomorphism gives a splitting
	$$U(K_v)\cong \cU(\cO_v)\oplus N.$$
	Furthermore, if $\widetilde{X}$ is a smooth complete split toric variety and $f\colon \widetilde{X}\rightarrow X$ is a toric morphism corresponding to the morphism of lattices $\overline{f}\colon \widetilde{N}\rightarrow N$, then we have $\phi_v\circ f=\overline{f}\circ \phi_v$. The map $\phi_v$ on the left corresponds to the map on the points in $X(K_v)$ and the map $\phi_v$ on the right corresponds to the map on the points in $\widetilde{X}(K_v)$.

    As a consequence if $f$ is birational, then for any toric pair $(X,M)$ there are equalities of monoids $N_M^+=N^+_{f^*M}$ and $N_M=N_{f^*M}$.
\end{proposition}
\begin{proof}
Let $P\in X(K_v)$ be a point. Since $\cX_{\Sigma}$ has an open cover of affine toric schemes $\cV_{\sigma}=\mathbb{A}^d$ corresponding to maximal cones $\sigma\in \Sigma$, $P\in\cV_\sigma(\cO_v)$ is satisfied for some maximal cone $\sigma\in \Sigma$. 
Thus we can represent the the point with Cox coordinates $P=(p_1:\dots:p_n)$ such that $p_i=1$ if $\rho_i$ is not a ray of the cone $\sigma$. For such a point $P\in \mathbb{A}^d(\cO_v)\subset \cX_{\Sigma}$, it follows that $n_v(\cD^c_i,P)=v(p_i)$. Therefore, if we write $p_i=u_i\pi^{m_i}$ for units $u_1,\dots, u_n\in \cO_v^\times$ and $m_1,\dots, m_n\in \mathbb{N}^*$, then $\phi_v(u_1\pi^{m_1}:\dots: u_n\pi^{m_n})=\sum_{i=1}^n m_i n_{\rho_i}.$
We will now prove that this equality is still true even without the constraints on the $p_i$.
By the equality $$\cG(K_v)=\{(t_1,\dots, t_n)\mid \prod_{i=1}^n t_i^{\langle e_j, n_{\rho_i} \rangle}=1\text{ for all } 1\leq j\leq d\},$$
where $e_1,\dots, e_d$ is a choice of a basis of $N^\vee$,
we see that $(\pi^{m_1}:\dots: \pi^{m_n})=(1:\dots:1)$ if and only if $\sum_{i=1}^n m_i n_{\rho_i}=0$.
Therefore if $(\pi^{m_1}:\dots: \pi^{m_n})=(\pi^{m'_1}:\dots: \pi^{m'_n})$ then $\sum_{i=1}^n (m_i-m'_i) n_{\rho_i}=0$ so
$\sum_{i=1}^n m_i n_{\rho_i}=\sum_{i=1}^n m'_i n_{\rho_i}$.

Thus we see that
$$\phi_v(u_1\pi^{m_1}:\dots: u_n\pi^{m_n})=\sum_{i=1}^n m_i n_{\rho_i},$$
for all $m_i\in \Z$ and units $u_i\in \cO_v^\times.$
In particular it is clear that $\phi_v$ is a group homomorphism with kernel $\cU(\cO_v)$.

The surjectivity follows from the fact that $n_{\rho_1},\dots,n_{\rho_n}$ span $N$ as a lattice, combined with the identity $\phi_v(1:\dots: 1: \pi:1: \dots:1)=n_{\rho_i}$, where $i$ is index of the coordinate different from $1$.
The splitting is a direct consequence of the fact that $N$ is a free abelian group.

For verifying the identity $\phi_v\circ f=\overline{f}\circ \phi_v$, it suffices to consider affine opens $\mathbb{A}^{d'}_{\cO_v}\subseteq \widetilde{\cX}$ and $\mathbb{A}^{d}_{\cO_v}\subseteq \cX$ such that $f$ restricts to a morphism $\mathbb{A}^{d'}_{\cO_v}\rightarrow \mathbb{A}^{d}_{\cO_v}$. Now by comparing this map with the map $\overline{f}$, the result directly follows.

The final claims follows from the previous part by noticing that $\overline{f}$ is just the identity.
\end{proof}

\begin{remark} \label{remark: description M-points toric variety}
    By Proposition \ref{prop: Compatibility multiplicity with fan}, the description of $\cM$-points on projective space as in Section \ref{subsection: M-points on projective space} generalizes to toric pairs $(X,M)$, with $X$ complete and smooth, with toric integral model $(\cX,\cM)$. By replacing the homogeneous coordinates from that section with Cox coordinates, we obtain a description for the $\cM$-points on $(\cX, \cM)$.
\end{remark}

Now we can extend the definitions of $N_M$ and $N_M^+$ to the singular case.
\begin{definition} \label{def: N_M singular} 
    Let $X$ be a complete normal split toric variety with lattice of cocharacters $N$. We define $N_M=N_{f^*M}$ and $N_M^+=N_{f^* M}$ for any birational morphism $f\colon Y\rightarrow X$, where $Y$ is a complete smooth split toric variety, such that $D_i\times_X Y$ is a Cartier divisor for all $i\in\{1,\dots,n\}$.
\end{definition}

Such a $Y$ can always be found by taking $g\colon Z\rightarrow X$ to be the successive blowing up of the $D_i$, so that $g^{-1}D_i$ is a Cartier divisor, and then taking $Y$ to be a resolution of singularities of $Z$.

It follows from Proposition \ref{prop: Compatibility multiplicity with fan} that $N_M$ and $N_M^+$ are independent of the choice of the morphism $f$, so they are well defined. This is because for any two resolutions of singularities of $X$, there exists a common refinement of both.
\begin{remark}
If $X$ is a normal split toric variety such that $\Cl(X)$ contains torsion, then $N_M^+=N$ for the trival pair $(X,M)=(X,0)$. On the other hand, Proposition \ref{prop: torsion-free Picard} implies that the ray generators $n_{\rho_i}$ do not generate $N$, even as a group. Therefore, it is sometimes necessary to consider a resolution of singularities rather than directly trying to apply Definition \ref{def: invariants}, as that can give monoids which are too small. 
\end{remark}

The next notion measures divisibility of the unit group of completions of the field $K$.
\begin{definition}
For a PF field $(K,C)$ we define $\rho(K,C)$ to be the set of $n\in\mathbb{N}^*$ such that the group $\cO_v^\times$ is $n$-divisible for all $v\in\Omega_K$.
\end{definition}
The set $\rho(K,C)$ is a submonoid of $\mathbb{N}^*$ generated by a subset of the prime numbers. In order to describe this notion for function fields we introduce the following definitions:
\begin{definition}
	Let $k$ be a field and let $n>1$ be an integer with $\CHAR(k) \nmid n$. We say that $k$ is \textit{$n$-closed} if one of the following equivalent properties hold:
    \begin{enumerate}
    \item For every finite extension $l/k$, the degree $[l:k]$ is coprime to $n$.
    \item For every finite extension $l/k$, the group $l^\times$ is $n$-divisible.
    \end{enumerate} 
\end{definition}
\begin{example}
    Separably closed fields are $n$-closed for all $n$ not divided by the characteristic. For any prime number $p$ and an integer $n>1$ with $p \nmid n$ the union of finite fields $\bigcup_{m\geq 1}\mathbb{F}_{p^{n^m}}$ is $n$-divisible. Similarly for a separably closed field $k$ with $\CHAR(k)=p$ and an integer $n>1$ with $p \nmid n$, the field $\bigcup_{m\geq 1}k\lParen t^{-1/n^m}\rParen$ is $n$-divisible.
\end{example}
Recall that a field $k$ is formally real if there exists an ordering on $k$ and it is formally Euclidean if this ordering can be chosen such that every nonnegative element is a square.
\begin{definition}
    We say that $k$ is \textit{hereditarily Euclidean} if every formally real algebraic extension of $k$ is formally Euclidean.
\end{definition}
The following lemma allows us to easily compute $\rho(K,C)$ for both number fields and function fields.
\begin{lemma} \label{lemma: computation rho}
For any PF field $(K,C)$, the monoid $\rho(K,C)$ is computed as follows:
	\begin{enumerate}
	    \item If $K$ is a number field then $\rho(K,C)=1$.
            \item If $k$ is a field and $K=k(C)$, where $C$ is a regular curve over $k$, then a prime number $p$ belongs to $\rho(K,C)$ if and only if either
        \begin{enumerate}
            \item $p\neq \CHAR(k)$ and $k$ is $p$-closed,
            \item or all of the following are satisfied:
            \begin{itemize}
                \item $k$ is a hereditarily Euclidean field,
                \item $p=2$,
                \item $C(k')=\emptyset$, where $k'$ is a real closure of $k$.
            \end{itemize}
        \end{enumerate}
	\end{enumerate}
 Furthermore, for any PF field $(K,C)$, if $n\in \mathbb{N}^*$ is an integer such that $n\not\in \rho(K,C)$, then there are infinitely many places $v\in \Omega_K$ such that $\cO_v^\times$ is not $n$-divisible.
\end{lemma}
\begin{remark}
In particular, if $K=k(C)$ with $k$ a finite field, a number field or a function field (of transcendence degree at least $1$), then $\rho(K,C)=1$. If on the other hand, $k$ is separably closed, then $\rho(K,C)=\{n\in \mathbb{N}^*\mid \CHAR(k) \nmid n\}$. Finally, if $k=\mathbb{R}$, then $\rho(K,C)=\{n\in \mathbb{N}^*\mid 2\nmid n\}$ if $C(\mathbb{R})\neq \emptyset$ and $\rho(K,C)=\mathbb{N}^*$ otherwise. 
\end{remark}
\begin{proof}
    We will prove the last statement in tandem with the computation of $\rho(K,C)$. Note that for this statement we can assume that $n=p$ is a prime.
    We split up the proof in two cases, depending on whether $K$ is a number field or a function field.
    We first treat the case where $K$ is a number field. For every prime number $p$ there exist infinitely many prime numbers $q\equiv 1\bmod p$ by Dirichlet's theorem on arithmetic progressions. For each place $v\in \Omega_K$ above such a prime number $q$, the group of units of the residue field $k_v^\times$ is not $p$-divisible since the order of $k_v^\times$ is divisible by $p$. Thus $\cO_v^\times$ is not $p$-divisible either for such $v$. In particular, $\cO_v^\times$ is not $p$-divisible for infinitely many places $v\in \Omega_K$ and therefore $\rho(K,C)=1$.

    Now we treat the case where $K$ is a function field of a curve over a ground field $k$.
	For any place $v\in \Omega_K$, the completion is given by $\cO_v\cong k_v\lBracket t\rBracket$, where $k_v$ is the residue field at $v$. For any $f\in k_v\lBracket t\rBracket^\times$ we can write $f=ag$ where $a\in k_v^\times$ and $g\in k_v\lBracket t\rBracket^\times$ has constant coefficient $1$. Therefore $f$ is a $p$-th power if and only if $a$ and $g$ are $p$-th powers. If $p\neq \CHAR(k)$, then for any $x\in k_v\lBracket t\rBracket$ with $|x|_v<1$ the $p$-th root $\sqrt[p]{1+x}$ is well defined and lies in $k_v\lBracket t\rBracket$. In particular, since $|g-1|_v<1$, $g$ is a $p$-th power as long as $p\neq \CHAR(k)$.
    If on the other hand $p=\CHAR(k)$, then $1+t\in k_v\lBracket t\rBracket^\times$ is not a $p$-th power. Therefore $\cO_v^\times$ is $p$-divisible if and only if $p\neq \CHAR(k)$ and $k_v^\times$ is $p$-divisible.
 
    Therefore it follows that $p\in \rho(K,C)$ if $k$ is $p$-closed and $p\neq \CHAR(k)$ since in that case $k_v$ is $p$-closed for all $v\in \Omega_K$. Similarly, $2\in \rho(K,C)$ if $k$ is hereditarily Euclidean and $C(k')=\emptyset$ for a real closure $k'/k$, since then $k_v$ is $2$-closed for all $v\in \Omega_K$.
	
    For the other direction, we assume that $\cO_v^\times$ is $p$-divisible for all but finitely many places. We will show $p$ satisfies the conditions given in the statement of the lemma, and thus in particular $p\in \rho(K,C)$.
 
 As we have seen we must have that $p\neq \CHAR(k)$ and that for all but finitely many places $v\in \Omega_K$, the group $k_v^\times$ is $p$-divisible.
 First we prove that for any field $\tilde{k}$, if $\tilde{k}$ is not $p$-divisible, but $l^\times$ is $p$-divisible for some some finite extension $l/\tilde{k}$, then $p=2$ and $\tilde{k}$ is an Euclidean field. If $p=2$, then this follows from \cite[Lemma 2(2)]{ElWa87}. Assume therefore that $p$ is odd. Let $a\in \tilde{k}$ be an element which is not a $p$-th power. Then \cite[Theorem 9.1]{Lan02} shows that the polynomial $X^{p^n}-a$ is irreducible over $\tilde{k}$ for every $n\geq 1$. Thus if $l/\tilde{k}$ is a finite extension, then $X^{p^n}-a$ does not have a linear factor over $l$ if $p^n$ is larger than the degree of the extension $l/\tilde{k}$. Therefore $l^\times$ is not $p$-divisible either since $a$ is not a $p^n$-th power in $l$.

 Since for any place $v\in \Omega_K$, $k_v/k$ is a finite extension, $k^\times$ has to be either $p$-divisible or $p=2$ and $k$ is Euclidean.

 Now assume that $k^\times$ is $p$-divisible and that $k_v^\times$ is $p$-divisible for all but finitely many places $v\in \Omega_K$, as before. We will show that this implies that $k$ is $p$-closed.
 Since $k^\times$ is $p$-divisible, $k$ is not Euclidean if $p=2$. If $k$ is not $p$-closed, then there exists a finite extension $l/k$ such that $l^\times$ is not $p$-divisible. We can factor $l/k$ as a separable extension $\tilde{l}/k$ followed by a totally inseparable extension $l/\tilde{l}$ by \cite[Tag 030K]{Stacks}. Since any separable extension of $k$ is simple, $\tilde{l}$ is contained in the residue field of infinitely many closed points in $\P^1_k$. Thus, since there exists a dominant morphism $C\rightarrow \P^1_k$, there exist infinitely many places $v$ of $K$ for which $\tilde{l}\subset k_v$. This implies that $k_v^\times$ is not $p$-divisible for infinitely many places $v\in \Omega_K$, which is a contradiction. Thus if $k^\times$ is $p$-divisible, then $\tilde{l}^\times$ is $p$-divisible for every separable extension $\tilde{l}/k$.

 For any totally inseparable extension $l/\tilde{l}$, any element $\alpha\in l^\times$ has a minimal polynomial of the form $X^{q^n}-\alpha^{q^n}$ for some $n\in\mathbb{N}$, where $q=\CHAR(k)$ and $\alpha^{q^n}\in \tilde{l}^\times$. Since $\alpha^{q^n}$ is a $p$-th power in $\tilde{l}$, $\alpha$ is a $p$-th power in $l$ and thus $l^\times$ is also $p$-divisible. Therefore, we see that $k$ is $p$-closed.
 By the same argumentation, we also see that if $k$ is Euclidean, then $k$ is hereditarily Euclidean.
	
If $k$ is hereditarily Euclidean, but $C$ contains a $k'$-rational point for its real closure $k'$, then $C(k')$ is infinite since a real closed field is an ample field \cite{Pop96}. Therefore there are infinitely many places $v$ such that $k_v^\times$ is not $2$-divisible.
\end{proof}

The next lemma will be used in the proof of Theorem \ref{theorem: M-approx} over function fields.
\begin{lemma} \label{lemma: Picard group divisible}
    Let $k$ be a field and let $C$ be a projective regular curve over $k$ and let $p\in \rho(K,C)$ be a prime number. Then for any affine open $B\subset C$, $\Pic(B)$ is a $p$-divisible group.
\end{lemma}
\begin{proof}
    We first assume that $k$ is $p$-closed. Since the statement only depends on $p$ and on the scheme $B$, we can as in Remark \ref{remark: take geometrically connected} assume without loss of generality that $k$ is algebraically closed in $K$, so that $B$ is a geometrically integral curve over $k$. The connected component of the identity of the Picard scheme $J=\mathrm{Jac}(C)$ is a group scheme of finite type over $k$ \cite[Section 8.2, Theorem 3]{BLR90}. We will first prove that $J(k)$ is $p$-divisible. The multiplication-by-$p$ map $[p]$ gives the exact sequence of $G$-modules
    $$0 \to J(k^{\mathrm{sep}})[p] \to J(k^{\mathrm{sep}}) \stackrel{[p]}{\to} J(k^{\mathrm{sep}}) \to 0,$$
    where $k^{\mathrm{sep}}$ is the separable closure of $k$ and $G=\gal(k^{\mathrm{sep}}/k)$ is the absolute Galois group of $k$. By the induced long exact sequence in Galois cohomology, this implies that $J(k)$ is $p$-divisible if $H^1(G, J(k^{\mathrm{sep}})[p])=0$, which we will now show.

    By \cite[Chapter 7, Corollary 5.23.]{Liu06}, $J(k^{\mathrm{sep}})[p]$ is finite since $p$ is different from the characteristic.
    Furthermore every finite quotient of $G$ has order coprime to $p$, since $k$ is $p$-closed. Thus by \cite[Corollary 1.49]{Har20}
    we see that, for any normal subgroup $H\triangleleft G$ of finite index, $H^1(G/H, (J(k^{\mathrm{sep}})[p])^H)=0$. Thus the inflation-restriction exact sequence \cite[Theorem 1.42]{Har20} implies that the restriction map $H^1(G,J(k^{\mathrm{sep}})[p])\rightarrow H^1(H,J(k^{\mathrm{sep}})[p])^{G/H}$ is injective. For every continuous map $G\rightarrow J(k^{\mathrm{sep}})[p]$, one of the fibers contains an open neighbourhood of $1\in G$, and therefore a normal subgroup $H$ of finite index. Thus the map of the associated cohomology class gets sent to $0$ by the restriction map. Since the map is injective, $H^1(G,J(k^{\mathrm{sep}})[p])=0$ and thus we see that $J(k)$ is $p$-divisible.

    By applying \cite[Section 8.1, Proposition 4]{BLR90} with $S=T=\spec k$, we have an exact sequence
    $$0\rightarrow \Pic^0(C)\rightarrow J(k)\rightarrow \Br(k),$$
    where $\Pic^0(C)$ is the group of isomorphism classes of invertible sheaves on $C$ with degree $0$ and $\Br(k)$ is the Brauer group of $k$. Since $k$ is $p$-closed, $\Br(k)[p]=0$, since there does not exist a central division algebra of degree $p$ over $k$, as its splitting field would have degree $p$ over $k$. Thus it follows that $\Pic^0(C)$ is also $p$-divisible.

    Since $C$ is geometrically connected, the degree map gives an isomorphism $\Pic(C)/\Pic^0(C)\cong \mathbb{Z}$.
    Since the restriction map $\Pic(C)\rightarrow \Pic(B)$ is surjective, there is an exact sequence $$\Pic^0(C)\rightarrow \Pic(B)\rightarrow \Z/m\Z\rightarrow 0.$$
    Here $m$ is the greatest common divisor of the degrees of closed points in $C\setminus B$. The degree of a closed point on $C$ is not divisible by $p$, as the residue field of such a divisor would be an extension of $k$ with degree divisible by $p$. Therefore $p$ does not divide $m$, so $\Z/p\Z$ is $p$-divisible and thus $\Pic(B)$ is $p$-divisible.

    Now we treat the case where $k$ is not $p$-closed. As we have seen in Lemma \ref{lemma: computation rho}, this implies that $p=2$, $k$ is hereditarily Euclidean and $C$ is a curve with no points defined over a real closure of $k$. This implies that every divisor on $C$ defined over $k$ is of the form $D=D'+\sigma(D')$, where $D'$ is a divisor defined over the unique quadratic extension $l=k(\sqrt{-1})$ of $k$, and $\sigma$ is the automorphism of $l/k$ that sends $\sqrt{-1}$ to $-\sqrt{-1}$. As $l$ is $2$-closed by \cite[Lemma 2]{ElWa87}, it follows by the previous argument that $\Pic(B_l)$ is $2$-divisible, so there exists a divisor $D''$ over $l$ such that $2D''\sim D'$, and thus we see $D\sim 2(D''+\sigma(D''))$. Therefore $\Pic(B)$ is $2$-divisible as well.
\end{proof}

\subsection{Squarefree strong approximation on the affine line} \label{section: squarefree strong approximation}
In the proof of Theorem \ref{theorem: M-approx}, we will need the following definitions.
\begin{definition}
	Let $R$ be an Dedekind domain. An element $r\in R$ is \textit{squarefree} if $r$ is not contained in $\fp^2$ for any prime ideal $\fp\subset R$. If $R=\cO(B)$ for some regular curve $B$ over a field $k$, then we call an element $r\in R$ \textit{separable} if it is squarefree in $R\otimes_k k'$ for every extension $k'$ of $k$.
\end{definition}
Note that if $R=k[t]=\cO(\mathbb{A}_k^1)$, then we recover the familiar notions of squarefree polynomials and separable polynomials. If $k$ is perfect, then $r\in\cO(B)$ is separable if and only if it is squarefree.

Before we can carry out the proof of Theorem \ref{theorem: M-approx}, we prove a stronger version of strong approximation on $\mathbb{A}^1$ for both global fields and function fields. 

\begin{lemma} \label{Lift number field}
Let $(K,C)$ be a global field, let $S\subset \Omega_K$ be a finite set of places containing a distinguished place $v_0\in S$. Let $x_v\in K_{v}^\times$ for $v\in S$, and $\epsilon>0$. Then there exist infinitely many pairwise coprime squarefree elements $f\in \cO(B)$ such that $$|f-x_v|_{v}<\epsilon \text{ for all }v\in S\setminus\{v_0\},$$
where $B=C\setminus S$.

Let $R$ be an integer and assume that $v_0$ is an infinite place if $K$ is a number field.
If $|x_{v_0}|_{v_0}$ is sufficiently large, depending on $\epsilon$, $R$ and $|x_v|_{v}$ for $v\in S\setminus\{v_0\}$, then there exist at least $R$ such $f$ which additionally satisfy $$|f-x_{v_0}|_{v_0}<\epsilon |x_{v_0}|_{v_0}.$$

Furthermore, $f$ can be taken to be a prime element if $K$ is a function field and $v_0$ is a $k$-rational point, or if $K$ is a number field (with no condition on $v_0$). In general, $f$ can be taken to be the product of two prime elements. 


\end{lemma}

\begin{remark}
Note that for $K=\mathbb{Q}$ this lemma is just a consequence of the prime number theorem for arithmetic progressions \cite[Theorem 1.1]{BMBR18}. If $K$ is a number field and if $S$ does not contain an infinite place, then the statement from the lemma follows from Chebotarev's density theorem applied to $L/K$, where $L$ is the ray class field associated to the modulus $\infty I$ (see \cite[Chapter 2, \S 8]{Cox22}), where $\infty$ is the product of the infinite places and $I$ is the ideal in $\cO_K$ consisting of the elements $x\in \cO_K$ with $|x|_{v}<\epsilon$ for all $v\in S$.
\end{remark}
\begin{remark}
If $S$ contains all infinite places in the setting of Lemma \ref{Lift number field}, then there exist only finitely many coprime elements $f\in \cO(B)$ satisfying $|f-x_v|_{v}<\epsilon$ for all $v\in S\setminus\{v_0\}$ and $|f-x_{v_0}|_{v_0}<\epsilon |x_{v_0}|_{v_0},$ since these inequalities imply an upper bound on the norm of the ideal $(f)\in B$.
\end{remark}
\begin{proof}
The proof of the lemma uses the language of ideles and mainly relies on \cite[Chapter XV, Theorem 6]{Lan94}, which is an equidistribution result that can be viewed as a generalization of Chebotarev's density theorem. For more background on ideles and equidistribution, see \cite[Chapter VII]{Lan94} and \cite[Chapter XV]{Lan94}, respectively.

Denote the idele group of $K$ by $\mathbf{J}= \prod_{v\in \Omega_K} (K_v^\times, \cO_v^\times)$ and denote the $S$-idele group by $\mathbf{J}_{S}=\prod_{v\in \Omega_K\setminus S} \cO_v^\times\times \prod_{v\in S} K_v^\times$. The norm of an idele $a\in \mathbf{J}$ is $\| a\|=\prod_{v\in \Omega_K} |a_v|_v$. We denote the subgroups of elements of norm $1$ in $\mathbf{J}$ and $\mathbf{J}_S$ by $\mathbf{J}^0$ and $\mathbf{J}_S^0$.

The proof is split up in two parts, depending on whether $K$ is a number field or a function field. While the proofs of these cases differ, they follow the same general ideas and both hinge on applying \cite[Chapter XV, Theorem 6]{Lan94} to a retraction $\phi\colon \mathbf{J}\rightarrow \mathbf{J}^0$. First we give some generalities common to both proofs.
\subsubsection*{Step 0}
Without loss of generality we can assume that $r\geq 1$ and that $\epsilon<\min_{v\in S\setminus\{v_0\}}(1,|x_v|_{v})$.
Choose for each finite place $\mathfrak{q}\in \Omega_K^{<\infty}$ a uniformizer $\pi_{\mathfrak{q}}\in K^\times$ at $\mathfrak{q}$. Define the map $\tau\colon \Omega_{K}^{<\infty}\rightarrow \mathbf{J}$ by $(\tau(\mathfrak{q}))_v=\pi_{\mathfrak{q}}$ if $v\neq \mathfrak{q}$ and $(\tau(\mathfrak{q}))_{\mathfrak{q}}=1$. Note that for $\mathfrak{q}\not\in S$, $\tau(\mathfrak{q})\in K^\times \mathbf{J}_{S}$ if and only if $\mathfrak{q}$ is a principal ideal in $\cO(B)$. This is because $\tau(\mathfrak{p})\in K^\times \mathbf{J}_{S}$ means that there exists $u\in K^\times$ such that $u\pi_{\fq}\in \cO_v^\times$ for all $v\in \Omega_K\setminus (S\cup \{\fq\})$ and such that $u\in \cO_{\fq}^\times$, so $(u\pi_\fq)\in B$ is a prime ideal.

\subsubsection*{Step 1 for number fields: $v_0$ infinite}
We assume without loss of generality that $S$ contains all infinite places.
We also first treat the case when $v_0$ is an infinite place.
Define the retraction $\phi\colon \mathbf{J}\rightarrow \mathbf{J}^0$ by
$$\phi(a)_v=\begin{cases}
	a_v & \text{if } v\neq v_0, \\
	a_v/\| a\|^{1/e} &\text{if } v=v_0,
\end{cases}$$
where $e=1$ if $v_0$ is real and $e=2$ if $v_0$ is complex.
We define $\sigma$ to be the composition of $\phi$ with the quotient map $\mathbf{J}^0\rightarrow \mathbf{J}^0/K^\times$. Then $\sigma(\mathbf{J}^0)=\mathbf{J}^0/K^\times$, and $\sigma(K^\times)=1$. Therefore, by \cite[Chapter XV, Theorem 6]{Lan94}, $\Omega_K^{<\infty}$ is $\lambda$-equidistributed in $\mathbf{J}^0/K^\times$, where $\lambda=\sigma\circ \tau$. The map $\tau$ here is defined differently from \cite{Lan94}, but the composition yields the same map $\lambda$ after composing with the inversion map $(\cdot)^{-1}\colon \textbf{J}^0/K^\times\rightarrow \textbf{J}^0/K^\times$ and therefore the conclusion still follows.

Let
\begin{align*}
    U:=&\left\{z\in \mathbf{J}_S^0 \colon  |z_{v}-x_{v}|_{v}<\epsilon\min\left(1,\frac{|x_{v}|_{v}}{2^{r+2}}\right),\, \forall v\in S\setminus\{v_0\}\right\}\\ &\cap \left\{z\in \mathbf{J}_S^0 \colon  |\arg(z_{v_0})-\arg(x_{v_0})|_{v_0}<\frac{\epsilon}{4}\right\},
\end{align*}
is a nonempty open set of $\mathbf{J}^0$, where $r$ is the cardinality of $S$ and $\arg(z)$ is the principal argument of $z\in \mathbb{C}^\times$. Note that $U$ is a nonempty open set of $\mathbf{J}^0$. Denote the image of $U$ in $\mathbf{J}^0/K^\times$ by $\overline{U}$. The maps we have defined together form the following commutative diagram
$$
\begin{tikzcd}
&& U \arrow[d, hook] \arrow[r] & \overline{U} \arrow[d, hook] \\
\Omega_K^{<\infty} \arrow[r, "\tau"] \arrow[rrr, "\lambda", bend right] & \mathbf{J} \arrow[r, "\phi"] \arrow[rr, "\sigma", bend right] & \mathbf{J}^0 \arrow[r]      & \mathbf{J}^0/K^\times.       
\end{tikzcd}
$$

As the indicator function of $\overline{U}$ is integrable, a positive density of prime ideals $\mathfrak{q}\in B$, ordered by their norms, satisfy $\lambda(\mathfrak{q})\in \overline{U}$, by the definition of equidistribution given in \cite[page 316]{Lan94}. Since the image of $\mathbf{J}_S^0$ in $\mathbf{J}^0/K^\times$ is $K^\times\mathbf{J}^0_S/K^\times$, such prime ideals are principal as we have seen in Step 0. 
By the Landau prime ideal theorem \cite[page 670]{Lan03} the number of prime ideals of norm up to $X$ grows asymptotically as $X/\log X$. Therefore the number of prime ideals $\fq\in B$ with $\lambda(\fq)\in \overline{U}$ of norm up to $X$ grows asymptotically as $cX/\log X$ for some constant $c>0$. Therefore, for $N_0\in \mathbb{R}$ sufficiently large, there exist $R$ distinct principal prime ideals $\fq=(q)\in B$ with $\lambda((q))\in \overline{U}$ and of norm $(1-\tfrac{\epsilon}{4})N_0<N(q)<(1+\tfrac{\epsilon}{4})N_0$. Note that $(\phi\circ\tau)(q)_v=uq$ for all $v\in \Omega_K\setminus\{v_0,\fq\}$, $(\phi\circ\tau)(q)_{\fq}=u$ and $(\phi\circ\tau)(q)_{v_0}=uq/N(q)^{1/e}$, where $u\in K^\times\cap \cO_{\fq}^\times$ and $N(q)=\prod_{v\in S} |q|_{v}$ is the norm of $q$ in $\cO(B)$.

By definition of $\overline{U}$, for all prime ideals $(q)$ of $\cO(B)$  with $\lambda((q))\in \overline{U}$, there exists $u'\in K^\times$ such that $u(\phi\circ\tau)(q)\in U$. This implies that $uu'q\in \cO_v^\times$ for all $v\in \Omega_K\setminus\{v_0,\fq\}$, $uu'\in \cO_{\fq}^\times$ and $|\arg((uq/N(q)^{1/e})_{v_0})-\arg(x_{v_0})|_{v_0}<\frac{\epsilon}{4}$. Thus $uu'q\in \cO(B)$, $(uu'q)=(q)$ and $uu'q\in U$ via the natural embedding $K^\times\subset \mathbf{J}^0$, since the argument of $uu'q$ in $K_{v_0}$ is not affected by scaling by a positive real number. Therefore, we can choose the prime element $q\in \cO(B)$ to lie in $U$ itself.

Without loss of generality, we assume $|x_{v_0}|_{v_0}$ is sufficiently large so that $N_0=\prod_{v\in S} |x_v|_{v}$ is large enough for $R$ pairwise coprime elements $q\in U$ to exist. Then we have
\begin{equation} \label{eq: inequalities 1}
(1-\tfrac{\epsilon}{4})|x_{v_0}|_{v_0}< \frac{N(q)}{\prod_{v\in S\setminus\{v_0\}} |x_v|_{v}}<(1+\tfrac{\epsilon}{4})|x_{v_0}|_{v_0},  
\end{equation}
for such $p\in U$.

The triangle inequality implies $\left|\frac{q}{x_{v}}\right|_{v}-1\leq \left|\frac{q}{x_{v}}-1\right|_{v}\leq \left|\frac{q}{x_{v}}\right|_{v}+1$ and by combining this with the definition of $U$ we find $1-\frac{\epsilon}{2^{r+2}}<\frac{|q|_{v}}{|x_{v}|_{v}}<1+\frac{\epsilon}{2^{r+2}}$ for $v\in S\setminus\{v_0\}$, and thus 
\begin{equation} \label{eq: inequalities 2}
  1-\tfrac{\epsilon}{8}<\prod_{v\in S\setminus\{v_0\}} \frac{|p|_{v}}{|x_{v}|_{v}}<1+\tfrac{\epsilon}{8}.  
\end{equation}
By combining the inequalities \eqref{eq: inequalities 1} and \eqref{eq: inequalities 2}, we obtain
$$\left(1-\frac{\epsilon}{2}\right)|x_{v_0}|_{v_0} <|q|_{v_0}<\left(1+\frac{\epsilon}{2}\right)|x_{v_0}|_{v_0}.$$
If $v_0$ is real, then the inequalities on the argument show that
$v_0(q)$ has the same sign as $x_{v_0}$ so $\left|\frac{q}{x_{v_0}}-1\right|_{v_0}<\frac{\epsilon}{2}$. If $v_0$ is complex, then the inequalities on the argument show
$$\left|\frac{q}{x_{v_0}}-1\right|_{v_0}<\left|\left(1+\frac{\epsilon}{2}\right)e^{\epsilon i/4}-1\right|^2<\left(\frac{\epsilon}{2}+\left(1+\frac{\epsilon}{2}\right)\frac{\epsilon}{4}+\sum_{n=2}^\infty \frac{(\epsilon/4)^n}{n!}\right)^2<\epsilon^2< \epsilon.$$
This proves the existence of $R$ coprime prime elements $f=q\in U$ satisfying the second condition of the lemma and by definition of $U$, they also satisfy the first condition. 

\subsubsection*{Step 2 for number fields: $v_0$ finite}
Now we will prove the lemma when $K$ is a number field and $v_0$ is a finite place. We will derive this from the previously treated case when $v_0$ was infinite by choosing an infinite place $v'\in S$ and by letting it play the role of $v_0$ so that we can apply the previously proven case of the lemma. By the generalization of Dirichlet's unit theorem to $S$-integers \cite[Theorem 3.12]{Nar04}, there exists $u\in \cO(B)^\times$ with $|u|_{v}=1$ for all finite places $v\in \Omega_K^{<\infty}\setminus \{v_0\}$ and positive valuation at $v_0$. Therefore by the product formula there exists an infinite place $v'$ such that $|u|_{v'}>1$, and by taking powers of we can take $|u|_{v'}$ to be larger than any given bound.

For $\epsilon'>0$ and any integer $R>0$, if $|u|_{v'}$ is sufficiently large, the part of the lemma proven in Step 1 implies that we can find $R$ pairwise coprime prime elements $q\in \cO_S$ such that $|q-ux_v|_{v}<\epsilon'$ for all places $v\in S\setminus\{v_0,v'\}$ and $|q-ux_{v'}|_{v'}<\epsilon' |ux_{v'}|_{v'}$. If we set $f=q/u$, then $|f-x_v|_{v}<\epsilon'$ for all places $v\in S\setminus\{v_0,v'\}$ and $|f-x_{v'}|_{v'}< \epsilon' |x_{v'}|_{v'}$. As $R$ was arbitrary, this implies that for every $\epsilon>0$ there exist infinitely many pairwise coprime prime elements $f\in \cO(B)$ such that $|f-x_v|_{v}<\epsilon$ for all $v\in S\setminus\{v_0\}$ since we can take $\epsilon'=\epsilon/\max(1, |x_{v'}|_{v'})$.

\subsubsection*{Step 1 for function fields: $v_0$ a rational point}
Now we will prove the statement for global function fields using a similar strategy as for number fields, relying on \cite[Chapter XV, Theorem 6]{Lan94}. While this theorem is formulated for number fields, the statement is true for global fields. This is because the proof of this result relies on Theorems 1, 2, 3 and 5 as well as Proposition 1 in \cite[Chapter XV]{Lan94}. Theorem 1 and Proposition 1 are purely analytic statements, not involving number fields, while Theorem 2, 3 and 5 are true over global fields using the same argumentation as given in the book. As noted in Remark \ref{remark: take geometrically connected}, we can take $k$ to be the field of constants of $C$ so that $C$ is geometrically integral.

We will first assume that $v_0$ is a $k$-rational point and prove the general case afterwards.
Let $l$ be the cardinality of $k$ and define the retraction $\phi\colon \mathbf{J}\rightarrow \mathbf{J}^0$ by
$$\phi(a)_v=\begin{cases}
	a_v & \text{if } v\neq v_0, \\
	a_v/\pi_{v_0}^{\log_l \|a\|} &\text{if } v=v_0,
\end{cases}$$
where $\pi_{v_0}\in \cO_{v_0}$ is a uniformizer and $\|a\|$ is the norm of $a$.
We define $\sigma$ to be the composition of $\phi$ with the quotient map $\mathbf{J}^0\rightarrow \mathbf{J}^0/K^\times$, as in the case for number fields. Then $\sigma(\mathbf{J}^0)=\mathbf{J}^0/K^\times$, and $\sigma(K^\times)=1$.

Therefore we can use \cite[Chapter XV, Theorem 6]{Lan94} as in Step 1 for number fields to conclude that $\Omega_K$ is $\lambda$-equidistributed in $\mathbf{J}^0/K^\times$, where $\lambda=\sigma\circ \tau$.

For any $b\in K_{v_0}^\times$ satisfying $|b|_{v_0}\prod_{v\in S\setminus\{v_0\}} |x_v|_{v}=1$, define the nonempty open subset
\begin{align*}
U_{b}:=\left\{z\in \mathbf{J}_S^0 \colon  \begin{matrix} |z_{v}-x_{v}|_{v}&\!\!\mkern-18mu\quad< \epsilon,\, \forall v\in S\setminus\{v_0\}\\ |z_{v_0}-b|_{v_0}&\!\!\mkern-18mu\quad< \epsilon/\prod_{v\in S\setminus\{v_0\}}|x_v|_{v}\end{matrix}\right\}
\end{align*}
and denote its image in $\mathbf{J}^0/K^\times$ by $\overline{U}_b$.
 As in the proof for number fields, the maps defined fit into the following commutative diagram
$$
\begin{tikzcd}
&& U_b \arrow[d, hook] \arrow[r] & \overline{U}_b \arrow[d, hook] \\
\Omega_K^{<\infty} \arrow[r, "\tau"] \arrow[rrr, "\lambda", bend right] & \mathbf{J} \arrow[r, "\phi"] \arrow[rr, "\sigma", bend right] & \mathbf{J}^0 \arrow[r]      & \mathbf{J}^0/K^\times.       
\end{tikzcd}
$$

Note that $(\phi\circ\tau)(q)_v=uq$ for all $v\in \Omega_K\setminus\{v_0,\fq\}$, $(\phi\circ\tau)(q)_{\fq}=u$ and $(\phi\circ\tau)(q)_{v_0}=uq/\pi_{v_0}^{\log_l N(q)}$, where $u\in K^\times\cap \cO_{\fq}^\times$ and $N(q)=\prod_{v\in S} |q|_{v}$ is the norm of $q$ in $\cO(B)$.

Let $n\geq 1$ be an integer. By the Hasse-Weil bound \cite[Corollary 7.2.1]{Poo17} for $B$ over $\mathbb{F}_{l^n}$, $\cO(B)$ has at least $l^n+O(l^{n/2})$ prime ideals of norm $l^n$, where implied constant depends on $C$ but not on $n$. Thus if $n$ is sufficiently large, then for every $b\in K_{v_0}^\times$ with $|b|_{v_0}\prod_{v\in S\setminus\{v_0\}} |x_v|_{v}=1$ there exist at least $R$ pairwise coprime primes $\fq=(q)$ of norm $l^n$ with $\lambda(\fq)\in\overline{U}_b$. In particular, if $|x_{v_0}|_{v_0}$ is sufficiently large, then there exists at least $R$ pairwise coprime primes $\fq=(q)$ of norm $\prod_{v\in S}|x_v|_v$ with $\lambda(\fq)\in\overline{U}_b$.

By definition of $\overline{U}_b$, for all prime ideals $(q)$ of $\cO(B)$  with $\lambda((q))\in \overline{U}_b$, there exists $u'\in K^\times$ such that $u(\phi\circ\tau)(q)\in U_b$. This implies that for all $uu'q\in \cO_v^\times$ for all $v\in \Omega_K\setminus\{v_0,\fq\}$, $uu'\in \cO_{\fq}^\times$ and $|uu'q/\pi_{v_0}^{\log_l N(q)}-b|_{v_0}<\epsilon/\prod_{v\in S\setminus \{v_0\}}|x_{v}|_{v}$. This implies that $uu'q\in \cO(B)$ and $(uu'q)=(q)$. Furthermore, if $(q)$ has norm $N(q)=\prod_{v\in S}|x_v|_v$, this implies that $|uu'q-b\pi_{v_0}^{\log_l N(q)}|_{v_0}<\epsilon|x_{v_0}|_{v_0}$

In particular, by taking $b=x_{v_0}\pi_{v_0}^{-\log_l \prod_{v\in S} |x_v|_{v}}$ it follows that there exist $R$ coprime prime elements $f=q\in \cO_S$ with $|f-x_v|_{v}<\epsilon$ for $v\in S\setminus\{v_0\}$ and $|f-x_{v_0}|_{v_0}<\epsilon |x_{v_0}|_{v_0}$.

\subsubsection*{Step 2 for function fields: $v_0$ not a rational point}
Now it remains to consider the case where $v_0$ is not a $k$-rational point. By the Hasse-Weil bound, there exists a place $v'\in \Omega_K\setminus S$ such that $\gcd(\deg(v), \deg(v'))=1$ for every $v\in S$. For $v\in S$ choose a factorisation $x_v=c_v^{\deg(v_0)} d_v^{\deg(v')}$ where $c_v,d_v\in K_{v}^\times$ and such that $0\leq -v(d_v)<\deg(v_0)$ and thus $|d_v|_v\geq 1$. Let $\widetilde{k}$ be the splitting field of the closed point $v'$ and let $\widetilde{K}$ be the fraction field of $C_{\widetilde{k}}$. Denote the complement of $S\sqcup\{v'\}$ in $C$ by $B'$. For every place in $S$ there is a unique place $\widetilde{v}\in \Omega_{\widetilde{K}}$ lying above it, by the coprimality assumption. Let $\widetilde{S}$ be the set of places in $\Omega_{\widetilde{K}}$ above the places in $S$.
Every place $\widetilde{v}'\in \Omega_{\widetilde{K}}$ above $v'$ has degree $1$ and thus corresponds to a rational point on $B'_{\widetilde{k}}$. As we already know that the statement is true if $v_0$ is a rational point, we can apply the lemma to $C_{\widetilde{k}}$, where $B'_{\widetilde{k}}$ plays the role of $B$ and $\widetilde{v}'$ plays the role of $v_0$, for some choice of $\widetilde{v}'\in \Omega_{\widetilde{K}}$ above $v'$. Therefore, for every $\epsilon_2>0$ we can find 
$R$ coprime prime elements $\widetilde{q}\in \cO(B'_{\widetilde{k}})$ such that for every place $v\in S$, we have $$|\widetilde{q}-d_v|_{\widetilde{v}}<\epsilon_2.$$ 

By taking $\epsilon_2<\min(|d_v|_{\widetilde{v}},1)$, we ensure that $|\widetilde{q}|_{\widetilde{v}}=|d_v|_{\widetilde{v}}=|d_v|_{v}^{\deg(v')}.$

By the Hasse-Weil bound $\cO(B'_{\widetilde{k}})$ has $l^n+O(l^{n/2})$ prime ideals of norm $l^n$ lying above completely split primes in $\cO(B')$. This is because the Hasse-Weil bound implies that the number of primes ideals in $\cO(B'_{\widetilde{k}})$ of norm $l^n$ lying above primes in $\cO(B')$ which are not completely split is bounded from above by $\sum_{d|n}\left(l^d+O(l^{d/2})\right)$. Therefore the prime element $\widetilde{q}$ can be chosen such that $q_2:=\widetilde{q}\sigma(\widetilde{q})\dots \sigma^{\deg(v')-1}(\widetilde{q})$ is a prime element in $\cO(B')$, where $\sigma$ is a generator of $\gal(\widetilde{k}/k)$. Note furthermore that for all $v\in S$ and $a\in \tilde{K}$, we have $|\sigma(a)|_{\widetilde{v}}=|a|_{\widetilde{v}}$, since $\widetilde{v}$ is the unique place above $v$.
Therefore, by the ultrametric triangle inequality, there exist $R$ coprime prime elements $q_2\in \cO(B')$ such that
\begin{align*}
|q_2-d_v^{\deg(v')}|^{\deg(v')}_{v}&=|q_2-d_v^{\deg(v')}|_{\widetilde{v}}
\\&\leq \max(|\widetilde{q}\sigma(\widetilde{q})\dots \sigma^{\deg(v')-2}(\widetilde{q})d_v-q_2|_{\widetilde{v}},\dots, |d_v^{\deg(v')}-\widetilde{q}d_v^{\deg(v')-1}|_{\widetilde{v}})
\\&
<\epsilon_2\prod_{i=0}^{\deg(v')-1}{|\sigma^i(\widetilde{q})|_{\widetilde{v}}}/\min(|\widetilde{q}|_{\widetilde{v}},\dots, |\sigma^{\deg(v')-1}(\widetilde{q})|_{\widetilde{v}})\\&
= \epsilon_2|\widetilde{q}|_{\widetilde{v}}^{\deg(v')-1}=\epsilon_2|q_2|_{v}^{\deg(v')-1}
\end{align*}
for all $v\in S$.

By the same reasoning, for every $\epsilon_1>0$ and sufficiently large $|c_{v_0}|_{v_0}$, there also exist $R$ pairwise coprime prime elements $q_1$, pairwise coprime to the chosen prime elements $q_2$, with
$$|q_1-c_v^{\deg(v_0)}|^{\deg(v_0)}_{v}<\epsilon_1 |q_1|_{v}^{\deg(v_0)-1}\text{ for all } v\in S$$
and
$$|q_1-c_{v_0}^{\deg(v_0)}|^{\deg(v_0)}_{v_0}<\epsilon_1 |q_1|_{v_0}^{\deg(v_0)},$$
and $|q_1q_2|_{v'}=1$. Hence $f:=q_1 q_2$ is a squarefree element in $\cO(B)$.

Note that for all $v\in S$ the ultrametric triangle inequality implies $$|f-x_v|_{v}=|q_1q_2-c_v^{\deg(v_0)}d_v^{\deg(v')}|_{v}\leq \max(|q_2|_{v} |q_1-c_v^{\deg(v_0)}|_{v}, |q_1|_{v} |q_2-d_v^{\deg(v')}|_{v}).$$ Combining this inequality with the inequalities on $|q_1-c_v^{\deg(v_0)}|_{v}$ and $|q_2-d_v^{\deg(v')}|_{v}$ gives

$$|f-x_v|_{v}\leq |x_v|_{v}\max\left(\epsilon_1|q_1|_{v}^{-1/\deg(v_0)}, \epsilon_2|q_2|_{v}^{-1/\deg(v')}\right)\leq |x_v|_{v}\max\left(\epsilon_1|q_1|_{v}^{-1/\deg(v_0)}, \epsilon_2\right)$$
for $v\in S\setminus\{v_0\}$ and
$$|f-x_{v_0}|_{v_0}\leq |x_{v_0}|_{v_0}\max\left(\epsilon_1, \epsilon_2|q_2|_{v_0}^{-1/\deg(v')}\right)\leq |x_{v_0}|_{v_0}\max(\epsilon_1, \epsilon_2).$$
In particular if we choose $$\epsilon_1=\epsilon/\max_{v\in S\setminus \{v_0\}}(1,|x_v|_{v}|q_1|_{v}^{-1/\deg(v_0)})$$ and
$$\epsilon_2=\epsilon/\max_{v\in S\setminus\{v_0\}}(1,|x_v|_v)$$
then $f$ is a squarefree element in $\cO(B)$ satisfying the desired conditions and by varying the choices for $q_1$ and $q_2$ there are at least $R$ pairwise coprime elements $f$ satisfying the conditions.
\end{proof}

Now we prove the analogous statement for function fields of a curve over an infinite field.
\begin{lemma} \label{Lift function field}
	Let $K$ be a function field of a regular projective curve $C$ over an infinite field $k$ and let $S\subset \Omega_{K}$ be a finite set of places containing a distinguished place $v_0\in S$. For $v\in S$ Let $x_v\in K_{v}^\times$ and let $\epsilon>0$. Then there exist infinitely many pairwise coprime separable elements $f\in \cO(B)$ such that $$|f-x_v|_{v}<\epsilon \text{ for all }v\in S\setminus\{v_0\},$$
 where $B=C\setminus S$.
 Furthermore, if $|x_{v_0}|_{v_0}$ is sufficiently large, depending on $\epsilon$ and $|x_v|_{v}$ for $v\in S\setminus\{v_0\}$, then there exist infinitely many such $f$ which additionally satisfy $$|f-x_{v_0}|_{v_0}<\epsilon |x_{v_0}|_{v_0}.$$
\end{lemma}
\begin{proof}
    For $v\in S$ we write $D_v$ for the divisor on $C$ associated to the place $v$ and $g(C)$ for the genus of $C$. Let $n>0$ be some integer such that $p^{-n}<\epsilon$, where $p=\CHAR(k)$ if $k$ has positive characteristic and $p=2$ if $k$ has characteristic $0$. For an integer $m>0$ and a place $v\in S\setminus\{v_0\}$, we write
    $$\tilde{D}_{v,m}=mD_{v_0}-v(x_v)D_v-\sum_{\tilde{v}\in S\setminus\{v_0,v\}} n D_{\tilde{v}}.$$
    For every place $v\in S\setminus\{v_0\}$, Riemann-Roch \cite[Theorem 7.3.26]{Liu06} implies
    $h^0(\tilde{D}_{v,m})= \deg(\tilde{D}_{v,m})-g(C)+1$ and $h^0(\tilde{D}_{v,m}-D_v)= \deg(\tilde{D}_{v,m})-\deg(D_v)-g(C)+1$, as long as 
    $\deg(\tilde{D}_{v,m})-\deg(D_v)>2g(C)-2$. This inequality is satisfied whenever $m$ is large enough, so for such $m$ we have $$h^0(\tilde{D}_{v,m})-h^0(\tilde{D}_{v,m}-D_v)\geq \deg(D_v)\geq 1.$$
    Therefore there exists an element $\tilde{h}_v\in\cO_C(\tilde{D}_{v,m})\setminus \cO_C(\tilde{D}_{v,m}-D_v)\subset \cO(B)$, which therefore satisfies $v(\tilde{h}_v)=v(x_v)$ and $|\tilde{h}_v|_{\tilde{v}}<\epsilon$ for all $\tilde{v}\in S\setminus\{v_0,v\}$. Furthermore we have $|\tilde{h}_v|_{v_0}\leq p^{-m\deg(D_{v_0})}<\epsilon|x_{v_0}|_{v_0}$ whenever $|x_{v_0}|_{v_0}$ is sufficiently large. The divisor 
    $$\tilde{D}_{v_0}=-v_0(x_{v_0})D_{v_0}-\sum_{v\in S\setminus\{v_0\}} n D_v$$ is very ample if $|x_{v_0}|_{v_0}$ is sufficiently large, as this implies that $-v(x_{v_0})$ is a large positive integer, so in the same manner we construct $\tilde{h}_{v_0}\in \cO(B)$ with $v_0(\tilde{h}_{v_0})=v_0(x_{v_0})$ and $|\tilde{h}_{v_0}|_{v}<\epsilon$ for all $v\in S\setminus\{v_0\}$.
	
    For any $v\in S$ we have the inequality $v(x_v-c_v\tilde{h}_v)<v(x_v)$ for some $c_v\in k^\times$. For $v\in S\setminus\{v_0\}$, by applying the the above construction of $\tilde{h}_v$ to $x_v-c_v\tilde{h}_v$ instead of $x_v$ we iteratively construct $\tilde{h}_{v,1},\dots, \tilde{h}_{v,r}\in \cO(B)$ and $c_{v,1},\dots, c_{v,r}\in k^\times$ such that $|\tilde{h}_{v,1}|_{\tilde{v}},\dots, |\tilde{h}_{v,r}|_{\tilde{v}} <\epsilon$ for all $\tilde{v}\in S\setminus\{v_0,v\}$, $|\tilde{h}_{v,1}|_{v_0},\dots, |\tilde{h}_{v,r}|_{v_0} <\epsilon|x_{v_0}|_{v_0}$ and
    $$v\left(x_v-\sum_{i=1}^{j} c_{v,i}\tilde{h}_{v,i}\right)<v\left(x_v-\sum_{i=1}^{j-1} c_{v,i}\tilde{h}_{v,i}\right)$$
    for all $j\in \{1,\dots, r\}$. By taking $r=v(x_v)+n$ and by setting $$h_v=\sum_{i=1}^r c_{v,i}\tilde{h}_{v,i}\in \cO(B),$$ the ultrametric triangle inequality implies $|h_v|_{\tilde{v}}<\epsilon$ for all $\tilde{v}\in S\setminus\{v,v_0\}$, $|h_v|_{v_0}<\epsilon|x_{v_0}|_{v_0}$ and $|h_v-x_v|<\epsilon$.
    
    In this manner, we also construct $h_{v_0}\in \cO(B)$ with $|h_{v_0}|_{v}<\epsilon$ for all $v\in S\setminus\{v_0\}$ and $|h_{v_0}-x_{v_0}|_{v_0}<\epsilon|x_{v_0}|_{v_0}$. If we write $h=\sum_{v\in S} h_v$ then it follows that $|h-x_v|_{v}<\epsilon$ for all $v\in S\setminus\{v_0\}$ and $|h-x_{v_0}|<\epsilon|x_{v_0}|_{v_0}$. So now we have found an $h\in \cO(B)$ with the desired properties, except for the fact that $h$ need not be separable. We resolve this by slightly perturbing $h$, by adding a function to it with small valuations at the places $v\in S$.
	
In a similar way as before we use Riemann-Roch to construct $g\in \cO(B)$ satisfying $|g|_{v}<\epsilon$ for all $v\in S\setminus\{v_0\}$ and $|g|_{v_0}<\epsilon |x_{v_0}|_{v_0}$, whenever $|x_{v_0}|_{v_0}$ is sufficiently large, such that $h$ and $g$ do not share any zeroes.
Additionally, we construct $g$ such that $\deg(g)$ is not divisible by the characteristic of $k$ and such that $\deg(g)>\deg(h)$. Then the closed subscheme $$X=\{h+tg=0\}\subset B\times_k \mathbb{A}_k^1$$ is integral, since $h$ and $g$ do not share any zeroes. The projection morphism $\pi\colon X\rightarrow \mathbb{A}_k^1$ has degree coprime to the characteristic and is therefore separable. By generic flatness \cite[Tag 052A]{Stacks}, there exists an nonempty open $V\subset \mathbb{A}_k^1$ such that the restriction $X\times_{\mathbb{A}_k^1} V\rightarrow V$ of $\pi$ is flat. Therefore, \cite[Proposition 2.4(1)]{Lui98} implies there exists a nonempty open $U\subseteq \mathbb{A}_k^1$ such that the fiber $\pi^{-1}(c)$ is geometrically regular for all $c\in U$. In particular, since $k$ is infinite, there exist infinitely many $c\in k$ such that for $f=h+cg\in \cO(B)$ the scheme $\div(f)\cap B$ is geometrically regular, which proves that $f$ is separable. Furthermore, any two different choices of $c\in k$ yield functions $f$ and $\tilde{f}$ which are coprime to each other.
\end{proof}
\subsection{Proof of Theorem \ref{theorem: M-approx}}
In this section we prove Theorem \ref{theorem: M-approx}, and thus completely characterize when a toric pair $(X,M)$ satisfies $M$-approximation off a finite set of places $T$. We will treat the cases $T\neq \emptyset$ and $T=\emptyset$ separately.

By Corollary \ref{corollary: M approx toric stable under birational transformations} we can assume without loss of generality that $X$ is smooth. Furthermore, by Proposition \ref{prop: relation M and Mfin} and Proposition \ref{prop: Open subset dense} we can also assume and $M=M_{\fin}$. We additionally assume without loss of generality that $\fM=\fM_{\red}$.

The pair $(X,M)$ satisfies $M$-appproximation off $T$ if and only if for every finite set of places $S$ containing $T\cup\Omega_K^{\infty}$, any choice of a point $Q_v=(q_{v,1}:\dots: q_{v,n})\in X(K_v)$ and any analytic open neighborhood $Q_v\in V_v$ for every $v\in S\setminus T$, there exists a rational point $Q=(q_1:\dots:q_n)\in X(K)$ such that $Q\in(\cX,\cM)(B)$ and $Q\in V_v$ for every $v\in S\setminus T$. Here $B=\Omega_K\setminus S$ and $(\cX,\cM)$ is the toric integral model of $(X,M)$ over $B$. We will write $d$ for the dimension of $X$, $\cU\cong \mathbb{G}_{m,\mathbb{Z}}^d$ for the open torus in $\cX$ and $U$ for its base change to $K$.

\begin{proof}[Proof of sufficiency of the conditions]
We will first show that $(X,M)$ satisfies $M$-approximation off $T$ if $T\neq \emptyset$ and $|N:N_M|\in \rho(K,C)$ or $T=\emptyset$ and $N=N_M^+$. The majority of the proofs of the two cases are the same, with only the last part of the proofs differing. The Cox morphism $\pi\colon \cY\rightarrow \cX$, introduced in Section \ref{section: Cox coordinates}, induces for every $v\in S\setminus T$ a continuous map $\cY(K_v)\rightarrow \cX(K_v)$. Therefore, there exists $\epsilon>0$ such that if for any $Q=(q_1:\dots:q_n)\in X(K)$ is a point such that $|q_i-q_{v,i}|_v<\epsilon|q_{v,i}|_v$ for all $i\in \{1,\dots, n\}$ and $v\in S\setminus T$, then $Q\in V_v$ for all $v\in S\setminus T$.  


We will now show that we can reduce to the case where $\mult_v(Q_v)\in N_M$ for all $v\in S\setminus T$. If $\rho(K,C)=\{1\}$ or $T=\emptyset$, then this is trivially true since then $N_M=N$. In particular, we only need to show this when $K$ is a function field.

If $K$ is a function field, Lemma \ref{lemma: Picard group divisible} implies that $\Pic(C\setminus T)$ is $|N:N_M|$-divisible. This means that for every divisor $D$ on $C\setminus T$, there is $u\in K^\times$ such that $D+\div u=|N:N_M|D'$ for some divisor $D'$ on $C\setminus T$. Therefore, there exists $\mathbf{u}=(u_1:\dots:u_n)$ with $u_1,\dots, u_n\in K^\times$ such that $\mult_v(u_1q_{v,1}:\dots: u_nq_{v,n})\in N_M$ for all $v\in S\setminus T$ and $\mult_v(\mathbf{u})\in N_M$ for all $v\in \Omega_K^{<\infty}\setminus S$. Let $S'$ be the finite set consisting of the places in $S$ together with all places $v$ for which $\mult_v(\mathbf{u})\neq 0$. For $v\in S\setminus T$, we set
$$Q'_v=(u_1q_{v,1}: \dots: u_nq_{v,n})$$
and $V'_v=\mathbf{u} V_v$, and set
$Q'_v=\mathbf{u}$ and $V'_v=\mathbf{u}\cU(\cO_v)$
for $v\in S'\setminus S$, where the multiplication is done coordinate-wise. If there exists $Q'\in (\cX,\cM)(B')$, where $B'=C\setminus S'$, such that $Q'\in V'_v$ for all $v\in S$, then $Q=\mathbf{u}^{-1}Q'\in \cU(\cO_v)$ for all $v\in S'\setminus S$. Therefore $Q$ satisfies $Q\in (\cX,\cM)(B)$ and $Q\in V_v$ for all $v\in S$, as desired.

If $K$ is a number field, we simply let $S'=S$, $Q_v'=Q_v$ and $V'_v=V_v$ for $v\in S\setminus T$.

Choose $\mathbf{m}_1,\dots, \mathbf{m}_l\in \fM$ such that $\phi(\mathbf{m}_1), \dots, \phi(\mathbf{m}_l)$ generate $N_M$ as a lattice. If $T=\emptyset$ and $N_M^+=N$ assume furthermore that they generate $N$ as a monoid. Let $\cG\cong \mathbb{G}_{m,\mathbb{Z}}^{n-d}$ be the torus as in Section \ref{section: Cox coordinates} and let $\pi_v\in \cO_v$ be a uniformizer for all $v\in S\setminus T$.
For a place $v\in S\setminus T$, we are going to construct $c_{\mathbf{m}_1, v},\dots, c_{\mathbf{m}_l,v}\in K_v^\times$ such that we have
\begin{equation} \label{eq: choice good coordinates}
\prod_{s=1}^l (c_{\mathbf{m}_s, v}^{m_{s,1}}:\dots: c_{\mathbf{m}_s, v}^{m_{s,n}})=Q'_{v}=(q'_{v,1}:\dots: q'_{v,n}),    
\end{equation}
in Cox coordinates, where the multiplication is defined coordinate-wise.
This is equivalent to the existence of $(t_1,\dots, t_n)\in \cG(K_v)$ 
for which 
$$
\prod_{s=1}^l (c_{\mathbf{m}_s, v}^{m_{s,1}},\dots, c_{\mathbf{m}_s, v}^{m_{s,n}})=(t_1,\dots, t_n)\cdot(q'_{v,1},\dots, q'_{v,n}),    
$$
where the products are again defined by coordinate-wise multiplication.

By the definition of $\cG$ this is equivalent to
$$\prod_{i=1}^n\left(\left(\prod_{s=1}^l c_{\mathbf{m}_s, v}^{m_{s,i}}\right)/q'_{v,i}\right)^{\langle n_{\rho_i},e_j\rangle}=1$$
for every $j=1,\dots, d$, where $\{e_1,\dots, e_d\}$ is a basis of $N$.
This, in turn is equivalent to
$$\prod_{s=1}^l c_{\mathbf{m}_s, v}^{\langle \phi(\mathbf{m}_s), e_j\rangle}=\prod_{i=1}^n {q'_{v,i}}^{\langle n_{\rho_i},e_j\rangle}$$
for every $j=1,\dots, d$. If we write $\gamma_{j,s}=\langle \phi(\mathbf{m}_s), e_j\rangle$ and $a_j=\prod_{i=1}^n {q'_{v,i}}^{\langle n_{\rho_i},e_j\rangle}$ this equation becomes
$$\prod_{s=1}^l c_{\mathbf{m}_s, v}^{\gamma_{j,s}}=a_j.$$
The $d\times l$ matrix $\Gamma$ with entries $\gamma_{j,s}$ induces a group homomorphism $\Gamma_{K_v}\colon (K_v^\times)^l\rightarrow (K_v^\times)^d\cong U(K_v)$ given by $$(c_1,\dots,c_l)\mapsto \left(\prod_{s=1}^l c_s^{\gamma_{1,s}}, \dots, \prod_{s=1}^l c_s^{\gamma_{d,s}}\right),$$
where the latter isomorphism is given by the choice of basis of $N$.
Since $\phi(\mathbf{m}_1),\dots, \phi(\mathbf{m}_l)$ span $N_M$ and $|N:N_M|\in \rho(K,C)$, this homomorphism restricts to a surjective group homomorphism $(\cO_v^\times)^l\rightarrow (\cO_v^\times)^d$. Since $K_v^\times\cong \mathbb{Z}\times \cO_v^\times$, the image of $\Gamma_{K_v}$ is exactly the points $\widetilde{Q}\in U(K_v)$ for which $\mult_v(\widetilde{Q})\in N_M$. Since $\mult_v(q_{v,1}:\dots:q_{v,n})\in N_M$ for all $v\in S\setminus T$,
for each place $v\in S\setminus T$ we can find $c_{\mathbf{m}_1, v},\dots, c_{\mathbf{m}_l, v}\in K_v^\times$ satisfying condition \eqref{eq: choice good coordinates}.

Now we distinguish between whether $T\neq\emptyset$ or $T=\emptyset$.
If $T\neq \emptyset$, then by Lemma \ref{Lift number field} if $K$ is a global field and by Lemma \ref{Lift function field} if $K$ is another function field, we can find coprime squarefree elements $c_{\mathbf{m}_1},\dots c_{\mathbf{m}_l}\in \cO(B)$ such that $|c_{\mathbf{m}_i}/c_{\mathbf{m}_i, v}-1|_v<\epsilon\left(\sum_{s=1}^l m_{s,i}\right)^{-1}$ and $|c_{\mathbf{m}_i}|_v\leq |c_{\mathbf{m}_i,i}|_v$ for every $v\in S\setminus T$.

Therefore if we take $Q'=(q'_1,\dots,q'_n)\in (\cX,\cM)(B)$ where $$q'_i=\prod_{s=1}^l c_{\mathbf{m}_s}^{m_{s,i}} \in \cO(B),$$
then $|q'_i-q'_{v,i}|_v=|\prod_{s=1}^l c_{\mathbf{m}_s}^{m_{s,i}}-\prod_{s=1}^l c_{\mathbf{m}_s,v}^{m_{s,i}}|_v<\epsilon|q'_{v,i}|_v$. Here we used the elementary fact that for any tuple $a_1,\dots,a_r\in K_v$ with $|a_i|_v\leq 1$, we have  $|\prod_{i=1}^r a_i-1|_v\leq \sum_{i=1}^r |a_i-1|_v$. Therefore \eqref{eq: choice good coordinates} implies that $Q'\in V'_v$ for all $v\in S'\setminus T$, as desired.

Now we assume $T=\emptyset$ and we assume $N=N_M^+$. Without loss of generality we assume that $S$ contains a place $v_0$, which is an infinite place if $K$ is a number field. 
Since $\phi(\mathbf{m}_1),\dots, \phi(\mathbf{m}_l)$ generate $N$ as a monoid, there exist integers $d_1,\dots, d_l>0$ such that $\sum_{i=1}^l d_i \phi(\mathbf{m}_i)=0$. Therefore for the place $v_0\in S$ we can rescale the constants $c_{\mathbf{m}_i,v_0}$ to $\tilde{c}_{\mathbf{m}_i,v_0}:=r^{Cd_i}c_{\mathbf{m}_i,v_0}$ for some integer $C>0$ and $r\in K_{v_0}$ with $|r|_{v_0}>1$, without changing the $K_{v_0}$-rational point defined, since in the torus we have 
\begin{align*}
\prod_{k=1}^l \left(c_{\mathbf{m}_k,v_0}^{m_{k,1}}:\dots: c_{\mathbf{m}_k,v_0}^{m_{k,n}}\right)&=\prod_{k=1}^l \left(\left(r^{Cd_k} c_{\mathbf{m}_k,v_0}\right)^{m_{k,1}}:\dots: \left(r^{Cd_k} c_{\mathbf{m}_k,v_0}\right)^{m_{k,n}}\right)\\
&=\prod_{k=1}^l \left(\tilde{c}_{\mathbf{m}_k,v_0}^{m_{k,1}}:\dots: \tilde{c}_{\mathbf{m}_k,v_0}^{m_{k,n}}\right)
\end{align*}
where the products are defined by the action of the torus on itself.
Therefore for every $\epsilon>0$, by taking $C$ large enough we can apply the stronger form of Lemma \ref{Lift number field} for global fields and Lemma \ref{Lift function field} for other function fields to get a lifting $c_{\mathbf{m}_i}\in \cO(B)$ satisfying
$$|c_{\mathbf{m}_i}/c_{\mathbf{m}_i, v}-1|_v<\epsilon\left(\sum_{s=1}^l m_{s,i}\right)^{-1}$$
for every $v\in S\setminus\{v_0\}$
as before, but with the additional condition that
$$|c_{\mathbf{m}_i}\left(c_{\mathbf{m}_i, v} r^{Cd_i}\right)^{-1}-1|_v<\epsilon\left(\sum_{s=1}^l m_{s,i}\right)^{-1}.$$ 
Thus if we define
$Q'=(q'_1:\dots:q'_n)\in (\cX,\cM)(B)$
where
$$q'_i=\prod_{k=1}^l c_{\mathbf{m}_k}^{m_{k,i}},$$
as before, then $Q'\in V'_v$ for all $v\in S$ when $\epsilon$ is chosen sufficiently small.
Therefore $(X,M)$ satisfies $M$-approximation.
\end{proof}

\begin{proof}[Proof of necessity of $|N:N_M|\in \rho(K,C)$ when $\Pic(C)$ is finitely generated and $T\neq \emptyset$]
Now we will prove that if $\Pic(C)$ is finitely generated and $|N:N_M|\not \in \rho(K,C)$, then $(X,M)$ does not satisfy $M$-approximation off $T$ for any finite set of places $T\subset \Omega_K$.
Since $\Pic(C)$ is finitely generated, we can find a finite set of places $S\subset \Omega_K$ containing the infinite places such that $B=\Omega_K\setminus S$ satisfies $\Pic(B)=1$. We will show that this implies that for every finite set of places $T\subset S$, the toric integral model $(\cX,\cM)$ does not satisfy integral $\cM$-approximation off $T$. By Proposition \ref{prop: inclusion M approx}, this in turn implies that $(X,M)$ does not satisfy $M$-approximation off $T$. Since every finite set of places $T$ is contained in a subset $S$ with $\Pic(B)=1$, $(X,M)$ does not satisfy $M$-approximation off $T$ for every finite set of places $T$.

Let $\mathbf{m}_1,\dots, \mathbf{m}_l$ generate $N_M$ as before in the proof of the sufficiency of the condition. Additionally, we first assume that $|N:N_M|$ is not a power of the characteristic of $K$.
By Proposition \ref{prop: Compatibility multiplicity with fan}, the image of 
the set $(\cX,\cM)(B)\subset X(K)$ is contained in the image of the map
$$\cU(B)\times \left(\cO(B)\setminus \{0\}\right)^l\rightarrow X(K)$$
given by
$$((u_1:\dots:u_n),(a_1,\dots,a_l))\mapsto \left(u_1\prod_{i=1}^l a_i^{m_{i,1}}:\dots: u_n\prod_{i=1}^l a_i^{m_{i,n}}\right),$$
where $\cU\cong \mathbb{G}_{m,\mathbb{Z}}^d$ is the open torus in $\cX$.
In particular, if $(\cX,\cM)$ were to satisfy integral $\cM$-approximation off $T$, then the induced map
$$g_{S'}\colon \cU(B)\times \prod_{v\in S'} (k_v^\times)^l\rightarrow \prod_{v\in S'} \cU(k_v)$$
would be surjective for any finite set of places $S'\subset \Omega_K\setminus S$. 

As in the proof of the sufficiency of $N=N_M$ for $M$-approximation off $T$, let $\{e_1,\dots,e_d\}$ be a basis of $N$ and $\Gamma$ be the $d\times l$ matrix with coefficients $\gamma_{j,s}=\langle \phi(\mathbf{m}_s),e_j \rangle$. The isomorphism $N\cong \mathbb{Z}^d$ induced by the choice of this basis induces an isomorphism $\cU(k_v)\cong (k_v^\times)^d$ for all $v\in \Omega_K$. Under this isomorphism, the homomorphism $(k_v^\times)^l\rightarrow \cU(k_v)$
$$(a_1,\dots,a_l)\mapsto \left(\prod_{i=1}^l a_i^{m_{i,1}}:\dots: \prod_{i=1}^l a_i^{m_{i,n}}\right)$$
is given by the homomorphism $\Gamma_{k_v}\colon (k_v^\times)^l\rightarrow (k_v^\times)^d$ induced by the matrix $\Gamma$. Since $|N:N_M|\not\in \rho(K,C)$, $\Gamma_{k_v}$ is not surjective for infinitely many choices of $v\in \Omega_K$.
If $K$ is a global field, then $\cU(B)$ is finitely generated, say by $t$ elements. Then $g_{S'}$ cannot be surjective for any finite set of places $S'$ containing strictly more than $t$ places $v\in \Omega_K\setminus S$ for which $\Gamma_{k_v}$ is not surjective.

Now assume that $K$ is a function field of a curve $C$ over a field $k$.
If $g_{S'}$ is surjective, then the induced homomorphism $$\{f\in \cU(B)\mid f(v')=(1:\dots:1)\}\times \prod_{v\in S'\setminus \{v'\}}(k_v^\times)^l\rightarrow \prod_{v\in S'\setminus\{v'\}} \cU(k_v)$$ is also surjective, where $f(v')$ is the image of $f$ in $\cU(k_{v'})$.
We now give an analogous argument as for global fields. The group $\{f\in \cU(B)\mid f(v')=(1:\dots:1)\}$ injects into $\cU(B)/\cU(k)$, since if $f(v')=(1:\dots:1)$ and $g\in \cU(k)\setminus\{(1:\dots:1)\}$, then $(f\cdot g)(v')\neq (1:\dots:1)$.
Therefore, since $\cU(B)/\cU(k)\cong \left(\cO(B)^\times/\cO(k)^\times\right)^d$ is finitely generated, the group $\{f\in \cU(B)\mid f(v')=(1:\dots:1)\}$ is finitely generated as well. Suppose that it is generated by $t$ elements, then $g_{S'}$ cannot be surjective as soon as it contains strictly more than $t+1$ places $v$ for which $\Gamma_v$ is not surjective.

Finally, if $|N:N_M|$ is a power of the characteristic of $K$, the argument given above might not work, since in this case $\Gamma_{k_v}$ could be surjective for all $v\in \Omega_K$. However there are still infinitely many places $v\in \Omega_K$ such that the map $(\cO_v^\times)^l\rightarrow (\cO_v^\times)^d$ induced by $\Gamma$ is not surjective. Therefore, the argument is easily amended by replacing the role of $k_v$ by $\cO_v/\pi_v^{n_v}$ for these places $v$, where $\pi_v$ is a uniformizer and $n_v$ is the least positive integer such that the homomorphism $((\cO_v/\pi_v^{n_v})^\times)^l\rightarrow ((\cO_v/\pi_v^{n_v})^\times)^d$ is not surjective.

\end{proof}
\begin{proof}[Proof of necessity of $N=N_M^+$ when $T=\emptyset$.]
Now we will show that if $(X,M)$ satisfies $M$-approximation, then $N_M^+=N$. We argue by contradiction and assume $N_M^+\neq N$. Then we have $N_M\neq N$ or $N_{M,\mathbb{R}}^+\neq N\otimes_{\Z} \mathbb{R}$, where $N_{M,\mathbb{R}}^+$ is the convex cone generated by  $N_M^+$. This follows from the fact that $N_{M,\mathbb{R}}^+= N\otimes_{\Z} \mathbb{R}$ implies that $N_M^+$ contains a lattice of finite index in $N$, so combined with $N_M=N$ this gives $N_M^+=N$.

First we assume $N\neq N_M$. If $K$ is a global field, then we have seen that this assumption implies that $(X,M)$ does not satisfy $M$-approximation off $T$ for $T$ nonempty, since then $\rho(K,C)=1$. Thus we assume that $K$ is a function field of a curve, and we consider the map $U(K)\rightarrow N$ given by $P\mapsto \sum_{v\in \Omega_K} \phi_v(P)$, where $\phi_v$ is defined as in \eqref{eq: phi_v}. By Proposition \ref{prop: Compatibility multiplicity with fan} this map is identically zero, as $\deg\div(f)=0$ in $\Pic(C)$ for any $f\in K^\times$. Furthermore, if $P\in (\cX,\cM)(\cO_v)$ for some place $v\in \Omega_K^{<\infty}$, then $\phi_v(P)\in N_M$ by definition of $N_M$. In particular, if $\phi_{v'}(Q_{v'})=0$ for some place $v'\in S$ and $a\in N\setminus N_M$, and $\phi_v(Q_v)=0$ for all $v\in S\setminus\{v'\}$, then these points cannot be all simultaneously well approximated by some $Q\in (\cX,\cM)(B)$. This is because the equality $\phi_v(Q)=\phi_v(Q_v)$ for all $v\in S$ is incompatible with $\sum_{v\in \Omega_K} \phi_v(Q)=0$. Thus $(X,M)$ does not satisfy $M$-approximation.

Now we assume $N_{M,\mathbb{R}}^+\neq N\otimes_{\Z} \mathbb{R}$. The cone $N_{M,\mathbb{R}}^+$ is contained in some half space $H$ of $N\otimes_{\mathbb{Z}}\mathbb{R}$, which is given as $$H=\left\{\sum_{i=1}^d x_ie_i\middle\vert \sum_{i=1}^d a_ix_i\geq 0\right\}\subset N\otimes_{\mathbb{Z}}\mathbb{R},$$ where the $e_1,\dots,e_d$ form a basis of $N$ and $a_1,\dots, a_d\in \mathbb{R}$, not all zero.

Let $S$ be a nonempty set of places containing $\Omega_K^\infty$ and let $Q\in (\cX,\cM)(B)$. Under the isomorphism $\cU(K)\cong (K^\times)^d$ induced by the choice of basis of $N$ given above, we can write $Q=(x_1,\dots,x_d)\in (K^\times)^d$.
Since $Q\in (\cX,\cM)(B)$, we have by Proposition \ref{prop: Compatibility multiplicity with fan} that $\sum_{i=1}^d a_i v(x_i)\geq 0$ for every $i\in\{1,\dots, d\}$ and every place $v\in \Omega_K\setminus S$. Furthermore, the product formula gives $\prod_{v\in \Omega_K}|x_i|_v=1$ and thus we see $\prod_{v\in S}\prod_{i=1}^d |x_i|_v^{a_i}\geq 1$. However for $v\in S$, we can consider points $Q_v\in X(K_v)$ which are sent to $(x_{v,1},\dots, x_{v,d})\in \left(K_v^\times\right)^d$ under the isomorphism induced by $N$, such that $\prod_{v\in S}\prod_{i=1}^d |x_{v,i}|_v^{a_i}<\frac{1}{2}$. Such a tuple $(Q_v)_{v\in S}$ cannot lie in the closure of the map $(\cX,\cM)(B)\rightarrow \prod_{v\in S} U(K_v)$, and hence $(X,M)$ does not satisfy $M$-approximation.
\end{proof}

\subsection{Integral \texorpdfstring{$\cM$}{M}-approximation on toric varieties}
For completeness, we also characterize when integral $\cM$-approximation holds on a toric variety.
\begin{proposition} \label{prop: integral M-approximation on toric}
Let $X$ be a complete normal split toric variety and let $T\subset\Omega_K$ be a finite set of places. If $(X,M)$ satisfies $M$-approximation off $T$, then the toric integral model $(\cX,\cM)$ satisfies integral $\cM$-approximation off $T$ if and only if $\fM_{\text{red}}$ is contained in the closure of $\fM_{\fin}$.
\end{proposition}
\begin{proof}
By Proposition \ref{prop: inclusion M approx}, $(\cX,\cM)$ satisfies integral $\cM$-approximation off $T$ if and only if $(\cX,\cM_{\fin})(\cO_v)$ is dense in $(\cX,\cM)(\cO_v)$ for all $v\in \Omega_K\setminus T$ and by Corollary \ref{corollary: M approx toric stable under birational transformations} we can assume $X$ is smooth. By Proposition \ref{prop: continuity multiplicity} $\mult_v$ is continuous, so a point $P\in(\cX,\cM)(\cO_v)$ can only lie in the closure of $(\cX,\cM_{\fin})(\cO_v)$ if $\mult_v(P)$ lies in the closure of $\fM_{\fin}$. Since the image of $(\cX,\cM)(\cO_v)$ under $\mult_v$ is $\fM_{\text{red}}$, this shows that integral $\cM$-approximation off $T$ can only hold if $\fM_{\text{red}}$ is contained in the closure of $\fM_{\fin}$. Conversely, if $P=(a_1\pi^{m_1}:\dots:a_n \pi^{m_n})$ with $\pi\in \cO_v$ a uniformizer, $a_1,\dots,a_n\in\cO_v^\times$, and $(m_1,\dots,m_n)\in \overline{\fM_{\fin}}$, we can choose a sequence $((m_{1,j},\dots, m_{n,j}))_{j\in \mathbb{N}}$ in $\fM_{\fin}$ converging to $(m_1,\dots,m_n)$. By setting $P_j=(a_1\pi^{m_{1,j}}:\dots:a_n \pi^{m_{n,j}})\in (\cX,\cM_{\fin})(\cO_v)$ we obtain a sequence converging to $P$, finishing the proof.
\end{proof}
\section{The \texorpdfstring{$\cM$}{M}-Hilbert property and split toric varieties}
While the previous section considered the situation where $\cM$-points are plentiful, in this section we will consider when the set of such points is thin.
We will also investigate the different degrees of thinness that these sets have. For example, by Dirichlet's unit theorem \cite[Theorem 3.12]{Nar04}, $\mathbb{G}_m(\cO_K)$ is finitely generated if $K$ is a number field, while the set of squares in $\mathbb{G}_m(K)$ is not finitely generated as a group. Therefore, the former can be thought of as `thinner' than the latter. We introduce several variants of thinness, which allows us to make this idea precise.
\begin{definition} \label{def: d-thin set}
	Let $X$ be an integral variety over $K$, let $A\subseteq X(K)$ and let $d>1$ be an integer. We say that $A$ is of \textit{type II($d$)} if there is an integral variety $Y$ with $\dim Y=\dim X$ and a generically finite morphism $f\colon Y\rightarrow X$ of degree $\geq d$ such that $A\subseteq f(Y(K))$.
    We say that $A$ is $d$-\textit{thin} if it is a finite union of sets of type I and II($d$), where type I is defined as in Definition \ref{def: thin set}. We say that $A$ is \textit{strictly} $d$-\textit{thin} if the morphisms $f$ have degree exactly $d$.
\end{definition}
We also introduce a notion of thinness which is preserved under taking inverse images of dominant morphisms.
\begin{definition}
	Let $K$ be a field, $X$ an integral variety and let $A\subseteq X(K)$ be a subset. Then we say that $A$ is \textit{stably thin} if for every dominant morphism $f\colon Y\rightarrow X$ of integral varieties over $K$, $f^{-1}A\subseteq Y(K)$ is thin.
\end{definition}
This property is preserved under many operations. For example, in the above situation, $f^{-1}A$ is also stably thin. A stably thin set can be viewed as a sort of ``$\infty$-thin'' set, as the next proposition shows.
\begin{proposition} \label{prop: stably thin}
	Let $K$ be a field, $X$ an integral variety and let $A\subseteq X(K)$ be a subset. Then $A$ is stably thin if and only if it is $d$-thin for every $d>1$.
\end{proposition}
\begin{proof}
We first prove by induction that if $A$ is stably thin, then it is $d$-thin for every $d>1$. Assume that $A$ is stably thin and that $A$ is $d$-thin for some $d>1$. Then there is an integer $n>1$ and for each $i\in \{1,\dots,n\}$ a morphism $f_i\colon Y_i\rightarrow X$ of integral $K$-varieties of degree at least $d$ such that $A\setminus A'\subseteq \bigcup_{i=1}^n f_i(Y_i(K))$, for some subset $A'\subset A$ which is not Zariski dense in $X$. Since $A$ is stably thin, $B_i\colon =f_i^{-1}(A)$ is thin for each $i\in \{1,\dots,n\}$, so we see that $A$ is $2d$-thin. Since any stably thin set is $2$-thin, this implies by induction that any stably thin set is $d$-thin for every integer $d>1$.
	
Conversely, assume $A$ is $d$-thin for every $d>1$ and let $g\colon Z\rightarrow X$ be a dominant morphism of integral varieties over $K$. Then the algebraic closure of the function field $K(X)$ in $K(Z)$ is a finite extension of degree $l$. For a generically finite morphism $f\colon Y\rightarrow X$ of degree $d>l$, and for every irreducible component $Y'$ of $(Y\times_X Z)_{\red}$, consider the induced morphism $Y'\rightarrow Z$. This is a generically finite morphism, since the restriction $f^{-1}U\rightarrow U$ is finite for some dense open $U\subset X$ so $(f^{-1}U\times_X Z)_{\red}\rightarrow Z\times_X U$ is finite. The degree of $Y'\rightarrow Z$ is at least $2$, since $K(Y')$ contains a finite extension of $K(X)$ of degree $d$, which cannot be contained in $K(Z)$.

Since $A$ is $d$-thin, there exists an integer $n$, a subset $A'\subseteq A$ which is not Zariski dense and for $i\in\{1,\dots, n\}$ a dominant generically finite morphism $f_i\colon Y_i\rightarrow X$ of integral $K$-varieties such that $A\setminus A'\subseteq \bigcup_{i=1}^n f_i(Y_i(K))$ and each $f_i$ has degree at least $d>l$. Denote the irreducible components of $(Y_i\times_X Z)_{\red}$ by $Y'_{ij}$ and the induced morphisms to $Z$ by $f_{ij}\colon Y'_{ij}\rightarrow Z$. The set $g^{-1}A'$ is not Zariski dense in $Z$, and by construction $g^{-1}(A\setminus A')\subseteq \bigcup_{i,j} f_{ij}(Y'_{ij}(K))$. Thus we see that $g^{-1}A$ is thin and therefore $A$ is stably thin.
\end{proof}
Rational points on abelian varieties and integral points on tori give examples of stably thin sets, as the next example shows.
\begin{example}
Let $A$ be a finitely generated subgroup of $G(K)$ for an semiabelian variety $G$ of positive dimension over a field $K$. For example, $A=G(K)$ when $K$ is a global field and $G$ is an abelian variety. Then for any integer $d>1$, the group $A/dA$ is finite. Let $a_1,\dots,a_n\in A$ be a set of representatives for the classes in $A/dA$. If $a\in dA+a_i$ for some $i\in\{1,\dots,n\}$, then it is the image of a $K$-point under the morphism $G\rightarrow G$ given by multiplication by $d$ followed by translating by $a_i$. This morphism has degree divisible by $d$, so $A$ is $d$-thin. As $d$ was arbitrary, this implies that $A$ is stably thin.
\end{example}
This notion gives a way to formalize the intuition that for a number field $K$ there are more integer squares in $\mathbb{G}_m(K)$ than integral units $\mathbb{G}_m(\cO_K)\subset \mathbb{G}_m(K)$. The former is thin, but probably not $3$-thin, while the latter is stably thin.

\begin{theorem} \label{theorem: thinness}
Let $(K,C)$ be a PF field such that $\Pic(C)$ is finitely generated and let $d>1$ be an integer. If $K$ is a function field assume that $(k^\times)/(k^\times)^d$ is finite.
Let $(X,M)$ be a toric pair where $X$ is a normal complete split toric variety over $K$ and let $(\cX,\cM)$ be any integral model over $B$. Let $N_M$ be the lattice as in Definition \ref{def: N_M singular}. Then
	\begin{enumerate}
		\item if $N_M$ has finite index in $N$ and $d$ divides $|N:N_M|$, then $(\cX,\cM)(B)\subset X(K)$ is strictly $d$-thin.
		\item if $N_M$ does not have finite index in $N$, then $(\cX,\cM)(B)$ is stably thin.
		\item in the function field case, if $(\cX,\cM)$ is the toric integral model and $N_M^+\neq N_M$, then $(\cX,\cM)(C)\subset X(K)$ is stably thin.
	\end{enumerate}
Furthermore, if $(\cX,\cM)$ is the toric integral model and $\mathbb{G}_m(B)$ is finite, then $(\cX,\cM)(B)$ is not Zariski dense in $X$ if and only if $(\cX,\cM)(B)$ is stably thin. 
\end{theorem}
\begin{remark}
	If $(K,C)$ is a function field of a curve and $d\in \rho(K,C)$, then $(k^\times)/(k^\times)^d$ is trivial or has order $2$.
	If $k$ is perfect of characteristic $p$, then $(k^\times)^p=k^\times$.
\end{remark}
\begin{remark}
	If the group $\mathbb{G}_m(B)$ is infinite, then the points on the toric integral model $(\cX,\cM)(B)$ are always Zariski dense, so Theorem \ref{theorem: thinness} completely characterizes when the $\cM$-points on a toric integral model are Zariski dense. Furthermore, the group $\mathbb{G}_m(B)$ is finite if and only if $K=\mathbb{Q}$ or an imaginary quadratic number field and $B=\spec \cO_K$, or $K=k(C)$ for $k$ a finite field and $C\setminus B$ contains at most one point.
\end{remark}
\begin{proof}
    Note that in the first two statements $B$ can be chosen as small as we want, since $B'\subset B$ implies $(\cX,\cM)(B')\supset(\cX,\cM)(B)$. Thus it follows from Proposition \ref{prop: spread out models} that we can assume without loss of generality that $(\cX,\cM)$ is the toric integral model, and for the first two statements we can assume that $\Pic(B)$ is trivial. We can also assume that $M=M_{\fin}$ without loss of generality.
    Since for a toric resolution of singularities $f\colon Y\rightarrow X$ with $f^{-1}D_i$ Cartier for all $i=1,\dots, n$, the set $(\cX,\cM)(B)$ is thin if and only if $(\cY,f^*\cM)(B)$ is thin, we can assume that $X$ is smooth.
	
	
Assume that $N_M\neq N$, $|N:N_M|$ is finite, $d>1$ divides $|N:N_M|$ and $K$ is a global field or $(k^\times)/(k^\times)^d$ is finite. Then $\cU(B)/\cU(B)^d$ is finite, where $\cU=\mathbb{G}^{\dim X}_{m,B}$, since $\cO(B)^\times$ is finitely generated when $K$ is a global field and $\cO(B)^\times/k^\times$ is finitely generated when $K$ is a function field. There exists a lattice $N'$ such that $N_M\subseteq N' \subseteq N$ such that $|N:N'|=d$ and so, using the surjectivity of $\phi$ proven in Proposition \ref{prop: Compatibility multiplicity with fan}, we can choose $M\subseteq M'$ such that $|N:N_{M'}|=d$. The inclusion $M\subset M'$ implies $(\cX,\cM)(B) \subset(\cX,\cM')(B)$, and thus it suffices to consider the case $d=|N:N_M|$. By intersecting the fan of $X$ with $N_M$, we get a new complete normal split toric variety $X_M$ and a degree $d$ morphism $X_M\rightarrow X$. By a resolution of singularities $X'\rightarrow X_M$, we find a degree $d$ morphism $f\colon X'\rightarrow X$ of smooth complete split toric varieties. By Proposition \ref{prop: Compatibility multiplicity with fan} it follows that for every place $v\in B$ and point $P\in (\cX,\cM)(\cO_v)$, there exists $P'\in (\cX',f^* \cM)(\cO_v)\subseteq X'(K_v)$ such that $\phi_v(P)=\phi_v(f(P'))$. Since we assumed $\Pic(B)$ is trivial, this means that for all $P\in (\cX,\cM)(B)$ there exists $P'\in X'(K)$ such that $\phi_v(P)=\phi_v(f(P'))$ for all $v\in B$. Thus we see by Proposition \ref{prop: Compatibility multiplicity with fan} that the image of $f$ contains an element from every $\cU(B)$-orbit in $(\cX,\cM)(B)$. The image of $\cU'(B)$ in $\cU(B)$, where $\cU'$ is the torus in $\cX'$, contains $\cU(B)^d\subset \cU(B)$ and by the assumption on $d$ we know that $\cU(B)/\cU(B)^d$ is finite. Therefore $H=\cU(B)/f(\cU'(B))$ is finite, and by choosing representatives $u_1,\dots, u_r\in \cU(B)$ for $H$, we find degree $d$ morphisms $$f_i\colon X'\rightarrow X$$ defined by $P'\mapsto u_if(P')$ for $P'\in X'(K)$, so that every point in $(\cX,\cM)(B)$ lies in the image of one of the $f_i$. This proves the first statement.

Now assume that $N_M$ does not have finite index in $N$. 
Then there exists an embedding $N_M\rightarrow \mathbb{Z}^{\dim X-1}\times\{0\}\subset \mathbb{Z}^{\dim X}\cong N$. By Proposition \ref{prop: Compatibility multiplicity with fan}, this implies that there is an embedding $(\cX,\cM)(B)\subset \mathbb{G}^{\dim X-1}_m(K)\times \mathbb{G}_m(B)\subset \mathbb{G}^{\dim X}_m(K)$. Since $\mathbb{G}_m(B)$ is finitely generated if $K$ is a global field, and otherwise $\mathbb{G}_m(B)/k^\times$ is finitely generated, $\mathbb{G}_m(B)/\mathbb{G}_m(B)^{d^l}$ is finite for every integer $l>0$. Thus $\mathbb{G}_m(B)$ is $d^l$-thin for every $l>0$ and therefore by Proposition \ref{prop: stably thin} implies that it is a stably thin subset of $\mathbb{G}_m(K)$. Therefore $(\cX,\cM)(B)$ is also a stably thin subset of $X(K)$. If furthermore $\mathbb{G}_m(B)$ is finite and thus not Zariski dense in $\mathbb{G}_m$, then $(\cX,\cM)(B)$ is not Zariski dense in $X$.

Now we prove the third statement, so we assume that $B=C$ and $N_M^+\neq N_M$. Let $N'$ be the largest lattice contained in $N_M^+$. $N'$ does not have finite index in $N$ since otherwise the cone generated by $N_M^+$ would be $N_{\mathbb{R}}$, which implies $N_M^+=N_M$. Since by the product formula any $P\in U(K)$ satisfies $\sum_{v\in \Omega_K} \phi_v(P)=0$, any $P\in (\cX,\cM)(C)$ satisfies $\phi_v(P)\in N'$ for all $v\in \Omega_K$. Let $(X,M')$ be the largest toric pair contained in $(X,M)$ such that $\phi_v(P)\in N'$ for all $P\in (\cX,\cM')(\cO_v)$. Then $(\cX,\cM')(C)$ contains $(\cX,\cM)(C)$ and it is a stably thin subset of $X(K)$ since $N'=N_{M'}$ does not have finite index in $N$. If $\mathbb{G}_m(C)$ is finite, then the same reasoning implies that $(\cX,\cM)(C)$ is not Zariski dense in $X$.


\end{proof}
\begin{remark}
The assumption on the finiteness of $(k^\times)/(k^\times)^d$ in Theorem \ref{theorem: thinness} is satisfied by many fields, such as
\begin{enumerate}
    \item $d$-closed fields, such as separably closed fields if $\CHAR(k)\nmid d$,
    \item perfect fields if $d$ is a power of $\CHAR(k)$,
    \item finite fields,
    \item real closed fields,
    \item local fields if $\CHAR(k) \nmid d$ \cite[Chapter I, Section 1, Proposition 5]{CaFr67},
    \item Euclidean fields if $d$ is a power of $2$.
\end{enumerate}
\end{remark}

\begin{remark}
    The generically finite morphisms used in the proof of Theorem \ref{theorem: thinness} are ramified since any complete toric variety is geometrically simply connected \cite[Expos\'e XI, Corollaire 1.2]{SGA1}. Therefore the thin sets in the theorem are strongly thin as defined in \cite{BFP23}.
\end{remark}

We now prove Corollary \ref{corollary: M-approx equiv to M-Hilbert} by specializing Theorem \ref{theorem: thinness} to global fields.
\begin{proof}[Proof of Corollary \ref{corollary: M-approx equiv to M-Hilbert}]
    First assume $T\neq \emptyset$. By Theorem \ref{theorem: M-approx}, the toric pair $(X,M)$ satisfies $M$-approximation off $T$ if and only if $N=N_M$. By Theorem \ref{theorem: M-approx implies not thin} this implies the $\fM$-Hilbert property over $B$ for any integral model $(\cX,\cM)$ of $(X,M)$ satisfying $(\cX,\cM)(B)\neq \emptyset$. On the other hand Theorem \ref{theorem: thinness} implies that $(\cX,\cM)(B)$ is thin if $N\neq N_M$, so $(\cX,\cM)(B)$ is thin if and only if $N\neq N_M$.

    Now we assume $T=\emptyset$. Then Theorem \ref{theorem: M-approx} implies that the toric pair $(X,M)$ satisfies $M$-approximation if and only if $N_M^+=N$. If $(\cX,\cM)$ is the toric integral model of $(X,M)$ and $N\neq N_M^+$, then Theorem \ref{theorem: thinness} implies that $(\cX,\cM)(B)$ is stably thin.
\end{proof}

In Theorem \ref{theorem: thinness}, the condition on the finiteness of $(k^\times)/(k^\times)^d$ is necessary for the result to be true, as the next proposition shows that the $\cM$-Hilbert property is always satisfied on a split toric variety for a function field $K=k(C)$, where $k$ is a Hilbertian field of characteristic zero. See \cite[Chapter 12]{FrJa05} for background on Hilbertian fields.

\begin{proposition}
Let $(K,C)$ be a PF field, where $C$ is a curve over an Hilbertian field $k$ of characteristic zero. Let $X$ be a proper integral variety over $k$ with $X(k)$ not thin and let $\cX=X\times_k C$. Let $(X,M)$ be a pair over $(K,C)$ such that $D_\alpha\neq X$ for all $\alpha\in \cA$. Then the set $(\cX,\cM^c)(C)\subset X(K)$ is not thin.
\end{proposition}
\begin{proof}
    As the closed subschemes $D_\alpha$ are all proper closed subsets of $X$, $U=X\setminus_{\alpha\in \cA} D_\alpha$ is a dense open. The constant sections of $U\times C\rightarrow C$ correspond to the $k$-rational points on $U$. These are $\cM$-points, as they avoid the closed subschemes $\cD_\alpha=D_\alpha\times_k C$ entirely so $U(k)\subset (\cX,\cM^c)(C)$. By \cite[Theorem 1.1]{BFP23}, $U(k)$ is not a thin subset of $X(K)$ since $K/k$ is a finitely generated extension. Thus $(\cX,\cM^c)(C)$ is not thin.
\end{proof}

Using Theorem \ref{theorem: M-approx} and Theorem \ref{theorem: thinness}, we can produce examples over some PF fields where integral $\cM$-approximation does not imply the $\cM$-Hilbert property, in contrast to the situation over global fields.
\begin{corollary} \label{corollary: M-approx does not imply M-Hilbert}
    Let $(K,C)$ be a PF field with $\rho(K,C)\neq 1$, let $T\subset \Omega_K$ be a nonempty finite set of places, and set $B=C\setminus (T\cap C)$. Let $(X,M)$ be a toric pair with $M=M_{\fin}$ and $|N:N_M|\in \rho(K,C)$, $N_M\neq N$. Then the toric integral model $(\cX,\cM)$ satisfies integral $\cM$-approximation off $T$, but $(\cX,\cM)(B)$ is thin.
\end{corollary}
\begin{proof}
Combine Theorem \ref{theorem: M-approx} and Theorem \ref{theorem: thinness} together with the observation that the ground field of $C$ satisfies $k^\times=(k^\times)^{|N:N_M|}$ by Lemma \ref{lemma: computation rho} since $|N:N_M|\in \rho(K,C)$.
\end{proof}
For example, for any smooth split toric variety $U$ over $K$ such that $\Pic(U)$ contains torsion, the $B$-integral points on the toric integral model $\cU$ are thin for any nonempty open $B\subset C$. This follows from combining Theorem \ref{theorem: thinness} with Proposition \ref{prop: torsion-free Picard}. However, as we will see in Corollary \ref{corollary: strong approx and Pic}, $U$ still satisfies strong approximation off any single place if the orders of torsion in $\Pic(U)$ are contained in $\rho(K,C)$ and $\cO(X_{\overline{K}})^\times=\overline{K}^\times$.

Corollary \ref{corollary: M-approx does not imply M-Hilbert} leaves several natural questions on potential extensions of Theorem \ref{theorem: M-approx implies not thin}.
\begin{question} \label{question: rho=1 gives Hilbert property}
    Let $(K,C)$ be a PF field with $\rho(K,C)=\{1\}$. Let $(X,M)$ be a pair over $(K,C)$ with integral model $(\cX,\cM)$ over $B\subset C$, such that $X$ is a geometrically integral variety and $D_{\alpha}\neq X$ for any $\alpha\in \cA$. Suppose that $(\cX,\cM)$ satisfies integral $\cM$-approximation off a finite set of places $T\subset \Omega_K$ and $(\cX,\cM)(B)\neq \emptyset$. Does $(\cX,\cM)$ satisfy the $\cM$-Hilbert property over $B$?
\end{question}
In order to obtain pairs satisfying integral $\cM$-approximation but failing the $\cM$-Hilbert property in Corollary \ref{corollary: M-approx does not imply M-Hilbert}, we needed to take $T\neq \emptyset$, so it also makes sense to ask the following variant of the previous question.
\begin{question}
    Let $(K,C)$ be a PF field with $\rho(K,C)\neq \{1\}$ and let $T=\emptyset$. With the other assumptions as in Question \ref{question: rho=1 gives Hilbert property},  does $(\cX,\cM)$ satisfy the $\cM$-Hilbert property over $B$? 
\end{question}
In the setting of split toric varieties, it seems likely that the results for the $\cM$-Hilbert property should extend, so we pose the following conjecture.
\begin{conjecture}
    Let $(K,C)$ be a PF field and let $(X,M)$ be a toric pair. If $N_M=N$, then the toric integral model $(\cX,\cM)$ satisfies the $\cM$-Hilbert property over any open $B\subsetneq C$. If furthermore $N_M^+=N$, then $(\cX,\cM)$ satisfies the $\cM$-Hilbert property over $C$.
\end{conjecture}

\section{Strong approximation and \texorpdfstring{$M$}{M}-approximation for Campana points} \label{section: consequences}
In this section we will consider special cases of Theorem \ref{theorem: M-approx} and its implications for integral points, Campana points and Darmon points on split toric varieties.

If a toric pair $(X,M)$ encodes the integrality condition of an open $V\subset X$, then the lattice $N_M$ is related to the fundamental group of $V$.
\begin{proposition} \label{prop: N_M fundamental group integral}
    Let $K$ be a PF field of characteristic $0$, let $X$ be a complete normal split toric variety over $K$ and let $V\subseteq X$ be an open toric subvariety.
 
    If $(X,M)$ is the toric pair corresponding to $V\subseteq X$ as in the first example \ref{pair integral points} of Section \ref{subsection: M-points examples}, then there exists an isomorphism of profinite groups $$\pi_1(V_{\overline{K}})\cong \widehat{N/N_M},$$
    where $\pi_1(V_{\overline{K}})$ is the étale fundamental group of $V_{\overline{K}}$ and $\widehat{N/N_M}$ is the profinite completion of $N/N_M$.

    Furthermore, the only global sections of $V_{\overline{K}}$ are constant if and only if the cone generated by $N_M^+$ is $N_{\mathbb{R}}$.
\end{proposition}
\begin{proof}
	By \cite[Tag 0A49]{Stacks} we have a natural isomorphism $\pi_1(V_{\overline{K}})\cong \pi_1(V_{\mathbb{C}})$ of étale fundamental groups, where $V_{\mathbb{C}}$ is the toric variety over $\mathbb{C}$ with the same fan as $U$. Now the isomorphism follows directly from \cite[Theorem 12.1.10]{CLS11}.

    Note that the cone generated by $N_M^+$ is exactly the support $|\Sigma_V|$ of the fan $\Sigma_V$ defining $V$, as defined in \cite[Definition 3.1.2]{CLS11}. By \cite[Exercise 4.3.4]{CLS11}, which generalizes to arbitrary fields, $|\Sigma|=N_{\mathbb{R}}$ is equivalent to $\cO(V_{\overline{K}})=\overline{K}$.
\end{proof}
\begin{proof}[Proof of Corollary \ref{corollary: strong approximation 1}]
Combine Proposition \ref{prop: N_M fundamental group integral} with Theorem \ref{theorem: M-approx}.
\end{proof}
The description of the fundamental group given in Proposition \ref{prop: N_M fundamental group integral} does not hold if $K$ has positive characteristic, since then even the affine line $\mathbb{A}_{\overline{K}}^1$ has an infinite fundamental group \cite[Theorem 1.1]{Kum14}. Nevertheless, a similar weaker result is still true if one considers only covers of degree coprime to $\CHAR(K)$, see Remark \ref{remark: M-approximation for Darmon points pos char}. The following corollary gives another characterisation of strong approximation, which is also valid in positive characteristic.
\begin{corollary} \label{corollary: strong approx and Pic}
	Let $(K,C)$ be a PF field, let $V$ be a smooth split toric variety, and let $T\subset \Omega_K$ be a nonempty finite set of places.
	\begin{enumerate}
	\item The variety $V$ satisfies strong approximation off $T$ if $\cO(X_{\overline{K}})^\times=\overline{K}^\times$ and $|\Pic(V)_{\textnormal{tors}}|\in \rho(K,C)$, where $\Pic(V)_{\textnormal{tors}}$ is the torsion subgroup of $\Pic(V)$. The converse also holds if $\Pic(C)$ is finitely generated.
	\item The variety $V$ satisfies strong approximation if and only if $\cO(V)=K$ and $\Pic(V)$ is torsion-free.
	\end{enumerate}

\end{corollary}
\begin{proof}
    Choose a smooth toric compactification $V\subseteq X$.
    The first claim is a direct consequence of Proposition \ref{prop: torsion-free Picard} and Theorem \ref{theorem: M-approx}, while the second one follows from combining these results with \cite[Exercise 4.3.4]{CLS11} (which holds over general fields). The last claim follows from \cite[Proposition 4.2.5]{CLS11} (where we note that the result is independent of the field.)
\end{proof}
\begin{remark}
     By \cite[Proposition 4.2.5]{CLS11}, which generalizes to arbitrary fields, any toric variety whose fan contains a cone of maximal dimension has a torsion-free Picard group.
     By using a resolution of singularities, Corollary \ref{corollary: strong approx and Pic} implies that a normal affine toric variety $V$ satisfies strong approximation off a nonempty set of places $T$ if and only if it does not have torus factors.
\end{remark}

Over number fields, we have yet another characterisation of strong approximation for toric varieties.
\begin{corollary} \label{corollary: strong approximation and Brauer Manin}
	Let $K$ be a number field, let $V$ be a smooth split toric variety and let $T\subset \Omega_K$ be a nonempty finite set of places. Then the following are equivalent:
	\begin{enumerate}
		\item The variety $V$ satisfies strong approximation off $T$.
		\item $\Br(V)/\Br_0(V)= 0$.
		\item $\Br_1(V)/\Br_0(V)= 0$.
	\end{enumerate}
	Here, $\Br(V)$ is the Brauer group of $V$, $\Br_1(V)=\ker(\Br(V)\rightarrow \Br(V_{\overline{K}}))$ is the algebraic Brauer group, and $\Br_0(V)=\im(\Br(K)\rightarrow \Br(V))$ consists of the constant elements in $\Br(V)$.
	If $V$ satisfies any of the above conditions (1)-(3), then $V$ satisfies strong approximation if and only if $\cO(V)=K$.
\end{corollary}
\begin{proof}
	We choose a smooth toric compactification $V\subseteq X$. If $(X,M)$ is the pair corresponding to integral points on $V$, then by Theorem \ref{theorem: M-approx}, $V$ satisfies strong approximation if and only if $N=N_M$. By \cite[Corollary 1.3]{DeFo93}, $N=N_M$ implies that the transcendental Brauer group $\Br(V_{\overline{K}})\cong \Br(V)/\Br_1(V)$ is trivial.
	
	If $V$ has a torus factor, then $V=V'\times \mathbb{G}_m$ for some split toric variety $V'$. So since $\Br_1(\mathbb{G}_m^1)/\Br_0(\mathbb{G}_m^1)$ is nontrivial, it follows that $\Br_1(V)/\Br_0(V)\not\cong 0$.
	If $V$ does not have torus factors, then \cite[Proposition 5.4.2, Remark 5.4.3(3)]{CoSk21} implies that $\Br_1(V)/\Br_0(V)\cong H^1(K,\Pic V_{\overline{K}})$. The Galois cohomology group $H^1(K,\Pic V_{\overline{K}})$ is trivial if and only if $\Pic V_{\overline{K}}\cong \Pic V$ is torsion-free, since $\gal(\overline{K}/K)$ acts trivially on $\Pic V$. Now the equivalence of the statements follows from Corollary \ref{corollary: strong approx and Pic}.
\end{proof}
Using the above criteria for strong approximation, we can characterize when $M$-approximation is satisfied for Campana points.
Now we will use Theorem \ref{theorem: M-approx} to prove Corollary \ref{corollary: M-approx for Campana points}, characterising $M$-approximation for Campana points.
\begin{proof}[Proof of Corollary \ref{corollary: M-approx for Campana points}]
	Let $(X,M')$ be the pair associated with the integral points on $V$. Then $M\subseteq M'$ by definition. Thus if $X$ satisfies $M$-approximation off $T$, then $V$ satisfies strong approximation off $T$ by Proposition \ref{prop: inclusion M approx}. Conversely assume that $V$ satisfies strong approximation off $T$.
	
    Now we set $m= \max_{m_i<\infty}(m_i)$.
	If $X$ is smooth then $m_in_{\rho_i},(m_i+1) n_{\rho_i}\in N_M$ for all $i\in\{1,\dots,n\}$ with $m_i<\infty$, so $n_{\rho_i}\in N_M$ and thus $N_M=N_{M'}$, so $M$-approximation holds if $T\neq \emptyset$ by Theorem \ref{theorem: M-approx}. If furthermore $T=\emptyset$, then $N_{M'}^+=N$ by Corollary \ref{corollary: strong approx and Pic}. Thus every element $c\in N$ can be written as $c=\sum_{i=1}^n c_in_{\rho_i}$ for some $c_i\geq 0$ with $c_i=0$ if $m_i=\infty$, and therefore we have $m'c\in N_M^+$ for $m'\geq m$. Thus $c=(m+1)c+m(-c)\in N_M^+$, showing that $(X,M)$ satisfies $M$-approximation off $T$ if $X$ is smooth.
	
    If $X$ is singular, then we consider a toric resolution of singularities $f\colon\tilde{X}\rightarrow X$. For every torus invariant prime divisor $\tilde{D}$ with $f(\tilde{D})\not\subseteq X\setminus V$, $\tilde{D}\not\subseteq f^{-1}D_i$ for any $i\in \{1,\dots,n\}$ with $m_i=\infty$. Denote the Zariski closure of $\tilde{D}$ in the toric integral model $\tilde{\cX}$ of $\tilde{X}$ by $\tilde{\cD}$. For each divisor $D_i$ on $X$ with $\tilde{D}\subseteq f^{-1}D_i$ as schemes, we have
    $$n_v(\tilde{\cD},P)\leq n_v(f^{-1}\cD_i,P)= n_v(\cD_i,f(P))$$
    for all $v\in \Omega_K^{<\infty}$ and $P\in \tilde{X}(K_v)$, where the equality is due to Proposition \ref{prop: inverse pair}. This gives the inclusion $(\tilde{X},\tilde{M})\subseteq (\tilde{X}, f^{-1}M)$, where $\tilde{M}$ is the Campana condition for the divisor
    $$\tilde{D}_{\tilde{\mathbf{m}}}=\sum_{\substack{i=1 \\ f(\tilde{D}_i)\not\subseteq X\setminus V}}^{\tilde{n}} \left(1-\frac{1}{m}\right)\tilde{D}_i+\sum_{\substack{i=1 \\ f(\tilde{D}_i)\subseteq X\setminus V}}^{\tilde{n}}\tilde{D}_i,$$
    where $\tilde{D}_1,\dots,\tilde{D}_{\tilde{m}}$ are the torus invariant prime divisors on $\tilde{X}$.
    Since $V$ satisfies strong approximation off $T$, $V\times_X \tilde{X}\subseteq \tilde{X}$ also satisfies strong approximation off $T$ by Corollary \ref{corollary: M approx toric stable under birational transformations}. Since $\tilde{X}$ is smooth, the first part of the proof implies $(\tilde{X},\tilde{M})$ satisfies $\tilde{M}$-approximation off $T$ and thus $(\tilde{X}, f^{-1}M)$ satisfies $f^{-1}M$-approximation. Now Corollary \ref{corollary: M approx toric stable under birational transformations} implies $(X,M)$ satisfies $M$-approximation.
\end{proof}

\section{Darmon points and root stacks} \label{section: root stacks}
For smooth split toric varieties, we can generalize the connection between the fundamental group and strong approximation to $M$-approximation for Darmon points, using root stacks.

\begin{definition} \label{def: root stack}
    Let $X$ be a scheme and let
    $$D_{\mathbf{m}}=\sum_{i=1}^n \left(1-\frac{1}{m_i}\right)D_i,$$
    for distinct prime Cartier divisors $D_1,\dots,D_n$ on $X$ and integers $m_1,\dots, m_n\in \mathbb{N}^*\cup \{\infty\}$. Then We define \textit{the root stack} associated to $(X,D_{\mathbf{m}})$ to be
    $$(X,\sqrt[\mathbf{m}]{D}):=\bigg(X\setminus \bigcup_{\substack{i=1 \\ m_i=\infty}}^n D_i\bigg)_{\tilde{\mathbf{D}},\tilde{\mathbf{m}}},$$
    where the right hand side is as defined in \cite[Definition 2.2.4]{Cad07}. Here $\tilde{\mathbf{m}}=(m_i)_{i\in I}$ and $\tilde{\mathbf{D}}=(D_i)_{i\in I}$, where $I=\{i\in \{1,\dots, n\}\mid m_i\neq \infty\}$.
\end{definition}
The following proposition illustrates the close relationship between root stacks and Darmon points.
\begin{proposition} \label{prop: relation Darmon points root stack}
    Let $(K,C)$ be a PF field, let $B\subset C$ be a nonempty open subset and let $X$ be a proper variety over $K$ with an integral model $\cX$ over $B$. Let $\cD_{\mathbf{m}}=\sum_{i=1}^n \left(1-\frac{1}{m_i}\right)\cD_i$ for $m_1,\dots,m_n\in \mathbb{N}^*\cup\{\infty\}$ and prime Cartier divisors $\cD_1,\dots, \cD_n$ on $\cX$. Let $v\in B$, let $T\subset \Omega_K$ be a finite set of places, and let $R$ be $\cO_v$, $\mathbf{A}_K^T$, $\mathbf{A}_B^T$ or a field extension $L$ of $K$.
    For every place $v\in B$, the Darmon points over $R$ on $(\cX,\cD_{\mathbf{m}})$ as in Definition \ref{def: Campana points and Darmon points} are exactly the points $\cP\in \cX(R)$ such that there exists a factorisation
\begin{equation} \label{eq: factorisation}
\begin{tikzcd}
\spec R \arrow[r, dotted] \arrow[rd, "\cP"] & (\cX,\sqrt[\mathbf{m}]{\cD}) \arrow[d] \\
                                              & \cX.                             
\end{tikzcd}\end{equation}
Moreover, if the divisors $\cD_i$ pull back to Cartier divisors on $\spec R$, then the above factorisation is unique.

In particular, if $\cP\colon B\rightarrow \cX$ is a morphism such that $\Pic(B)=0$ or $\im\cP\not\subset\bigcup_{i=1}^n\cD_i$, then $\cP$ is a Darmon point on $(\cX,\cD_{\mathbf{m}})$ if and only if $\cP$ factors through the root stack as in \eqref{eq: factorisation}.
\end{proposition}
\begin{proof}
    For $i\in \{1,\dots,n\}$, let $L$
    By the definition of $(\cX,\sqrt[\mathbf{m}]{\cD})$ and \cite[Remark 2.2.2, Remark 2.2.5]{Cad07}, a morphism $\spec R\rightarrow (\cX,\sqrt[\mathbf{m}]{\cD})$ is determined by a morphism $\cP\colon \spec R\rightarrow \cX\setminus\bigcup_{\substack{i=1 \\ m_i=\infty}}^n \cD_i$ together with isomorphism classes $(L_1,s_1),\dots, (L_n,s_n)$ of line bundles on $\spec R$ with a given global section and isomorphisms $\phi_i\colon L_i^{\otimes m_i}\rightarrow \cP^*\cO(\cD_i)$ such that $\phi_i(s_i^{m_i})=\cP^*1_{\cD_i}$, where $1_{\cD_i}$ denotes the canonical section of $\cO(\cD_i)$. Since $\Pic(R)$ is trivial, which for the adelic rings follows from Proposition \ref{prop: Picard group Adeles}, the line bundles $\cP^*\cO(\cD_i)$ are trivial and thus the isomorphism classes just correspond to ideals on $R$. Thus the factorisation exists if and only if for all $i$ with $m_i\neq \infty$ the ideal defining the closed subscheme $\cP\cap\cD\subset \spec R$ is an $m_i$-th power of an ideal (which is uniquely determined by this property).

    If the image of $\cP$ is not contained in $\bigcup_{i=1}^n \cD_i$, then the pullbacks $(\cP^*\cD_1,\cP^*s_1),\dots, (\cP^*\cD_n,\cP^*s_n)$ all exist as effective Cartier divisors and thus correspond to invertible ideals on $R$. For $i\in\{1,\dots,n\}$, any automorphism of $\cP^*\cD_i$ fixing $\cP^*s_i$ corresponds to an automorphism of $R$-modules $R\rightarrow R$ fixing a nonzero divisor, and is therefore trivial. Thus the factorisation is necessarily unique.

    Now consider a morphism $\cP\colon B\rightarrow \cX$. If $\Pic(B)=0$, then by the same reasoning as above, $\cP$ is a Darmon point on $(\cX,\cD_{\mathbf{m}})$ if and only if it factors through the root stack. If instead $\im\cP\not\subset\bigcup_{i=1}^n\cD_i$, then the divisors $\cD_i$ pull back to effective Cartier divisors on $B$, and therefore \cite[Remark 2.2.2]{Cad07} implies that $\cP$ is a Darmon point if and only if it factors though the root stack.
\end{proof}

The next proposition gives conditions for a root stack to be regular.
\begin{proposition} \label{prop: root stack regular}
    Let $X$ be a regular scheme and let $$D_{\mathbf{m}}=\sum_{i=1}^n \left(1-\frac{1}{m_i}\right)D_i,$$
    for distinct prime Cartier divisors $D_1,\dots,D_n$ on $X$ and integers $m_1,\dots, m_n\in \mathbb{N}^*\cup \{\infty\}$ such that the support of $D_{\mathbf{m}}$ is an strict normal crossings divisor. Then $(X,\sqrt[\mathbf{m}]{D})$ is regular.
\end{proposition}
\begin{proof}
Without loss of generality, we can assume that $m_1,\dots, m_n$ are all finite, since a the restriction of a strict normal crossings divisor to an open is still strict normal crossings.    
As the statement is local, we can assume that $X=\spec A$ is affine and $D_i=\spec A/(s_i)$. By \cite[Example 2.4.1]{Cad07} this implies $$(X,\sqrt[\mathbf{m}]{D})\cong [\spec R/ (\mu_{m_1}\times_\Z\dots\times_\Z \mu_{m_n})],$$
where $R=A[t_1,\dots, t_n]/(x_1^{m_1}-s_1,\dots, x_n^{m_n}-s_n)$
and $\mu_{m_i}$ acts trivially on $A$ and $x_j$ for $j\neq i$, and acts on $x_i$ by $t_i\cdot x_i=t_i^{-1}x_i$.
We will first prove that $\spec R$ is regular. Since $R$ is finite over $A$, for every maximal ideal $\mathfrak{m}\in \spec A$ with an extension to a maximal ideal $\mathfrak{m}'\in \spec R$, the dimensions of the local rings agree: $\dim A_\mathfrak{m}=\dim R_{\mathfrak{m}'}$. Furthermore, since the support of $D_{\mathbf{m}}$ is a strict normal crossings divisor, the elements in $\{s_1,\dots,s_n\}\cap \mathfrak{m}$ are part of a regular system of parameters for $\mathfrak{m}$.  and therefore the elements in $\{x_1,\dots,x_n\}\cap \mathfrak{m}'$ is part of a regular system of parameters for $\mathfrak{m}'$, since if $I$ is the ideal generated by the $x_1,\dots, x_n$ contained in $\mathfrak{m}$ and $I'$ is generated by the $s_1,\dots, s_n$ contained in $\mathfrak{m}'$, then $R_{\mathfrak{m}'}/I'\cong A_\mathfrak{m}/I$. Thus $R_{\mathfrak{m}'}$ is a regular local ring, so $\spec R$ is regular.

Note that $\spec R\rightarrow (X,\sqrt[\mathbf{m}]{D})$ is surjective, flat and of finite presentation. Consider a smooth cover $Y\rightarrow (X,\sqrt[\mathbf{m}]{D})$, where $Y$ is a scheme. Then $\spec R\times_{(X,\sqrt[\mathbf{m}]{D})} Y$ is a regular algebraic space since regularity is local in the smooth topology by \cite[Tag 036D]{Stacks}. Since $\spec R\times_{(X,\sqrt[\mathbf{m}]{D})} Y\rightarrow Y$ is surjective, flat and of finite presentation, $Y$ is regular by \cite[Tag 06QN]{Stacks} and thus $(X,\sqrt[\mathbf{m}]{D})$ is regular.
\end{proof}

Proposition \ref{prop: relation Darmon points root stack} allows us to relate $M$-approximation with strong approximation on root stacks, as studied in \cite{Chr20,San23-1}.
\begin{definition} \label{definition: strong approximation stack}
Let $U$ be a stack over a PF field $(K,C)$ and let $T\subset\Omega_K$ be a finite set of places. We say that $U$ satisfies strong approximation off $T$ if the map $$U(K)\rightarrow U(\mathbf{A}_K^T)$$ has dense image, where the topology on $U(\mathbf{A}_K^T)$ is defined as in \cite[Definition 5.0.10]{Chr20}.
\end{definition}
\begin{proposition} \label{prop: M-approximation implies strong approximation stack}
    Let $(K,C)$ be a PF field, $T\subset \Omega_K$ a finite set of places, and $X$ a smooth proper variety over $K$. Let $D_{\mathbf{m}}=\sum_{i=1}^n \left(1-\frac{1}{m_i}\right)D_i$ for $m_1,\dots,m_n\in \mathbb{N}^*\cup\{\infty\}$ and smooth prime Cartier divisors $D_1,\dots, D_n$ on $X$. Assume that the support of $D_{\mathbf{m}}$ is an strict normal crossings divisor. Let $(X,M)$ be the pair corresponding to the Darmon points on $(X,D_{\mathbf{m}})$. Then $(X,M)$ satisfies $M$-approximation off $T$ if and only if the root stack $(X,\sqrt[\mathbf{m}]{D})$ satisfies strong approximation off $T$.
\end{proposition}
\begin{proof}
    Write $\tilde{X}=(X,\sqrt[\mathbf{m}]{D})$ and $\tilde{\cX}=(\cX,\sqrt[\mathbf{m}]{\cD})$, for some integral model $\cX$ over some open subset $B\subset C$ such that $\cD_i$ is the closure of $D_i$ in $\cX$ and it is a prime Cartier divisor. By \cite[Proposition 13.0.2]{Chr20}, 
    $$(X,\sqrt[\mathbf{m}]{D})(\mathbf{A}_K^T)=\prod_{v\in \Omega_K\setminus (B\cup T)}\tilde{X}(K_v) \times\prod_{v\in B\setminus T}\left(\tilde{X}(K_v),\tilde{\cX}(\cO_v)\right)$$
    as topological spaces (the cited proposition is formulated for global fields and with $T=\emptyset$, but extends to this setting).
    By the assumptions on $X$ and on $D_{\mathbf{m}}$, the root stack $(X,\sqrt[\mathbf{m}]{D})$ is geometrically regular, and thus smooth, by Proposition \ref{prop: root stack regular}. This is because the divisors $D_1,\dots,D_n$ are smooth so the support of $D_{\mathbf{m},\overline{K}}$ is a strict normal crossings divisor.
    Therefore by \cite[Proposition 7.0.8]{Chr20} there exists a scheme $\cZ$ over $B$ and a surjective smooth morphism $\pi\colon \cZ\rightarrow \tilde{\cX}$, such that $\pi(\cZ(\cO_v))=\tilde{\cX}(\cO_v)$ for all $v\in B$ and $\pi(\cZ(K_v))=\tilde{X}(K_v)$ for any place $v\in \Omega_K$. In general \cite[Proposition 7.0.8]{Chr20} only gives a family of schemes $\cZ_N$ with this property, but we can take the disjoint union of these to obtain $\cZ$. Since $\tilde{X}$ is smooth over $K$, $\cZ_K$ is smooth over $K$ as well, and thus is locally a variety. Therefore Proposition \ref{prop: Open subset dense} implies that $U(K_v)\subset \cZ(K_v)$ is dense for any dense Zariski open subset $U\subset Z$ and any $v\in \Omega_K$. In particular this implies that $(X,M_{\fin})(K_v)=(X\setminus \bigcup_{i=1}^n D_i)(K_v)$ lies dense in $\tilde{X}(K_v)$. Therefore, it follows that the image of $(X,M_{\fin})(\mathbf{A}_K^T)$ is dense in $\tilde{X}(\mathbf{A}_K^T)$. Thus, by Proposition \ref{prop: inclusion M}, $\tilde{X}(K)$ is dense in $\tilde{X}(\mathbf{A}_K^T)$ if and only if $(X,M)(K)$ is dense in $(X,M)(\mathbf{A}_K^T)$.
\end{proof}

We now consider the fundamental group $\pi_1(X,\sqrt[\mathbf{m}]{D})$ of the root stack as defined in \cite{Nooh04}, which classifies étale covers $f\colon Y\rightarrow (X,\sqrt[\mathbf{m}]{D})$, where $Y$ is an algebraic stack.
\begin{definition}
    A morphism $f\colon Y\rightarrow X$ of connected algebraic stacks is an \textit{étale cover} if it is a finite étale morphism.
\end{definition}
Note that such a morphism is always representable, see for example \cite[Tag 0CHT]{Stacks}, so étale covers coincide with what Noohi calls covering maps.

Similarly to the fundamental group of schemes, different choices of a base point yield the same fundamental group up to isomorphism \cite[page 9]{Nooh04}, so we will leave the choice of the point implicit. The following lemma classifies the étale covers of root stacks in terms of the coarse spaces, and thus gives a concrete description of its étale fundamental group.

\begin{lemma}
    Let $X$ be a connected locally Noetherian scheme and let $D_1,\dots,D_n$ be distinct prime Cartier divisors on $X$, and $m_1,\dots, m_n\in \mathbb{N}^*$ whose images in $\cO(X)$ are invertible. Assume that the support of $$D_{\mathbf{m}}=\sum_{i=1}^n \left(1-\frac{1}{m_i}\right)D_i,$$
    is a strict normal crossings divisor.

    Then there is a one-to-one correspondence between
    \begin{enumerate}
        \item Étale covers $\tilde{f}\colon \tilde{Y}\rightarrow (X,\sqrt[\mathbf{m}]{D})$, and
        \item Finite morphisms of connected schemes $f\colon Y\rightarrow X$, such that $f$ is étale over $X\setminus \bigcup_{i=1}^n D_i$, and for all $i\in\{1,\dots,n\}$ the pullback satisfies $f^*D_i=\sum_{\beta\in\cB_i}e_{i,\beta}\tilde{D}_{i,\beta}$ for distinct prime Cartier divisors $\tilde{D}_{i,\beta}$ on $Y$ such that $e_{i,\beta}|m_i$ for all $\beta\in \cB_i$. \label{morphism on coarse spaces}
    \end{enumerate}
    
\end{lemma}
\begin{proof}
Recall that our schemes are by assumption separated so since $m_1,\dots, m_n$ are invertible in $\cO(X)$, the stack $(X,\sqrt[\mathbf{m}]{D})$ is separated and Deligne-Mumford by \cite[Corollary 2.3.4]{Cad07}. Therefore for any étale cover $\tilde{Y}\rightarrow (X,\sqrt[\mathbf{m}]{D})$, $\tilde{Y}$ is separated as the map is finite and hence affine \cite[Tag 01S7]{Stacks} and Deligne-Mumford since it is étale \cite[Tag 0CIQ]{Stacks}. Therefore by the Keel-Mori Theorem \cite[page 631]{Ryd13}, there exists a coarse moduli space $\tilde{Y}\rightarrow Y$, where $Y$ is an algebraic space. Since the coarse moduli space is universal for maps to algebraic spaces by definition \cite[Definition 6.8]{Ryd13}, the cover $\tilde{f}$ descends to a morphism $f\colon Y\rightarrow X$.
    We thus obtain a commutative diagram
\begin{equation} \label{eq: etale map root stack}
\begin{tikzcd}
\tilde{Y} \arrow[r, "\tilde{f}"] \arrow[d] & {(X,\sqrt[\mathbf{m}]{D})} \arrow[d] \\
Y \arrow[r, "f"]                           & X.                                
\end{tikzcd}
\end{equation}
We will show that $Y$ is a scheme and that $f\colon Y\rightarrow X$ satisfies the desired properties.

Since the statement is local on $X$, we can reduce to the case that $X=\spec A$ is affine and the $D_i=\spec A/(s_i)$. By \cite[Example 2.4.1]{Cad07} this implies $$(X,\sqrt[\mathbf{m}]{D})\cong [\spec R/ (\mu_{m_1}\times_\Z\dots\times_\Z \mu_{m_n})],$$
where $R=A[t_1,\dots, t_n]/(x_1^{m_1}-s_1,\dots, x_n^{m_n}-s_n)$
and $\mu_{m_i}$ acts trivially on $A$ and $x_j$ for $j\neq i$, and acts on $x_i$ by $t_i\cdot x_i=t_i^{-1}x_i$. Write $Z=\spec R$. Note that since $R$ is a finite $A$-algebra, the map $Z\rightarrow X$ is finite.

We now prove that $f\colon Y\rightarrow X$ is a finite morphism. Let $U\rightarrow Y$ be a finite étale morphism from a scheme $U$ and consider the following commutative diagram consisting of Cartesian squares:

$$\begin{tikzcd}
U\times_X Z \arrow[r] \arrow[d]           & \tilde{Y}\times_X Z \arrow[r] \arrow[d]    & Z \arrow[d]                  \\
U\times_{Y} \tilde{Y} \arrow[r] \arrow[d] & \tilde{Y} \arrow[r, "\tilde{f}"] \arrow[d] & {(X,\sqrt[\mathbf{m}]{D})} \arrow[d] \\
U \arrow[r]                               & Y \arrow[r, "f"]                           & X.                                   
\end{tikzcd}$$
\newline
Since $U\times_X Z\rightarrow Z$ and $Z\rightarrow X$ are finite morphisms of schemes, $U\times_X Z\rightarrow X$ is a finite morphism as well.

Since $X$ is locally Noetherian, \cite[Theorem 6.12]{Ryd13} implies that the map $f$ is proper and thus locally of finite type. Therefore $U\rightarrow X$ is also locally of finite type, and since $U\times_X Z\rightarrow U$ is surjective, \cite[Tag 0GWS]{Stacks} implies that $U\rightarrow X$ is quasi-finite, and by \cite[Tag 02LS]{Stacks} it is finite and hence $f$ is finite as well. By \cite[Tag 03XX]{Stacks} $Y$ is a scheme.

Since $(X,\sqrt[\mathbf{m}]{D})\rightarrow X$ is étale outside the support of $D_{\mathbf{m}}$, $f$ is étale outside its support as well. Furthermore, since $\tilde{f}$ is étale and the pullback of $D_i$ along the morphism $(X,\sqrt[\mathbf{m}]{D})\rightarrow X$ is $m_iE_i$ for a prime divisor $E_i$, the multiplicity of every prime divisor appearing in the pullback of $D_i$ to $\tilde{Y}$ is exactly $m_i$. By the commutativity of the diagram \eqref{eq: etale map root stack}, the multiplicities $e_{i,\beta}$ divide $m_i$.

Conversely, for any such morphism $f\colon Y\rightarrow X$ satisfying the properties in \eqref{morphism on coarse spaces}, we will show we can construct a root stack $\tilde{Y}$ over $Y$ and an étale cover $\tilde{f}\colon \tilde{Y}\rightarrow (X,\sqrt[\mathbf{m}]{D})$ inducing the morphism $f$ on coarse spaces. Since the support $\operatorname{sup}(D_{\mathbf{m}})$ is strict normal crossings, Proposition \ref{prop: root stack regular} implies all points $\tilde{x}\in (X,\sqrt[\mathbf{m}]{D})\times_X \operatorname{sup}(D_{\mathbf{m}})$ are regular. Since the statement to be proved is local on $X$, and the regular locus is open, we can assume that $(X,\sqrt[\mathbf{m}]{D})$ is regular. Take $\tilde{Y}$ to be the root stack $\tilde{Y}=(Y,\sqrt[\tilde{\mathbf{m}}]{\tilde{D}})$, where $\tilde{D}_{\tilde{\mathbf{m}}}=\sum_{i=1}^n \sum_{\beta\in \cB_i}\left(1-\frac{e_{i,\beta}}{m_i}\tilde{D}_{i,\beta}\right)$. Since $f$ is a finite map of schemes, $\tilde{Y}\times_X Z\rightarrow Z$ is finite, so $\tilde{Y}\times_X Z$ is a scheme by \cite[Tag 03XX]{Stacks}. The morphism $\tilde{Y}\times_X Z\rightarrow Z$ is étale at the codimension $1$ points appearing in the pullback of $D_i$ to $\tilde{Y}\times_X Z$, for every $i\in\{1,\dots, n\}$. Therefore purity of the ramification locus \cite[Tag 0EA4]{Stacks} implies that $\tilde{Y}\times_X Z\rightarrow Z$ is étale at every point above $\bigcup_{i=1}^n D_i$, and thus $\tilde{Y}\times_X Z\rightarrow Z$ is étale everywhere. Since $Z\rightarrow (X,\sqrt[\mathbf{m}]{D})$ is an étale cover, this implies $\tilde{f}\colon \tilde{Y}\rightarrow (X,\sqrt[\mathbf{m}]{D})$ is an étale cover.
\end{proof}

Now we can explicitly compute the fundamental group of a toric root stack. For an abelian group $G$, let $\widehat{G}=\lim_{\leftarrow \\ H} G/H$ be the profinite completion, where $H$ runs over all normal subgroups of $G$ with finite index in $G$. Similarly, for a prime number $p$, we let $\widehat{G}^{(p)}=\lim_{\substack{\leftarrow \\ H,\, p\nmid [G:N]}} G/H$ be the prime-to-$p$ completion. For a stack $Z$, we write $\pi_{1, p\nmid}(Z)$ for the quotient of $\pi_1(Z)$ corresponding to covers with degree coprime to $p$.
\begin{lemma} \label{lemma: fundamental group toric stack}
    Let $X$ be a smooth toric variety over an algebraically closed field $\overline{K}$ and let $D_1,\dots,D_n$ be the the torus invariant prime divisors on $X$ and let $m_1,\dots, m_n\in \mathbb{N}^*\cup\{\infty\}$ whose images in $\cO(X)$ are invertible and let $D_{\mathbf{m}}$ be the corresponding Campana divisor. Let $(X,M)$ be the pair corresponding to the Darmon points on $(X,D_{\mathbf{m}})$. If $\CHAR(\overline{K})=0$, then $$\pi_1(X,\sqrt[\mathbf{m}]{D})\cong \widehat{N/N_M},$$
    while if $\overline{K}$ has characteristic $p>0$ then
    $$\pi_{1,\,p\nmid}(X,\sqrt[\mathbf{m}]{D})\cong \widehat{N/N_M}^{(p)}.$$    
\end{lemma}
\begin{proof}
    We only write the proof for characteristic $0$, as the positive characteristic case is analogous. We write $\Sigma$ for the fan of $X\setminus \lfloor D_{\mathbf{m}} \rfloor$. If $X\setminus \lfloor D_{\mathbf{m}} \rfloor$ is just the dense open torus (so if $m_1=\dots=m_n=\infty$) then the statement is true by \cite[Propostion 1.1]{BrSz13}. By this we see that a Galois cover of $\mathbb{G}_{m}^d$ is necessarily just a pair of $d$ coverings of $\mathbb{G}_{m}$. Therefore, we can without loss of generality assume that $X\setminus \lfloor D_{\mathbf{m}} \rfloor$ does not have torus factors.
	By \cite[Theorem 4]{AlPa12} any finite morphism $f\colon Y\rightarrow X\setminus \lfloor D_{\mathbf{m}} \rfloor$ of connected schemes, such that the degree is coprime to $p$, is a morphism of toric varieties. The cited theorem is only stated for complete toric varieties, but the proof also works for the toric varieties without torus factors. Furthermore \cite[Definition 1, Lemma 1]{AlPa12} imply that the étale covers $f\colon Y\rightarrow (X,\sqrt[\mathbf{m}]{D})$ exactly correspond to maps of fans $(N',\Sigma)\rightarrow (N,\Sigma)$, where $N'\subseteq N$ is a sublattice of finite index containing $N_M$. Furthermore, such a cover is a $N/N'$-cover. By letting $N'$ run over all lattices $N'\subseteq N$ of finite index which contain $N_M$ we obtain the result.
\end{proof}

\begin{proof}[Proof of Corollary \ref{corollary: M-approximation for Darmon points}]
	Combine Theorem \ref{theorem: M-approx} with Lemma \ref{lemma: fundamental group toric stack}. For the second part, note that the coarse space of $(X,\sqrt[\mathbf{m}]{D})$ is $X\setminus \lfloor D_{\mathbf{m}}\rfloor$ and use Proposition \ref{prop: N_M fundamental group integral}.
\end{proof}
\begin{remark} \label{remark: M-approximation for Darmon points pos char}
    Corollary \ref{corollary: M-approximation for Darmon points} extends to characteristic $p$ to give a necessary condition for strong approximation on a toric root stack with invertible multiplicities. If $(X,\sqrt[\mathbf{m}]{D})$ satisfies strong approximation off $T$, then $|\pi_{1, p\nmid}(X_{\overline{K}},\sqrt[\mathbf{m}]{D_{\overline{K}}})|\in\rho(K,C)$ and similarly strong approximation implies that $(X_{\overline{K}},\sqrt[\mathbf{m}]{D_{\overline{K}}})$ does not have étale covers of degree coprime to $p$. However, this is not a sufficient condition, since the quotient group $N/N_M$ corresponding to Darmon points may have $p$-torsion.
\end{remark}

\begin{example}
    If $\CHAR(K)=0$, $X=\P^1$ and $D_\mathbf{m}=\frac{1}{2}(0)+\frac{1}{2}(\infty)$ then the morphism $\P_K^1\rightarrow \P_K^1$ given by $(x_0:x_1)\mapsto (x_0^2:x_1^2)$ factors as the étale morphism $\P_K^1\rightarrow (\P_K^1,\sqrt[\mathbf{m}]{D})$ followed by the coarse moduli space morphism $(\P^1,\sqrt[\mathbf{m}]{D})\rightarrow \P^1$. The étale morphism corresponds to the nontrivial element in $\pi_1(\P^1_{\overline{K}},\sqrt[\mathbf{m}]{D_{\overline{K}}})\cong \Z/2\Z$.
\end{example}
More generally, if $(\P^n,M)$ is a toric pair corresponding to Darmon points, then the group $N/N_M$ is nontrivial as soon as two multiplicities are not coprime. Therefore we obtain the following consequence of Corollary \ref{corollary: M-approximation for Darmon points}.
\begin{corollary}
	Let $(K,C)$ be a PF field and let $T\subset \Omega_{K}$ be a finite nonempty set of places. Let $D_{\mathbf{m}}$ be the $\mathbb{Q}$-divisor on $\P_K^{n-1}$ given by $$D_{\mathbf{m}}=\sum_{i=0}^{n-1} \left(1-\frac{1}{m_i}\right)D_i,\quad D_i=\{x_i=0\}.$$ Then $(\P_K^{n-1},\sqrt[\mathbf{m}]{D})$ satisfies strong approximation off $T$ if $\gcd(m_i, m_j)\in \rho(K,C)$ for every $i\neq j$. The converse also holds if $\Pic(C)$ is finitely generated. Furthermore, $(\P_K^{n-1},\sqrt[\mathbf{m}]{D})$ satisfies strong approximation if and only if $m_i<\infty$ for all $i\in\{0,\dots, n\}$ and $\gcd(m_i, m_j)=1$ for every $i\neq j$. 
\end{corollary}
\begin{proof}
    Let $(\P_K^{n-1},M)$ be the pair corresponding to the Darmon points on $(\P_K^{n-1},D_{\mathbf{m}})$.
    Consider the matrix 
	$$\left(\begin{matrix}
		-m_0 & m_1 & 0 & \dots & 0 \\
		-m_0& 0   & m_2 & \dots & 0 \\
		\vdots &\vdots & \vdots & \ddots & \vdots \\
		-m_0 & 0 & 0 & \dots & m_{n-1} \\
	\end{matrix}\right),$$
    where the columns correspond to generators of $N_M$ (if $m_i=\infty$, we make the corresponding column $0$ instead).
    By Theorem \ref{theorem: M-approx}, $M$-approximation off $T$ is satisfied if and only if the matrix has full rank and induces surjective homomorphisms $(\Z/p\Z)^n\rightarrow (\Z/p\Z)^{n-1}$ for every prime number $p\not\in \rho(K,C)$. The matrix having full rank is equivalent to $m_i=
    \infty$ for at most one $i\in \{0,\dots, n-1\}$.
    The surjectivity at a prime number $p\not\in \rho(K,C)$ is equivalent to some $n\times n$ minor of the matrix not being divisible by $p$. As the maximal minors are, up to sign, of the form $\prod_{\substack{i=0\\ i\neq j}}^n m_i$ for some $0\leq j \leq n$, this is equivalent to $\gcd(\prod_{\substack{i=0\\ i\neq 0}}^n m_i,\dots, \prod_{\substack{i=0\\ i\neq n}}^n m_i)\in \rho(K,C)$. This is in turn equivalent to $\gcd(m_i, m_j)\in \rho(K,C)$ for every $i\neq j$. The proof for $T=\emptyset$ is analogous.
\end{proof}

The formula obtained above generalizes to a simple sufficient criterion for strong approximation on toric stacks.
\begin{corollary} \label{corollary: sufficient condition M-approx Darmon}
	Let $K$ be a PF field and let $T\subset \Omega_{K}$ be a nonempty set of places. Let $X$ be a smooth split toric variety and let $(X,\sqrt[\mathbf{m}]{D})$ be the root stack corresponding to $D_{\mathbf{m}}=\sum_{i=0}^n \left(1-\frac{1}{m_i}\right)D_i$, and let $P_{\sigma}(x_1,\dots, x_n)=\prod_{\substack{i=1 \\ \rho_i\subseteq \sigma}}^n x_i$ for a maximal cone $\sigma$.
	If $\gcd_{\sigma\in \Sigma_{\text{max}}}(P_{\sigma}(m_1,\dots, m_n))\in \rho(K,C)$, then $(X,\sqrt[\mathbf{m}]{D})$ satisfies strong approximation off $T$. Here $\Sigma_{\text{max}}$ is the set of maximal cones in the fan of $X$.
\end{corollary}
\begin{proof}
If $\cV_\sigma\cong \mathbb{A}^n$ is the affine open in the toric integral model $\cX$ corresponding to the maximal cone $\sigma$, then for any $v\in \Omega_K^{<\infty}$, the lattice $N_{\sigma}$ spanned by the image of $(\cX,\cM)(\cO_v)\cap \cV_\sigma(\cO_v)$ under $\phi_v$ in $N_M$ is generated by $m_in_{\rho_i}$ for $\rho_i\subseteq \sigma$. Thus $|N:N_\sigma|=P_{\sigma}(m_1,\dots, m_n)$. Since $N_\sigma\subset N_M$, this implies $|N:N_M|$ divides $P_{\sigma}(m_1,\dots, m_n)$, giving the result.
\end{proof}

While Corollary \ref{corollary: sufficient condition M-approx Darmon} is a sufficient criterion, it is not a necessary condition on many toric varieties, such as Hirzebruch surfaces.
\begin{example}[Hirzebruch surfaces] \label{example: Hirzebruch surfaces}
	Let $(K,C)$ be a PF field such that $\Pic(C)$ is finitely generated, $T\subset \Omega_{K}$ a nonempty set of places and $r\geq 0$ an integer.
	Consider the Hirzebruch surface $H_r$ given by the fan with ray generators $n_{\rho_1}=(-1,r), n_{\rho_2}=(0,1), n_{\rho_3}=(1,0), n_{\rho_4}=(0,-1)$ and choose corresponding multiplicities $m_1, m_2, m_3, m_4$. If $(H_r,M)$ is the pair for the Darmon points with these multiplicities, then by looking at the generators modulo a prime number we see that the prime numbers dividing $|N:N_M|$ are the prime numbers dividing
    $$\gcd(m_1 m_2, m_1 m_4, m_2 m_3, m_3 m_4, rm_1 m_3).$$
  By Theorem \ref{theorem: M-approx}, the pair $(H_r,M)$ satisfies $M$-approximation off $T$ if and only if $|N:N_M|\in \rho(K,C)$. Moreover, $(H_r,M)$ satisfies $M$-approximation if and only if $N=N_M$, $m_1,m_3,m_4<\infty$ and $(r,m_2)\neq (0,\infty)$. Note the additional $rm_1 m_3$ factor showing up in the index of $|N:N_M|$ compared to the criterion given in Corollary \ref{corollary: sufficient condition M-approx Darmon}, showing that $M$-approximation can hold even if the condition in the corollary is violated.
\end{example} 

Finally we also consider $M$-approximation for pairs corresponding to Darmon points on a singular variety.
\begin{example}
Let $(K,C)$ be a PF field such that $\Pic(C)$ is finitely generated, $T\subset \Omega_{K}$ a nonempty set of places and $r\geq 1$ an integer.
Consider the weighted projective plane $\P_K(1,1,r)$ for $r\geq 1$ with rays generated by $n_{\rho_0}=(-1,r),n_{\rho_1}=(1,0),n_{\rho_2}=(0,-1)$ and choose corresponding multiplicities $m_0,m_1,m_2$ with $m_0,m_1<\infty$. Then the Darmon points on $\P_K(1,1,r)$ satisfy $M$-approximation off $T$ if and only if $\gcd(m_0,m_1),\gcd(m_0m_1, m_2, r-1)\in \rho(K,C)$. Furthermore the Darmon points on $\P_K(1,1,r)$ satisfy $M$-approximation if and only if $\gcd(m_0,m_1)=\gcd(m_0m_1, m_2, r-1)=1$ and $m_2<\infty$.
In particular, if $r=2$, then whether $M$-approximation is satisfied (off $T$ or off $\emptyset$) does not depend on the value of $m_2$.
\end{example}
\appendix
\section{\texorpdfstring{$(M,\cM')$}{(M,M')}-approximation} \label{appendix}
In this appendix we generalize the notion of adelic $M$-points and integral adelic $\cM$-points introduced in Definition \ref{def: adelic M-points} in order to relate these notions to the adelic points considered in \cite{MiNaSt22}.
\begin{definition} \label{def: (M,M')-points}
    Let $(K,C)$ be PF field, $T\subset \Omega_K$ be a finite sets of places and let $B\subset C$ be an open subscheme. Let $(X,M)\subset(X,M')$ be an inclusion of pairs with integral models $(\cX,\cM)\subseteq(\cX,\cM')$. We define the space of \textit{adelic $(\cM, \cM')$-points over $B$ prime to $T$} to be the restricted product $$(\cX,\cM,\cM')(\mathbf{A}_B^T)=\prod_{v \in \Omega_K\setminus T} ((\cX, \cM')(\cO_v)\cap (X,M')(K_v), (\cX, \cM)(\cO_v)).$$
\end{definition}
If $(\cX,\cM)=(\cX,\cM')$, this recovers the notion of integral adelic $\cM$-points, while if $(\cX,\cM')(\cO_v)=(X,M)(K_v)$ for all $v\in \Omega_K$ then this recovers the notion of adelic $M$-points.
\begin{remark} \label{remark: comparison with MNS}
The definition of $(\cM,\cM')$-points generalizes the notion of adelic semi-integral points considered in \cite{MiNaSt22}:
if $\cM'$ encodes the Campana (respectively Darmon) condition for a divisor $\cD_{\mathbf{m}}$ as in Definition \ref{def: Campana points and Darmon points} and $\cM$ is the integrality condition for the support of $\cD_{\mathbf{m}}$, then there is an equality of topological spaces
$$(\cX,\cM,\cM'_{\fin})(\mathbf{A}_B^T)= (\cX,\cD_{\mathbf{m}})^*_{\text{st}}(\mathbf{A}_{B}^T),$$
where the right hand side is the set of strict $T$-adelic semi-integral points as in \cite[Definition 2.15]{MiNaSt22}. However, the full set of $T$-adelic semi-integral points $(\cX,\cD_{\mathbf{m}})^*(\mathbf{A}_{B}^T)$ differs as a set from $(\cX,\cM,\cM')(\mathbf{A}_B^T)$, since a point $(P_v)_v\in (\cX,\cM,\cM')(\mathbf{A}_B^T)$ is integral with respect to the support of $\cD_{\mathbf{m}}$ at all but finitely many places $v$, while a non-strict point in $(\cX,\cD_{\mathbf{m}})^*(\mathbf{A}_{B}^T)$ lies in the boundary for all places.
\end{remark}

In the remainder of the section we will generalize some results of Section \ref{section: Adelic M-points and approx}.

The following proposition is an analogue of Proposition \ref{prop: independence model}, and shows that the space of adelic $(\cM, \cM')$-points does not depend on the choice of an integral model for $(X,M)$.
\begin{proposition} \label{prop: independence model (M,M')}
	If $(X,M)$ is a pair with integral models $(\cX,\cM)$ and $(\cX,\cM')$ over $B$, and $(\cX,\cM'')$ is an integral model over $B$ of the pair $(X,M)$ such that $(\cX,\cM),(\cX,\cM')\subseteq (\cX,\cM'')$ are open, then there is a canonical homeomorphism
	$$(\cX,\cM,\cM'')(\mathbf{A}_B^T)\cong (\cX,\cM',\cM'')(\mathbf{A}_B^T).$$
\end{proposition}
\begin{proof}
    This follows from the proof of Proposition \ref{prop: independence model} by replacing the term $(X,M)(K_v)$ in the proof with $(\cX,\cM')(\cO_v)$.
\end{proof}
In light of the previous proposition, and the fact that every pair over $K$ has an integral model over $C$, we will not specify the integral model $(\cX,\cM)$ of $(X,M)$ and write $(M,\cM')$ instead of $(\cM,\cM')$.

The next proposition is an analogue of Proposition \ref{prop: inclusion M approx} and gives conditions for when the natural inclusion maps are embeddings or have dense image.
\begin{proposition} \label{prop: inclusion (M,M')}
	Let $(X,M)\subseteq (X,M')\subseteq (X,M'')$ be pairs and let $(\cX, \cM)\subseteq(\cX, \cM')\subseteq(\cX, \cM'')$ be inclusions of integral models over $B$ of the respective pairs, and let $T\subset \Omega_K$ be a finite set of places. Then:
	\begin{enumerate}
		\item The natural inclusion $(\cX,M,\cM'')(\mathbf{A}_B^T)\hookrightarrow(\cX,M',\cM'')(\mathbf{A}_B^T)$ has dense image if $(\cX,M,\cM'')(\mathbf{A}_B^T)\neq \emptyset$.
		\item The natural inclusion $(\cX,M,\cM')(\mathbf{A}_B^T)\hookrightarrow (X,M,\cM'')(\mathbf{A}_B^T)$ is a topological embedding, and it is open if $(\cX,\cM')(\cO_v)\subset (\cX,\cM'')(\cO_v)$ is open for all $v\in \Omega_K\setminus T$. 
	\end{enumerate}
\end{proposition}
\begin{proof}
    This is a direct consequence of Proposition \ref{prop: inclusions restricted product}.
\end{proof}
We now generalize the notion $M$-approximation to $(M,\cM')$-approximation.

\begin{definition} \label{Def: (M,M')-approx}
	Let $T\subset \Omega_K$ be a finite set of places, let $(X,M), (X,M')$ be pairs over $(K,C)$ such that $M\subseteq M'$ and let $(\cX,\cM')$ be an integral model of $(X,M')$ over $B\subset C$. Then we say that $\cX$ satisfies \textit{$(M,\cM')$-approximation off $T$} if the image of the natural inclusion
	$$(\cX,\cM')(B)\cap (X,M)(K)\hookrightarrow (\cX,M,\cM')(\mathbf{A}_B^T)$$
	is dense.
\end{definition}
The restriction to $(X,M)(K)$ ensures that the inclusion map is well defined.

\begin{example}
If $(\cX,\cM')$ is the pair corresponding to the Campana (or Darmon) condition for the Campana pair $(\cX,\cD_{\mathbf{m}})$ and $M\subseteq M'$ is the integral condition with respect to the support of $\cD_{\mathbf{m}}$, then $(M,\cM')$-approximation on $X$ coincides with strong Campana (or Darmon) approximation on $(\cX, \cD_{\mathbf{m}})$ as studied in \cite{MiNaSt22}, up to the slightly differing adelic space as discussed in Remark \ref{remark: comparison with MNS}. 
\end{example}

\nocite{*}
\printbibliography
\end{document}